\documentclass[11pt,letterpaper]{article}
\usepackage[margin=1in]{geometry}
\usepackage[utf8]{inputenc} 
\usepackage[T1]{fontenc}    
\usepackage{dsfont}
\usepackage[maxfloats=100]{morefloats}[2015/07/22]
\usepackage{multirow}
\usepackage{wrapfig}
\usepackage[T1]{fontenc}
\usepackage{lmodern}

\usepackage{booktabs}       
\usepackage{amsfonts}       
\usepackage{nicefrac}       
\usepackage{microtype}      
\usepackage{enumerate}
\usepackage{lipsum}
\usepackage{mathtools}
\usepackage{cuted}
\usepackage{float}
\usepackage[dvipsnames,table,xcdraw]{xcolor}
\usepackage{bbm}
\usepackage{varioref}
\usepackage{hyperref}
                                
\usepackage{amssymb} 
\usepackage{rotating}
\usepackage{graphicx}
\usepackage{subcaption}
\usepackage[normalem]{ulem}

\usepackage{amsthm}
\DeclareUnicodeCharacter{00A0}{~}
\usepackage{algorithm}
\usepackage{algorithmic}

\newtheorem{assumption}{Assumption}
\newtheorem{theorem}{Theorem}
\newtheorem{lemma}{Lemma}
\newtheorem{proposition}{Proposition}
\newtheorem{definition}{Definition}
\newtheorem{corollary}{Corollary}
\theoremstyle{plain}
\newtheorem{remark}{Remark}
\RequirePackage[capitalize,nameinlink]{cleveref}[0.19]

\newcommand{\fy}[1]{{\color{black}#1}}
\newcommand{\fyy}[1]{{\color{black}#1}}
\newcommand{\far}[1]{{\color{black}#1}}

\newcommand{\ses}[1]{{\color{black}#1}}
\newcommand{\ssrtwo}[1]{{\color{black}#1}}

\usepackage[textsize=small]{todonotes}

\usepackage{tcolorbox}

\crefname{section}{section}{sections}
\crefname{subsection}{subsection}{subsections}
\Crefname{section}{Section}{Sections}
\Crefname{subsection}{Subsection}{Subsections}

\Crefname{figure}{Figure}{Figures}

\crefformat{equation}{\textup{#2(#1)#3}}
\crefrangeformat{equation}{\textup{#3(#1)#4--#5(#2)#6}}
\crefmultiformat{equation}{\textup{#2(#1)#3}}{ and \textup{#2(#1)#3}}
{, \textup{#2(#1)#3}}{, and \textup{#2(#1)#3}}
\crefrangemultiformat{equation}{\textup{#3(#1)#4--#5(#2)#6}}%
{ and \textup{#3(#1)#4--#5(#2)#6}}{, \textup{#3(#1)#4--#5(#2)#6}}{, and \textup{#3(#1)#4--#5(#2)#6}}

\Crefformat{equation}{#2Equation~\textup{(#1)}#3}
\Crefrangeformat{equation}{Equations~\textup{#3(#1)#4--#5(#2)#6}}
\Crefmultiformat{equation}{Equations~\textup{#2(#1)#3}}{ and \textup{#2(#1)#3}}
{, \textup{#2(#1)#3}}{, and \textup{#2(#1)#3}}
\Crefrangemultiformat{equation}{Equations~\textup{#3(#1)#4--#5(#2)#6}}%
{ and \textup{#3(#1)#4--#5(#2)#6}}{, \textup{#3(#1)#4--#5(#2)#6}}{, and \textup{#3(#1)#4--#5(#2)#6}}

\crefdefaultlabelformat{#2\textup{#1}#3}

\title{\LARGE \bf
Improved guarantees for optimal Nash equilibrium seeking and bilevel variational inequalities
}
\author{Sepideh Samadi\thanks{PhD student in the Department of Industrial and Systems Engineering, Rutgers University, Piscataway, NJ 08854, USA  ({\tt\small{sepideh.samadi@rutgers.edu}}).}
\and Farzad Yousefian\thanks{Assistant Professor in the Department of Industrial and Systems Engineering, Rutgers University, Piscataway, NJ 08854, USA
  ({\tt\small{farzad.yousefian@rutgers.edu}}).
 This work was funded in part by the NSF under CAREER grant ECCS-$1944500$, in part by the ONR under grant N$00014$-$22$-$1$-$2757$, and in part by the DOE under grant DE-SC$0023303$.}
}

\begin{document}
\sloppy
\maketitle
\thispagestyle{empty}
\pagestyle{plain}
\maketitle
\begin{abstract}We consider a class of hierarchical variational inequality (VI) problems that subsumes VI-constrained optimization and several other problem classes including the optimal solution selection problem \far{and} the optimal Nash equilibrium (NE) seeking problem. Our main contributions are threefold. (i) We consider bilevel VIs with \far{monotone} and Lipschitz continuous mappings and devise a single-timescale iteratively regularized extragradient method, \far{named IR-EG$_{{\texttt{m,m}}}$}. We improve the existing iteration complexity results for addressing both bilevel VI and VI-constrained convex optimization problems. (ii) Under the strong monotonicity of the outer level mapping, we develop \far{a method} named \far{IR-EG$_{{\texttt{s,m}}}$} and derive faster guarantees than those in (i). \far{We also study the iteration complexity of this method under a constant regularization parameter.} These results appear to be new for both bilevel VIs and VI-constrained optimization. (iii) To our knowledge, complexity guarantees for computing the optimal NE in nonconvex settings do not exist. Motivated by this lacuna, we consider VI-constrained nonconvex optimization problems and devise an inexactly-projected gradient method, named IPR-EG, where the projection onto the unknown set of equilibria is performed using \far{IR-EG$_{{\texttt{s,m}}}$} with \far{a} prescribed termination criterion and \far{an adaptive regularization parameter}. We obtain new complexity guarantees in terms of a residual map and an infeasibility metric for computing a stationary point. We validate the theoretical findings using preliminary numerical experiments for computing the best and the worst Nash equilibria.

\end{abstract}



\section{Introduction}\label{sec:intro}
Consider the class of optimization problems with variational inequality constraints (also known as VI-constrained optimization~\cite{facchinei2014vi,doi:10.1137/20M1357378}) of the form
\begin{align}\label{eqn:optVI} 
&\hbox{minimize}\quad f(x)\\
&\hbox{subject to}\quad x \in \mbox{SOL}(X,F),\nonumber
\end{align}
where $X \subseteq \mathbb{R}^n$ is a closed convex set, $f:X \to \mathbb{R}$ is a continuously differentiable (possibly nonconvex) function, and $F:X\to \mathbb{R}^n$ is a continuous and  \far{monotone} mapping. Here, $\mbox{SOL}(X,F)$ denotes the solution set of the variational inequality problem $\mbox{VI}(X,F)$, where $
\mbox{SOL}(X,F) \triangleq \left\{x \in X \mid F(x)\fyy{^\top}(y-x)\geq 0, \ \hbox{for all } y \in X\right\}$. One of the goals in this work is to devise a fully iterative method with complexity guarantees for addressing problem~\eqref{eqn:optVI} when $f$ is nonconvex. In addition, when $f$ is \far{convex} or strongly convex, we aim at considerably improving the existing complexity guarantees for solving this class of problems. Another goal in this work lies in improving the existing performance guarantees in solving a generalization of problem~\eqref{eqn:optVI}, \far{cast as a class of} bilevel variational inequality \far{problems},  given as follows.
\begin{align}\label{eqn:bilevelVI} 
&\hbox{Find}\quad x \in \mbox{SOL}(X^*_F,H) \qquad \hbox{where } X^*_F \triangleq \mbox{SOL}(X,F) 
\end{align}
and $H:X\to \mathbb{R}^n$ is a continuous and \far{monotone} mapping. In fact, note that when $f$ is convex and differentiable, \eqref{eqn:optVI} is an instance of~\eqref{eqn:bilevelVI} where $H(x):=\nabla f(x)$.

\subsection{Motivating problem classes}\label{subsection:motivation} Our motivation arises from some important problem classes in optimization and game theory presented as follows.

\subsubsection{Optimal solution selection problem} In optimization, the optimal solution may not be unique. When multiple optimal solutions exist, one may seek among the optimal solutions, one that minimizes a secondary metric. 
Consider 
\begin{align}\label{prob:opt_select}
&\hbox{minimize} \quad \phi(y) \\
& \hbox{subject to} \quad y \in Y^*,\notag
\end{align}
where $\phi:\mathbb{R}^p \to \mathbb{R}$ is a continuous function (possibly nonconvex) and $Y^*$ denotes the optimal solution set of the canonical constrained convex program of the form
\begin{align}\label{prob:cst_cvx}
&\hbox{minimize} \quad g(y) \\
& \hbox{subject to} \quad y \in Y, \qquad h_j(y) \leq 0, \qquad j=1,\ldots,q.\notag
\end{align}
Assume that $g:\mathbb{R}^p \to \mathbb{R}$, $h_j:\mathbb{R}^p \to \mathbb{R}$, $j=1,\ldots,q$, are continuously differentiable convex functions. Suppose the Slater condition holds (see~\cite[Assum.~6.4.2]{bertsekas2003convex}) and the set $Y \subseteq \mathbb{R}^p$ is closed and convex. Then, in view of the strong duality theorem~\cite[Prop.~6.4.3]{bertsekas2003convex}, the duality gap is zero. From the Lagrangian saddle point theorem~\cite[Prop.~6.2.4]{bertsekas2003convex}, any optimal solution-multiplier pair, denoted by $(y^*,\lambda^*)$, is a saddle-point of the Lagrangian function $\mathcal{L}:\mathbb{R}^{p+q} \to \mathbb{R}$ defined as $\mathcal{L}(y,\lambda) \triangleq g(y) + \lambda\fyy{^\top}h(y)$ where $\lambda\triangleq [\lambda_1;\ldots;\lambda_q]$ and $h(y)\triangleq [h_1(y);\ldots;h_q(y)]$. Then, problem~\eqref{prob:opt_select} can be written as a VI-constrained optimization of the form~\eqref{eqn:optVI}, where we define
\begin{align}\label{eqn:def_example1}
&x\triangleq [y;\lambda], \qquad X\triangleq Y\times \mathbb{R}^q_+, \qquad f(x)\triangleq \phi(y),\\
& F(x) \triangleq \left[\nabla_y \mathcal{L}(y,\lambda);-\nabla_\lambda \mathcal{L}(y,\lambda)\right]= \left[\nabla g(y) + \nabla h(y)\lambda;-h(y)\right], \nonumber
\end{align}
and $\nabla h(y) \triangleq [\nabla h_1(y),\ldots,\nabla h_q(y)]$. Notably, if functions $h_j$ are linear and $g$ is smooth, then $F$ is Lipschitz continuous on $X$. Further, given that $g$ and $h_j$ are convex in $y$, $F$ is \far{monotone} on $X$. \far{This is shown in Lemma~\ref{lem:example1_monotone_F}.}
\subsubsection{Optimal Nash equilibrium seeking problem} Nash games may admit multiple equilibria~\cite{papadimitriou2001algorithms}. A storied and yet, challenging question in noncooperative game theory pertains to seeking the best and the worst equilibria with respect to a given social welfare function. Addressing this question led to the emergence of efficiency estimation metrics, including Price of Anarchy (PoA) and Price of Stability (PoS), introduced in~\cite{koutsoupias1999worst,papadimitriou2001algorithms} and~\cite{anshelevich2008price}, respectively. Let us first formally recall the standard Nash equilibrium problem (NEP) and then, show how problem~\eqref{eqn:optVI} captures the optimal \fyy{equilibrium seeking} problem. Consider a canonical noncooperative game among $d$ players where player $i$, $i=1,\ldots,d$, is associated with a cost function $g_i(x^i,x^{-i})$, $x^i$ denotes the strategy of player $i$, and $x^{-i}\triangleq (x^{1},\ldots, x^{i-1}, x^{i+1}, \ldots,x^d)$ denotes the strategy of other players. Let $X_i \subseteq \mathbb{R}^{n_i}$ denote the strategy set of player $i$. In an NEP, player $i$ seeks to solve
\begin{align}\label{prob:NEP}
&\hbox{minimize}_{\ x^i} \quad g_i(x^{i},x^{-i}) \\
& \hbox{subject to} \quad x^i \in X_i.\notag
\end{align}
A Nash equilibrium is a tuple of specific strategies, $\bar{x}\triangleq (\bar{x}^1,\ldots,\bar{x}^d)$, where no player can reduce her cost by unilaterally
deviating from her strategy. This is mathematically characterized as follows. For all $i=1,\ldots,d$, 
$g_i(\bar{x}^{i},\bar{x}^{-i}) \leq g_i(x^{i},\bar{x}^{-i})$, for all $x^{i} \in X_i$. 
\far{Suppose} for each player $i$, $X_i$ is closed and convex, and \far{$g_i(\bullet,x^{-i})$} is continuously differentiable and \far{convex for all $x^{-i}$}. Then, in view of~\cite[Prop.~1.4.2]{FacchineiPang2003}, $\mbox{SOL}(X,F)$ is equal to the set of Nash equilibria, where $X \subseteq \mathbb{R}^n$ and $F:X \to \mathbb{R}^n$ are defined as 
\begin{align}\label{eqn:def_NE_VI}
&X\triangleq \textstyle\prod_{i=1}^dX_i, \quad  F(x) \triangleq \left[\nabla_{x^{1}} g_1(x);\ldots;\nabla_{x^{d}} g_d(x)\right], \quad  n\triangleq \sum_{i=1}^d n_i .  
\end{align}
 \fyy{Notably, if $F$ is monotone (e.g., when $F$ is continuously differentiable and its Jacobian is positive semidefinite), then $\text{SOL}(X, F)$ may contain multiple equilibria.} Let $\psi:\mathbb{R}^n \to \mathbb{R}$ denote a global metric for the group of the players. Some popular examples for $\psi$ include the following choices. (i) utilitarian approach: $\psi(x)\triangleq \frac{1}{d}\sum_{i=1}^d g_i(x)$. (ii) egalitarian approach: $\psi(x)\triangleq \max_{i=1,\ldots, d} g_i(x)$. (iii) risk-averse approach: $\psi(x)\triangleq \rho(g_1(x),\ldots,g_d(x))$ where $\rho:\mathbb{R}^d\to \mathbb{R}$ denotes a suitably defined risk measure for the group of the players. In all these cases, the problem of computing the best (or worst) NE with respect to the metric $\psi(x)$ is precisely captured by problem~\eqref{eqn:optVI} where $f(x):= \psi(x)$ (or $f(x):= -\psi(x)$), respectively.

\subsubsection{Generalized Nash equilibrium problem (GNEP)}\label{sec:GNEPs} \far{Consider a} generalized Nash game~\cite{facchinei2010generalized} where the strategy sets of the players may involve shared constraints.  \far{Player $i$,  for $i=1,\ldots,d$, seeks to solve a problem of the form}
\begin{align}\label{prob:GNEP}
&\hbox{minimize}_{\ x^i} \quad g_i(x^{i},x^{-i}) \\
& \hbox{subject to} \quad x^i \in \mathcal{X}_i	(x^{-i}),\notag 
\end{align}
where $\mathcal{X}_i(x^{-i}) \subseteq \mathbb{R}^{n_i}$ denotes the strategy set of player $i$ that depends on the strategies of the rival players. Let us define the point-to-set mapping $\mathcal{X}(x)\triangleq \prod_{i=1}^d \mathcal{X}_i(x^{-i})$. Then, $\bar{x}\triangleq (\bar{x}^1,\ldots,\bar{x}^d)$ is a generalized NE if and only if $\bar{x} \in \mathcal{X}(\bar{x})$ and for all $i$,
\begin{align*}
g_i(\bar{x}^{i},\bar{x}^{-i}) \leq g_i(x^{i},\bar{x}^{-i}), \qquad \hbox{for all } x^{i} \in \mathcal{X}_i	(\bar{x}^{-i}).
\end{align*}
\far{We review two instances of GNEPs and examine their formulation as a bilevel VI. }

 
\noindent \far{(i)} GNEP over a nonlinear system of equations. Consider the case when the shared constraints \far{are characterized as the solution set of} a system of nonlinear equations, given as $\mathcal{S}\triangleq \left\{x \in \mathbb{R}^n \mid F(x)=0\right\}$, where \far{$\mathcal{S} \neq \emptyset$} and $F$ is a continuous and \far{monotone} mapping.  \far{Notably, it follows that $\mbox{SOL}(\mathbb{R}^n,F) = \mathcal{S}$.  Also, in view of \cite[Thm. 2.3.5]{FacchineiPang2003},  $\mathcal{S} $ is convex.  This implies that $ \mathcal{X}_i	(x^{-i}) \triangleq \{x^i \in \mathbb{R}^{n_i} \mid (x^{i},x^{-i}) \in \mathcal{S}\}$ is convex for any $x^{-i}$,  which then implies that the GNEP \eqref{prob:GNEP} is {\it jointly convex}~\cite[Def.~2]{facchinei2010generalized}.  Let us define $H(x) \triangleq \left[\nabla_{x^{1}} g_1(x);\ldots;\nabla_{x^{d}} g_d(x)\right]$}. Then,  \far{the following holds. If $x^* \in \mathcal{S}$ solves \eqref{eqn:bilevelVI}, then $x^*$ is an NE to the GNEP \eqref{prob:GNEP}. It is important to recall that the converse is not necessarily true~\cite{facchinei2010generalized}.} 

\noindent \far{(ii)} GNEP with complementarity constraints.  \far{Let $\mathcal{S}\triangleq \left\{x \in \mathbb{R}^n \mid  0 \leq x \perp F(x) \geq 0 \right\}$. Similarly,} noting that $\mbox{SOL}(\mathbb{R}^n_{+},F) = \mathcal{S}$,  \far{if $x^* $ solves \eqref{eqn:bilevelVI}, then it is a solution to \eqref{prob:GNEP}.}


\subsection{Related work and gaps} Next, we briefly review the literature on VIs and then, provide an overview of the prior research in addressing~\eqref{eqn:optVI} and~\eqref{eqn:bilevelVI}. 

\noindent {\bf (i) Background on VIs.} The theory, applications, and methods for the VIs have been extensively studied in the past decades~\cite{FacchineiPang2003,rockafellar2009variational}. 
In addressing deterministic VIs with \far{monotone} and Lipschitz continuous maps, the extragradient-type methods improve the iteration complexity of the gradient methods from $\mathcal{O}(\epsilon^{-2})$ to $\mathcal{O}(\epsilon^{-1})$~\cite{korpelevich1976extragradient,nemirovski2004prox}. Employing Monte Carlo sampling, stochastic variants of the gradient and extragradient methods have been developed more recently in~\cite{jiang2008stochastic,juditsky2011solving,yousefian2018stochastic,yousefian2014optimal,kannan2019optimal,kotsalis2022simple}
and their variance-reduced extensions~\cite{iusem2017extragradient,alacaoglu2022stochastic}, among others. 

\noindent {\bf Gap (i).} Despite these advances, the abovementioned references do not provide iterative methods for addressing the two formulations considered in this work.

\noindent {\bf (ii) Methods for the optimal solution selection problem.} Problem~\eqref{prob:opt_select} has been recently studied as a class of bilevel optimization problems~\cite{yamada2005hybrid,solodov2007explicit,solodov2007bundle,
beck2014first,sabach2017first,yousefian2021bilevel,jiang2023conditional,
chen2023bilevel,shen2023online,latafat2023adabim,merchav2023convex,samadi2024achieving} and appears to have its roots in the study of ill-posed optimization problems and the notion of exact regularization~\cite{ferris1991finite,friedlander2008exact,tikhonov1963solution,amini2019iterative}.

\noindent {\bf Gap (ii).}  It appears that there are no methods endowed with complexity guarantees for addressing~\eqref{prob:opt_select} with explicit lower-level functional constraints and possibly, with a nonconvex objective function $\phi(\bullet)$. 
%
%
%
%

\noindent {\bf (iii) Methods for VI-constrained optimization.} Naturally, solving problem~\eqref{eqn:optVI} is challenging, due to the following two reasons. (i) The feasible solution set in \eqref{eqn:optVI} is itself the solution set of a VI problem and so, it is often unavailable. Indeed, as mentioned in~\cite{feinstein2023characterizing}, there does not seem to exist any algorithms that can compute all Nash equilibria. (ii) The standard Lagrangian duality theory can be employed when the feasible set of the optimization problem is characterized by a finite number of functional constraints. However, by definition, $\mbox{SOL}(X,F)$ is characterized by infinitely many, and possibly nonconvex, inequality constraints. Indeed, there are only a handful of papers that have developed efficient iterative methods for computing the optimal equilibrium~\cite{yamada2005hybrid,facchinei2014vi,yousefian2017smoothing,doi:10.1137/20M1357378,kaushik2023incremental,jalilzadeh2024stochastic}. Among these, only the methods in~\cite{doi:10.1137/20M1357378,kaushik2023incremental,jalilzadeh2024stochastic} are equipped with suboptimality and infeasibility error bounds. 

\noindent {\bf Gap (iii).} The abovementioned three references do not address the case when the objective $f$ is nonconvex. Also, when $f$ is convex, the best known iteration complexity for suboptimality and infeasibility metrics appears to be $\mathcal{O}(\epsilon^{-4})$. In this work, our aim is to improve such guarantees for the convex case and also, achieve new guarantees for the strongly convex and nonconvex cases.

%

\noindent {\bf (iv) Methods for bilevel VIs and GNEPs.} 
Methods for addressing bilevel VIs and their closely-related variants have been recently developed in references including~\cite{facchinei2014vi,thong2020strong,van2021regularization,lampariello2022solution}. The GNEP has been extensively employed in several application areas, including economic sciences and telecommunication systems~\cite{facchinei2010generalized,chan1982generalized}. While the research on GNEPs has been more active in recent years~\cite{benenati2023optimal,benenati2023semi}, there are a limited number of methods equipped with provable rates for solving GNEPs, including~\cite{alizadeh2024randomized}, where Lagrangian duality theory is employed for relaxing the shared constraints. 

\noindent {\bf Gap (iv).} Among the references on bilevel VIs, only the work in~\cite{lampariello2022solution} appears to provide complexity guarantees, that is $\mathcal{O}(\epsilon^{-8})$ when $H$ is \far{monotone}. We aim at improving these guarantees significantly and also, providing new guarantees when $H$ is strongly monotone. With regard to GNEPs, our work provides a new avenue for solving \far{a variational formulation of these} problems  when the shared constraints do not admit the conventional form, examples of which are provided in subsection~\ref{sec:GNEPs}.

\subsection{Contributions} Motivated by these gaps, our contributions are as follows. 

 {\footnotesize 
\begin{table}
  \centering
  \tiny
 \caption{Summary of main contributions and assumptions in addressing problem~\eqref{eqn:bilevelVI}}
\begin{tabular}{|c|c|c|c|c|c|c|c|c|}
    \hline
    \multirow{3}{*}{\far{Method}} & \multicolumn{4}{|c|}{\multirow{1}{*}{\far{Main assumptions}}} & \multicolumn{4}{|c|}{\multirow{1}{*}{\far{Main convergence rate statements}}} \\
    \cline{2-9}
     &  \multirow{2}{*}{\far{$F$}}& \multirow{2}{*}{\far{$H$}}& \multicolumn{2}{|c|}{\far{$\mbox{SOL}(X,F)$}} & \multicolumn{2}{|c|}{\multirow{1}{*}{\far{Outer-level error metric}}} & \multicolumn{2}{|c|}{ \multirow{1}{*}{\far{Inner-level error metric}}}  \\ 
    \cline{4-9}
\multirow{3}{*}{\ses{IR-EG$_{{\texttt{m,m}}}$}}  &\multirow{3}{*}{\far{m.}}  & \multirow{3}{*}{\far{m.}} & \ses{w.sh.} &  \far{$\alpha$ known?} & \far{L.B.} &\far{U.B.} & \far{L.B.}&\far{U.B.}\\
    \cline{1-9}
   & && \far{$\times$} & \far{$\times$} & \far{asymptotic}  & \far{$ \tfrac{1}{\sqrt{K}} $} &\far{0} & \far{$ \tfrac{1}{\sqrt{K}} $} \\
     \cline{4-9}
 \far{(both diminishing } &&&\ses{$\footnotesize\mathcal{M} \geq 1$} & \far{$\times$} &
   \far{$-\sqrt[2\mathcal{M}]{{\tfrac{1}{ K}} }$}  &\far{ $ \frac{1}{\sqrt{K}} $} & \far{0} & \far{$ \tfrac{1}{\sqrt{K}} $}\\
     \cline{4-9}
      \far{and constant reg.)} &&&\ses{$\footnotesize\mathcal{M}=1$}  & \far{$\checkmark$} & \far{$-\frac{1}{K}$} & \far{$\frac{1}{K}$}  &\far{0}& \far{$\frac{1}{K}$} \\
          \hline
     \cline{1-9}
         \multirow{2}{*}{\ses{IR-EG$_{{\texttt{s,m}}}$} } & \multirow{2}{*}{\far{m.}}& \multirow{2}{*}{\far{s.m.}} &\far{$\times$} & \far{$\times$} &\far{asymptotic}& \ses{$ \frac{1}{K} $}  &   \far{0} & \far{$\tfrac{1}{K}$}  \\
    \cline{4-9}
   \far{(diminishing reg.)}& && \ses{$\mathcal{M}\geq 1$} & \far{$\times$} &  \far{$-\sqrt[\mathcal{M}]{{\tfrac{1}{ K}} }$} &\far{$ \tfrac{1}{K} $} &   \far{0} & \far{$\frac{1}{K}$}  \\
    \hline
    \cline{1-9}
    \multirow{3}{*}{\ses{IR-EG$_{{\texttt{s,m}}}$}} & \multirow{3}{*}{\far{m.}}& \multirow{3}{*}{\far{s.m.}} &\far{$\times$} & \far{$\times$} &\far{-} & \far{$ \frac{1}{K^{p}} $}  &  \far{ 0} & \far{$\tfrac{1}{ K^{p+1}} +\tfrac{1}{K}$}  \\
    \cline{4-9}
 \far{ (constant reg.) }& && \ses{$\mathcal{M}\geq 1$} & \far{$\times$}&  \ses{$-\sqrt[\mathcal{M}]{{\tfrac{1}{ K^{p+1}}} +{\frac{1}{K}}}$}& \far{$ \tfrac{1}{K^{p}} $} &   0 & \far{$\tfrac{1}{ K^{p+1}} +\frac{1}{K}$ } \\
     \cline{4-9}
    & && \ses{$\mathcal{M}=1$} & \far{$\checkmark$} & \far{-$\rho_1^K$} & \far{ $\rho_1^K$} &  \far{0}  &  \far{$\rho_2^K$}\\
 \hline
     \cline{1-9}
    \multirow{2}{*}{\far{IPR-EG}} & \multirow{2}{*}{\far{m.}}& \far{ $\nabla f$},  &\ses{$\mathcal{M} \geq 1$} & \far{$\times$} &\far{0}& \far{$\frac{1}{\sqrt{K}}$}   &   \far{0} &  \far{$\tfrac{1}{K\sqrt{K}}$ } \\
    \cline{4-9}
   & &\far{ncvx $f$} & \ses{$\mathcal{M}=1$} & \far{$\checkmark$} &  \far{0} & \far{$\tfrac{1}{\sqrt{K}}$}  &   \far{0} & \far{ $\tfrac{1}{K^2}$}\\
    \hline
  \end{tabular} 
  \vspace{-0.05in}
  \label{table:contributions}
    \caption*{\tiny \far{Notation: m.  and s.m. denote monotone and strongly monotone, resp.; \ses{w.sh.}  denotes $\alpha$-weakly sharp \ses{of order $\mathcal{M}$}; L.B. and U.B. denote lower and upper bound, resp.; $p \geq 1$ is arbitrary, $\rho_1, \rho_2 \in (0,1)$. Logarithmic numerical factors are ignored.}}
\end{table}
}
 
\noindent {\bf (i) Improved guarantees for monotone bilevel VIs.} When mappings $F$ and $H$ are both \far{monotone} and Lipschitz continuous, we develop a single-timescale iteratively regularized extragradient method (\far{IR-EG$_{{\texttt{m,m}}}$}) for addressing problem~\eqref{eqn:bilevelVI} and derive new iteration complexity of $\mathcal{O}\left(\epsilon^{-2}\right)$ \far{where $\epsilon$ is an upper bound on suitably defined gap functions for the inner and outer-level VI problems. We also show that the lower bound on the outer-level VI's gap function converges to zero asymptotically. Under a weak sharpness assumption on the inner VI problem with order $\mathcal{M}\geq 1$, we derive a rate for the lower bound on the outer-level gap function.  We further show that  when $\mathcal{M}=1$ and the regularization parameter is below a prescribed threshold, the complexities improve to $\mathcal{O}\left(\epsilon^{-1}\right)$.} For bilevel VIs, this result improves the existing complexity $\mathcal{O}(\epsilon^{-8})$ for this class of problems in prior work~\cite{lampariello2022solution}. Also, when $H(x):=\nabla f(x)$, addressing VI-constrained optimization, this improves the existing complexity $\mathcal{O}(\epsilon^{-4})$ in~\cite{doi:10.1137/20M1357378,kaushik2023incremental,jalilzadeh2024stochastic} by leveraging smoothness of $f$ and Lipschitz continuity of $F$.   

\noindent {\bf (ii) New guarantees for strongly monotone bilevel VIs.} When $H$ is strongly monotone, we show that a variant of \far{IR-EG$_{{\texttt{m,m}}}$}, called \far{IR-EG$_{{\texttt{s,m}}}$}, under a prescribed weighted averaging scheme, can be endowed with a \far{sublinear} convergence speed in terms of an \far{upper} bound on the dual gap function of \far{both} the outer \far{and inner VI problems}. We also \far{develop a variant of this method under a constant regularization parameter and derive} complexities of \far{$\mathcal{O}\left(\epsilon^{-\frac{1}{p}}\right)$} and \far{$\mathcal{O}\left(\epsilon^{\frac{-1}{(p+1)}}+\epsilon^{-1}\right)$} for the outer and inner VI problem, respectively, for any arbitrary $p\geq 1$. \far{We further show that under the weak sharpness assumption,  when the regularization parameter is below a prescribed threshold,  the generated iterate converges to the unique solution with a linear speed.} These guarantees appear to be new for both \eqref{eqn:optVI} and \eqref{eqn:bilevelVI}.

\noindent {\bf (iii) New guarantees for VI-constrained nonconvex optimization.} We develop an iterative method, called IPR-EG, for addressing problem~\eqref{eqn:optVI} when $f$ is smooth and nonconvex. We \far{analyze the convergence of the method under the weak sharpness of the VI problem and show the following results.  (a) When a threshold for the regularization parameter is unknown a priori, we develop an adaptive update rule for the regularization parameter and derive} an overall iteration complexity of nearly $\mathcal{O}\left(\epsilon^{-\far{3\mathcal{M}-2}}\right)$ in computing a stationary point, \far{where $\mathcal{M}$ denotes the order of the weak sharpness property} (see Theorem~\ref{theorem:NC}); \far{(b) When the threshold is known \far{and $\mathcal{M}=1$},  using a constant regularization parameter, we show that the overall iteration complexity improves to nearly $\mathcal{O}\left(\epsilon^{-2}\right)$.} Our method appears to be the first fully iterative scheme equipped with iteration complexity for solving VI-constrained nonconvex optimization problems, and in particular, for computing the worst equilibrium in monotone Nash games. \fyy{Table~\ref{table:contributions} provides a summary of the main contributions.}

\subsection*{Outline of the paper} The remainder of the paper is organized as follows. In section~\ref{sec:prelim}, we present \far{some preliminaries. In section~\ref{sec:mon}, we address bilevel VIs with a monotone outer-level mapping. In section~\ref{sec:smon}, we address bilevel VIs with a strongly monotone outer-level mapping. We also provide the implication of the results for VI-constrained strongly convex optimization problems.  In section~\ref{sec:ncvx}, we address VI-constrained nonconvex optimization.} Some preliminary numerical results are presented in section~\ref{sec:num}. Lastly, we provide \far{some} concluding remarks in section~\ref{sec:conclude}.   

\subsection*{Notation} Throughout, a vector $x \in \mathbb{R}^n$ is assumed to be a column vector. 	The transpose of $x$ is denoted by $x\fyy{^\top}$. Given two vectors $x\in \mathbb{R}^n$ and $y\in \mathbb{R}^m$, we use $[x;y] \in \mathbb{R}^{(n+m)\times 1}$ to denote their concatenated column vector. When $n=m$, we use $[x,y] \in \mathbb{R}^{n\times 2}$ to denote their side-by-side concatenation. We let $\|\cdot\|$ denote the Euclidean norm. A mapping $F:X \to \mathbb{R}^n$ is said to be monotone on a convex set $X \subseteq \mathbb{R}^n$ if $(F(x)-F(y))\fyy{^\top}(x-y)\geq 0$ for any $x,y \in X$. The mapping $F$ is said to be $\mu$-strongly monotone on a convex set $X \subseteq \mathbb{R}^n$ if $\mu>0$ and $(F(x)-F(y))\fyy{^\top}(x-y)\geq \mu\|x-y\|^2$ for any $x,y \in X$. Also, $F$ is said to be Lipschitz continuous with parameter $L>0$ on the set $X$ if $\|F(x)-F(y)\|\leq L\|x-y\|$ for any $x,y \in X$. A continuously differentiable function $f:X \to \mathbb{R}$ is called $\mu$-strongly convex on a convex set $X$ if $f(x) \geq f(y)+\nabla f(y)\fyy{^\top}(x-y)+\frac{\mu}{2}\|x-y\|^2$. The Euclidean projection of vector $x$ onto a closed convex set $X$ is denoted by $ \Pi_{X}[x] $, where $\Pi_{X}[x]\triangleq  \mbox{arg}\min_{ y \in  X}\| x- y\|$. We let $\mbox{dist}(x,X)\triangleq \|x-\Pi_{X}[x]\|$ denote the distance of $x$ from the set $X$. \far{Throughout, we let $X^*$ denote the solution set of problem~\eqref{eqn:bilevelVI}.}

\section{Preliminaries}\label{sec:prelim}
In the following, we provide some definitions and preliminary results that will be utilized in the analysis in this section. 
\begin{lemma}[Projection theorem {\cite[Prop.~2.2.1]{bertsekas2003convex}}]\label{lem:Projection theorm}\em
Let $X$ be a closed convex set. Given any arbitrary $x \in \mathbb{R}^n$, a vector $x^* \in X$ is equal to $\Pi_X[x]$ if and only if for all $y\in X$ we have $(y-x^*)\fyy{^\top}(x-x^*) \leq 0$. 
\end{lemma}
To quantify the quality of the generated iterates by the proposed algorithms, we use the dual gap function~\cite{juditsky2011solving,yousefian2017smoothing,doi:10.1137/20M1357378}, defined as follows. 
\begin{definition}[Dual gap function]\em
Let $X\subseteq \mathbb{R}^n$ be a nonempty, closed, and convex set and $F:X \to \mathbb{R}^n$ be a continuous mapping. Then, the dual gap function associated with $\mbox{VI}(X,F)$ is an extended-valued function $\mbox{Gap}:\mathbb{R}^n \to \mathbb{R}\cup\{+\infty\}$ given as $ \mbox{Gap}(x,X,F)\triangleq \sup_{ y\in  {X}} F(y)\fyy{^\top}(x-y)$.
\end{definition}
\begin{remark} \label{rem:gap}
Note that when $F$ is continuous and monotone, if $x \in X$, then $\mbox{Gap}(x,X,F)=0$ if and only if $x \in \mbox{SOL}(X,F)$~\cite{juditsky2011solving}. Similar to~\cite[page 167]{FacchineiPang2003}, here we define the  dual gap function for vectors both inside and outside of the set $X$. Note that if $x \in X$, then $ \mbox{Gap}(x,X,F) \geq 0$. However, this may not be the case for $x \notin X$. In the case when the generated iterate by an algorithm is possibly infeasible to the set of the VI (e.g., this emerges in solving bilevel VIs), it is important to build both a lower bound and an upper bound on the dual gap function.
\end{remark}
In some of the rate results, we utilize the weak sharpness property, defined as follows. 
\begin{definition}[Weak sharpness{~\cite{marcotte1998weak,kannan2019optimal}}]\label{def:weaksharp}\em
Let $X\subseteq \mathbb{R}^n$ be a nonempty, closed, and convex set and $F:X \to \mathbb{R}^n$ be a continuous mapping. \far{$\mbox{VI}(X,F)$ is said to be $\alpha$-weakly sharp with order $\mathcal{M}$ if there exist scalars} $\alpha>0$ \ses{and $\mathcal{M} \geq 1$} such that for all $x \in X$ and all $x^* \in  \mbox{SOL}(X,F)$, we have 
\ses{$F(x^*)\fyy{^\top}(x- x^*) \geq \alpha\,  \mbox{dist}^\mathcal{M}(x, \mbox{SOL}(X,F)).$}
\end{definition}
\far{\begin{remark} \label{rem:ws}
The notion of weak sharp minima was first introduced in~\cite{burke1993weak} to account for the possibility of non-unique solutions in optimization and complementarity problems. Some sufficient conditions and characterizations for weak sharp minima were studied in~\cite{studniarski1999weak} subsequently. An important subclass of VIs that admits this property with $\mathcal{M}=1$ is the monotone linear complementarity problem (LCP), \fyy{under the assumption that the LCP is nondegenerate}~\cite[Thm. 3.7]{burke1993weak}. LCP is a powerful mathematical program that has been applied to a wide range of problems in operations research and game theory (cf.~\cite{cottle2009linear} for an in-depth study). 
\end{remark}}
Throughout, in deriving some of the rate results, we may utilize the following terms. 
\begin{definition}\label{def:terms_bounded}\em
Let $D_X^2\triangleq \sup_{x,y \in X}\ \frac{1}{2}\|x-y\|^2$ and define $B_F, B_H, B_f>0$ such that $B_F\triangleq \sup_{x \in X^*_F} \|F(x)\|$, $B_H\triangleq \sup_{x \in X^*_F} \|H(x)\|$, and $B_f \triangleq \sup_{x \in X^*_F} \|\nabla f(x)\|$. We also define \far{$B_F^*, B_H^*, B_f^*>0$ similarly, but with respect to the set $X^*$. Further,} we define $C_F, C_H, C_f>0$ \fy{in a similar vein}, with respect to the set $X$.
\end{definition}
\section{Monotone bilevel VIs and VI-constrained convex optimization} \label{sec:mon}
The main assumptions in this section are presented in the following. 
\begin{assumption}\label{assum:bilevelVI_m} \em
Consider problem~\eqref{eqn:bilevelVI}. Let the following statements hold. 

\noindent (a) Set $X$ is nonempty, closed, and convex.

\noindent (b) Mapping $F:X \to \mathbb{R}^n$ is $L_F$-Lipschitz continuous and \far{monotone} on $X$.

\noindent  (c) Mapping $H:X \to \mathbb{R}^n$ is  $L_H$-Lipschitz continuous and \far{monotone} on $X$.

\noindent (d) The solution set of problem~\eqref{eqn:bilevelVI} is nonempty.
\end{assumption}
\begin{remark}\label{rem:assum_monotone}
We note that Assumption~\ref{assum:bilevelVI_m}~(d) can be met under several settings. We provide an example in the following. Let $X$ be a closed convex set and $F$ be monotone. Then, in view of~\cite[Theorem~2.3.5]{FacchineiPang2003}, $\mbox{SOL}(X,F)$ is convex. If, additionally, $X$ is bounded, then from~\cite[Corollary~2.2.5]{FacchineiPang2003}, $\mbox{SOL}(X,F)$ is nonempty and compact. Now, let us consider $\mbox{VI}(\mbox{SOL}(X,F),H)$ and let $H$ be monotone. Invoking the preceding two results once again, we can conclude that $\mbox{SOL}(\mbox{SOL}(X,F),H)$ is nonempty, compact, and convex. 
\end{remark}
\subsection{Algorithm outline}
We devise Algorithm~\ref{alg:IR-EG} to address two problems. (i) The bilevel VI problem \eqref{eqn:bilevelVI} where both $F$ and $H$ are assumed to be \far{monotone} and Lipschitz continuous. (ii) The VI-constrained optimization problem~\eqref{eqn:optVI} when $f$ is $L$-smooth and convex, by setting $H(x):=\nabla f(x)$. Algorithm~\ref{alg:IR-EG} is essentially an iteratively regularized extragradient method, \far{named IR-EG$_{{\texttt{m,m}}}$}. A key idea lies in employing the regularization technique to incorporate both mappings $H$ and $F$. At iteration $k$, the vectors $y_k$ and $x_k$ are updated by using the regularized map $F(\bullet) +\eta_kH(\bullet)$. The regularization parameter $\eta_k$ is updated iteratively. Intuitively, the update rule of $\eta_k$ regulates the trade-off between incorporating the information of mapping $F$ and $H$ \far{in updating the iterates}. 

\begin{algorithm}[H]
\caption{\far{IR-EG$_{{\texttt{m,m}}}$} for bilevel VI~\eqref{eqn:bilevelVI} with \far{monotone} $H$ and $F$}
\label{alg:IR-EG}
\begin{algorithmic}[1]
\STATE \textbf{input:} Initial vectors $x_0, \far{\bar{y}_0} \in X$, a stepsize $\gamma > 0$, an initial regularization parameter $\eta_0 > 0$, and $0\leq b<1$.
\FOR{$k = 0, 1, \ldots, K-1$}
\STATE $y_{k+1}  := \Pi_X\left[x_k - {\gamma}\left(F(x_k)+\eta_k H(x_k)\right)\right]   \label{line:update_y}$
 
\STATE $x_{k+1} := \Pi_X\left[x_k - {\gamma}\left(F(y_{k+1}) +\eta_k H(y_{k+1})\right)\right]  \label{line:update_x}$ 
 
\STATE $\bar{y}_{k+1} := \left(k \bar{y}_k +  y_{k+1}\right)/(k+1)$
\STATE $\eta_{k+1} :=\frac{\eta_0}{(k+1)^b}$ 
\ENDFOR
\STATE {\bf return} $\bar{y}_K$
\end{algorithmic}
\end{algorithm}
\subsection{Convergence analysis for \far{IR-EG$_{{\texttt{m,m}}}$}}
In the bilevel VI \eqref{eqn:bilevelVI}, \far{note that} $\mbox{Gap}\left(x,\mbox{SOL}(X,F),H\right)  \geq 0$ for  $x \in \mbox{SOL}(X,F)$. This follows directly from Remark~\ref{rem:gap}. However, $\mbox{Gap}\left(x,\mbox{SOL}(X,F),H\right) $ might be negative for  $x \notin \mbox{SOL}(X,F)$. In the following, we establish a lower bound on the dual gap function of the outer VI problem for any vector $x\in X$ that might not belong to $\mbox{SOL}(X,F)$. \far{The proof is presented in section~\ref{sec:proof_lemma_ws}.}
\begin{lemma}\label{lem:bilevelVI_gap_ws}\em
Consider problem~\eqref{eqn:bilevelVI}. Let Assumption~\ref{assum:bilevelVI_m} hold. Then, the following results hold.

\far{
\noindent (i) $\mbox{Gap}\left(x,\mbox{SOL}(X,F),H\right)   \geq - B_H\, \mbox{dist}\left(x,\mbox{SOL}(X,F)\right) $ for any $x \in X$. \far{In particular,  for any $x \in \mbox{SOL}(X,F)$ we have $\mbox{Gap}\left(x,\mbox{SOL}(X,F),H\right)  \geq 0$.} 

 \noindent (ii)  Consider the special case $H(x) := \nabla f(x)$ where $f$ is a continuously differentiable and convex function. Then, for any $x \in X$ we have 
$$f(x) -f(x^*) \geq - \|\nabla f(x^*)\|\,\mbox{dist}\left(x,\mbox{SOL}(X,F)\right), $$ where $x^*$ denotes an optimal solution to problem~\eqref{eqn:optVI}. Further, if $f$ is $\mu$-strongly convex, then for any $x \in X$ we have 
\far{\begin{align*}  \frac{\mu}{2}\|x-x^*\|^2 &\leq f(x) -f(x^*)+ \|\nabla f(x^*)\|\,\mbox{dist}(x,\mbox{SOL}(X,F)),
\end{align*} }
where $x^*$ denotes the unique optimal solution to problem~\eqref{eqn:optVI}. 

\noindent (iii) Suppose $\mbox{SOL}(X,F)$ is $\alpha$-weakly sharp of order $\mathcal{M}\geq 1$.  Then, 

\noindent $\ \mbox{dist}\left(x,\mbox{SOL}(X,F)\right)\leq   \far{\sqrt[\mathcal{M}]{\alpha^{-1}\mbox{Gap}\left(x,X,F\right)}}$ for any $x \in X$.}



\end{lemma}
\far{Next, we provide some recursive inequalities} that will be used in the rate analysis. \far{The proof is presented in section~\ref{sec:proof_lemma_ineq}.}
\begin{lemma}\label{lem:monotone_lemma_ineq}\em
Consider problem~\eqref{eqn:bilevelVI}. Let $\{\bar{y}_k\}$ be generated by Algorithm~\ref{alg:IR-EG}. Let Assumption~\ref{assum:bilevelVI_m} hold. \far{The following statements hold for any  $x\in X$.}

\noindent \far{(i) For all $k\geq 0$, we have
\begin{align}\label{eqn:EG_lemma_mm_last}
 \|x_{k+1} - x\|^2 & \leq\|x_k - x\|^2   - \|y_{k+1} - x_k\|^2+ 2{\gamma}^2(L_F^2+\eta_k^2L_H^2)\|y_{k+1} - x_k\|^2 \\
 & + 2{\gamma}\left(F(y_{k+1})+\eta_kH(y_{k+1})\right)\fyy{^\top}\left(x-y_{k+1}\right).\notag
\end{align}} 
\noindent \far{(ii) Suppose ${\gamma}^2(L_F^2+\eta_k^2L_H^2) \leq 0.5$ for all $k\geq 0$. For all $k\geq 0$,} we have
\begin{align}\label{ineq:EG_lemma_merelymonotone}
2{\gamma}\left(F(x)+\eta_kH(x)\right)\fyy{^\top}\left(y_{k+1}-x\right)   & \leq  \|x_k - x\|^2   -\|x_{k+1} - x\|^2 .
\end{align}
\noindent \far{(iii) Suppose ${\gamma}^2(L_F^2+\eta_k^2L_H^2) \leq 0.5$ for all $k\geq 0$.} Further, if $H(x) := \nabla f(x)$ where $f$ is a convex and $L$-smooth function, then \far{for all $k\geq 0$, we have}
\begin{align}\label{ineq:EG_lemma_merelymonotone2}
2{\gamma}F(x)\fyy{^\top}\left(y_{k+1}-x\right) +2{\gamma}\eta_k(f(y_{k+1}) -f(x))  & \leq  \|x_k - x\|^2   -\|x_{k+1} - x\|^2 .
\end{align} 
\end{lemma}
\far{Next, we derive error bounds and show the asymptotic convergence of the method.}
 \begin{tcolorbox}[colback=blue!5!white,colframe=blue!55!black]
\far{\begin{theorem}[Error bounds and asymptotic convergence of IR-EG$_{{\texttt{m,m}}}$]\label{Thm:asymptotic-m.m}\em
Consider problem~\eqref{eqn:bilevelVI}. Let $\{\bar{y}_k\}$ be generated by Algorithm~\ref{alg:IR-EG}, let Assumption~\ref{assum:bilevelVI_m} hold, and let $X$ be bounded. Assume that $\{\eta_k\}$ is \ssrtwo{diminishing} and ${\gamma}^2(L_F^2+\eta_0^2L_H^2) \leq 0.5$. Then, the following results hold. 

\noindent (i) For all $K\geq 1$, we have 
$$ - B_H\, \mbox{dist}\left(\bar{y}_K,\mbox{SOL}(X,F)\right) \leq \mbox{Gap}\left(\bar{y}_K,\mbox{SOL}(X,F),H\right)      \leq  (\gamma\eta_{K-1} K)^{-1}  D_X^2.$$  
 
 \noindent (ii) For all $K\geq 1$, we have 
$$ 0 \leq \mbox{Gap}\left(\bar{y}_K,X,F\right)       \leq  (\gamma K)^{-1}D_X^2 +\sqrt{2} C_HD_X K^{-1}\textstyle\sum_{k=0}^{K-1}\eta_k.$$
 
 \noindent (iii) Suppose $\lim_{k\to \infty}k\eta_{k-1} = \infty$ and $\lim_{k\to \infty}\tfrac{\sum_{j=0}^{k-1} \eta_j}{k} = 0$, e.g., when $\eta_k:=\frac{1}{(k+1)^b}$ for $k\geq 0$, where $0<b<1$. Then, every accumulation point of $\{ \bar y_k\}$ belongs to the solution set  of the problem~\eqref{eqn:bilevelVI}, that is $ \mbox{SOL}(\mbox{SOL}(X,F),H)$.
\end{theorem}}
  \end{tcolorbox} 
\begin{proof}
\noindent (i) The lower bound holds using Lemma~ \ref{lem:bilevelVI_gap_ws}. Below, we show that the upper bound holds. Consider the inequality~\eqref{ineq:EG_lemma_merelymonotone}. Let $x^*_F \in X^*_F\triangleq \mbox{SOL}(X,F)$ denote an arbitrary solution to $\mbox{VI}(X,F)$. Thus, we have $F(x^*_F)\fyy{^\top}\left(y_{k+1}-x^*_F\right)\geq 0$, where we recall that $y_{k+1} \in X$. Invoking this relation, from~\eqref{ineq:EG_lemma_merelymonotone}, we have for all $k\geq 0$
\begin{align}\label{ineq:thm1_separated_ineq1}
 H(x^*_F)\fyy{^\top}\left(y_{k+1}-x^*_F\right)   & \leq 0.5(\gamma\eta_k)^{-1}\|x_k - x^*_F\|^2   -0.5(\gamma\eta_k)^{-1}\|x_{k+1} - x^*_F\|^2 ,
\end{align}
where we divided both sides by \far{$2\gamma\eta_k$. By adding} and subtracting $0.5\gamma^{-1}\eta_{k-1}^{-1}\|x_{k}-x^*_F\|^2$, we obtain
\begin{align*} 
 H(x^*_F)\fyy{^\top}\left(y_{k+1}-x^*_F\right)   & \leq 0.5\gamma^{-1}\eta_{k-1}^{-1}\|x_k - x^*_F\|^2 -0.5\gamma^{-1}\eta_k^{-1}\|x_{k+1} - x^*_F\|^2 \\
 &+0.5\gamma^{-1}\left(\eta_k^{-1}  - \eta_{k-1}^{-1}\right)\|x_k - x^*_F\|^2  .
\end{align*}
Note that $\far{\eta_k^{-1}  - \eta_{k-1}^{-1}\geq 0}$, because $\eta_k$ is \ssrtwo{diminishing}. We obtain
\begin{align*} 
 H(x^*_F)\fyy{^\top}\left(y_{k+1}-x^*_F\right)   & \leq 0.5\gamma^{-1}\eta_{k-1}^{-1}\|x_k - x^*_F\|^2 -0.5\gamma^{-1}\eta_k^{-1}\|x_{k+1} - x^*_F\|^2 \\
 &+ \gamma^{-1}\left(\eta_k^{-1}  - \eta_{k-1}^{-1}\right)D_X^2  .
\end{align*}
Summing both sides for $k=1,\ldots,K-1$, we obtain
 \begin{align*} 
 H(x^*_F)\fyy{^\top}\textstyle\sum_{k=1}^{K-1}\left(y_{k+1}- x^*_F\right)   & \leq 0.5\gamma^{-1}\eta_{0}^{-1}\|x_1 - x^*_F\|^2 -0.5\gamma^{-1}\eta_{K-1}^{-1}\|x_{K} - x^*_F\|^2 \\
 &+ \gamma^{-1}\left(\eta_{K-1}^{-1}  - \eta_{0}^{-1}\right)D_X^2.
\end{align*}
Substituting $k:=0$ in \eqref{ineq:thm1_separated_ineq1} and summing the resulting inequality with the preceding relation yields
  \begin{align*} 
 H(x^*_F)\fyy{^\top}\left(\textstyle\sum_{k=0}^{K-1}y_{k+1}-Kx^*_F\right)   & \leq 0.5\gamma^{-1}\eta_{0}^{-1}\|x_0 - x^*_F\|^2 -0.5\gamma^{-1}\eta_{K-1}^{-1}\|x_{K} - x^*_F\|^2 \\
 &+ \gamma^{-1}\left(\eta_{K-1}^{-1}  - \eta_{0}^{-1}\right)D_X^2.
\end{align*}
 Noting that $0.5\gamma^{-1}\eta_{0}^{-1}\|x_0 - x^*_F\|^2 \leq \gamma^{-1}\eta_{0}^{-1}D_X^2$ and dropping the nonpositive term $-0.5\gamma^{-1}\eta_{K-1}^{-1}\|x_{K} - x^*_F\|^2 $, we obtain 
\begin{align}\label{eqn:before_taking_sup}
 H(x^*_F)\fyy{^\top}\left(\bar{y}_{K}-x^*_F\right)   & \leq  (\gamma\eta_{K-1} K)^{-1}  D_X^2.
\end{align}
Taking the supremum on the both sides with respect to the set ${X}^*_F$ and invoking the definition of the dual gap function, we obtain $
 \mbox{Gap}\left(\bar{y}_K,X^*_F,H\right)      \leq  (\gamma\eta_{K-1} K)^{-1}  D_X^2$.

\noindent (ii) Next, we derive the bound on the infeasibility in the inner-level VI. Consider the inequality~\eqref{ineq:EG_lemma_merelymonotone}. From the Cauchy-Schwarz inequality, we obtain
\begin{align*}
F(x) \fyy{^\top}\left(y_{k+1}-x\right)   & \leq 0.5\gamma^{-1}\|x_k - x\|^2   -0.5\gamma^{-1}\|x_{k+1} - x\|^2 +\sqrt{2} \eta_kC_HD_X.
\end{align*}
Summing both sides for $k=0,\ldots,K-1$ and dividing both sides by $K$, we obtain 
\begin{align*} 
F(x) \fyy{^\top}\left(\bar{y}_K-x\right)   & \leq 0.5(\gamma K)^{-1}\left(\|x_0 - x\|^2   -\|x_{K} - x\|^2\right) +\sqrt{2} C_HD_X K^{-1}\textstyle\sum_{k=0}^{K-1}\eta_k .
\end{align*}
Dropping the nonpositive term and invoking the definition of $D_X$ again, we obtain 
\begin{align*} 
F(x) \fyy{^\top}\left(\bar{y}_K-x\right)   & \leq  (\gamma K)^{-1}D_X^2 +\sqrt{2} C_HD_X K^{-1}\textstyle\sum_{k=0}^{K-1}\eta_k .
\end{align*}
Taking the supremum on the both sides with respect to the set $X$ and invoking the definition of the dual gap function, we obtain the result in (ii). 

\noindent \far{(iii) Note that $\bar y_k \in X$ for all $k\geq 0$ and $X$ is compact. Invoking the Bolzano-Weierstrass theorem, $\{\bar y_k\}$ has  at least one accumulation point. Let us denote an arbitrary convergent subsequence of $\{\bar y_k\}$ by  $\{\bar y_{k_i}\}$ and let $\hat y$ denote the accumulation point. Recall that $\mbox{SOL}(X,F)$ is nonempty, compact, and convex (see Remark~\ref{rem:assum_monotone}) and $\mbox{Gap}\left(\bullet,\mbox{SOL}(X,F),H\right) $ is a continuous function~\cite[Chapter 2]{FacchineiPang2003}. By taking the limit along $\{\bar y_{k_i}\}$ on the relation in (i), we obtain $- B_H\, \mbox{dist}\left(\hat y,\mbox{SOL}(X,F)\right)  \leq \mbox{Gap}\left(\hat y,\mbox{SOL}(X,F),H\right)      \leq   0.$ 
Also, by taking the limit on the relation in (ii) we obtain $ \mbox{Gap}\left(\hat{y},X,F\right)=0 $. Invoking~\cite[Prop.~2.3.15]{FacchineiPang2003}, we have $\mbox{dist}\left(\hat{y},\mbox{SOL}(X,F)\right)=0$. Thus, we obtain $\mbox{Gap}\left(\hat y,\mbox{SOL}(X,F),H\right)  =0$ 
and again, in view of~\cite[Prop.~2.3.15]{FacchineiPang2003}, we conclude that $\hat y \in \mbox{SOL}\left(\mbox{SOL}\left(X,F\right),H\right).$}
\end{proof}
Next, we derive the convergence rate statements for the \ses{IR-EG$_{{\texttt{m,m}}}$} method. \far{The proof is presented in section~\ref{sec_app_thm:bilevelVI}.}
\begin{tcolorbox}[colback=blue!5!white,colframe=blue!55!black]
\far{\begin{theorem}[\ses{IR-EG$_{{\texttt{m,m}}}$}'s rate statements for bilevel VIs]\label{thm:bilevelVI}\em 
Consider problem~\eqref{eqn:bilevelVI}. Let Assumption~\ref{assum:bilevelVI_m} hold and $X$ be bounded. Assume that ${\gamma}^2(L_F^2+\eta_0^2L_H^2) \leq 0.5$.  

\medskip

{\bf [Case 1. Diminishing regularization]}
Suppose  $\eta_{k}=\frac{\eta_0}{(k+1)^{b}}$ where \ssrtwo{$b \in (0,1)$} is arbitrary. Then, the following results hold for all $K\geq 2^{1/(1-b)}$. 

\noindent {\bf (1-i)} $  \ssrtwo{- B_H\, \mbox{dist}\left(\bar{y}_K,\mbox{SOL}(X,F)\right) \leq } \mbox{Gap}\left(\bar{y}_K,\mbox{SOL}(X,F),H\right)      \leq  \left(\frac{ D_X^2}{\gamma\eta_0}\right)\frac{1}{{K}^{1-b}}$. 
 
\noindent {\bf (1-ii)} $ 0 \leq \mbox{Gap}\left(\bar{y}_K,X,F\right)       \leq   \left(\frac{D_X^2}{\gamma }\right)\frac{1}{K} +\left(\frac{\sqrt{2}\eta_0 C_HD_X}{1-b}\right)\frac{1}{K^b}$. 
 
 \noindent {\bf (1-iii)}  Further, if $\mbox{SOL}(X,F)$ is $\alpha$-weakly sharp \ses{of order $\mathcal{M}$}, then  
 
  $\mbox{Gap}\left(\bar{y}_K,\mbox{SOL}(X,F),H\right) \geq  -{B_H}\sqrt[\mathcal{M}]{ \alpha^{-1} \left( \left(\frac{D_X^2}{\gamma }\right) \frac{1}{K}+\left(\frac{\sqrt{2} C_HD_X \eta_0}{ 1-b}\right)\frac{1}{K^b}\right)}$.

\medskip

  {\bf [Case 2. Constant regularization with a priori known threshold \fyy{if} $\mathcal{M} =1$]} \far{\ses{Suppose $\mathcal{M}=1$}. Let $\eta$  be a constant regularization parameter that $\eta \leq \frac{\alpha}{2\|H(x^*)\|}$ where $x^*$ is an arbitrary solution to problem~\eqref{eqn:bilevelVI}.  Then, for all $K\geq 1$ the following results hold. 

 \noindent {\bf (2-i)} $\mbox{dist}(\bar{y}_K,\mbox{SOL}(X,F)) \leq \left(\frac{\|x_0-x^*\|^2}{\gamma \alpha}\right)\frac{1}{K}.$
  
 \noindent {\bf (2-ii)} $|\mbox{Gap}\left(\bar{y}_K,\mbox{SOL}(X,F),H\right)| \leq \max\left\{ \tfrac{D_X^2}{\gamma\eta}, \tfrac{ B_H\, \|x_0-x^*\|^2}{{ \gamma \alpha}  }\right\}\frac{1}{K}.$}

\end{theorem}}
 \end{tcolorbox}

\far{The proof of the next corollary is presented in section~\ref{sec:proof_lemma_cor:nested}.}
\begin{corollary}[\ses{ IR-EG$_{{\texttt{m,m}}}$}'s rate statements \far{for VI-constrained optimization}]\label{cor:nested}\em 
Consider Theorem~\ref{thm:bilevelVI}. In the special case when $H$ is the gradient map of a convex $L$-smooth function $f$, problem~\eqref{eqn:bilevelVI} captures problem~\eqref{eqn:optVI} and the results hold for $L_H:=L$, $C_H:=C_f$, \far{and $B_H:=B_f$}. In particular, the bounds in (1-i), \far{(1-iii), and (2-ii)} hold for the metric $f(\bar{y}_K) - \inf_{x \in {\tiny \mbox{SOL}(X,F)}} f(x)$.  
\end{corollary}
\begin{remark}
\far{By Theorem~\ref{thm:bilevelVI}, under the weak sharpness assumption and availability of the threshold $\tfrac{\alpha}{B^*_H}$, the iteration complexity of $\mathcal{O}(\tilde{\epsilon}^{-1})$ is achieved where $\tilde{\epsilon}$ is a scalar such that $|\mbox{Gap}\left(\bar{y}_K,\mbox{SOL}(X,F),H\right)| \leq \tilde{\epsilon}$. This is indeed the best known iteration complexity for addressing \far{monotone} Lipschitz continuous single-level VIs.  In \ses{the} absence of  the threshold, the} results in Theorem~\ref{thm:bilevelVI} imply that when $b:=0.5$, an iteration complexity of $\mathcal{O}(\epsilon^{-2})$ can be achieved for both the outer \far{(in terms of the upper bound)} and inner-level VIs. This improves the existing complexity for bilevel VIs. Also, for VI-constrained optimization, this improves the existing complexity $\mathcal{O}(\epsilon^{-4})$ in~\cite{doi:10.1137/20M1357378,kaushik2023incremental,jalilzadeh2024stochastic} by leveraging smoothness of $f$ and Lipschitz continuity of $F$. \far{We also note that when $\mathcal{M}>1$, the  complexity of $\mathcal{O}(\epsilon^{-2\ses{\mathcal{M}}})$ for the lower bound on the outer-level VI's gap function is novel and did not exist before.}
\end{remark}

\section{Strongly monotone bilevel VIs and VI-constrained strongly convex optimization}\label{sec:smon}
In this section, \far{our main assumptions are presented in the following.} 
\begin{assumption}\label{assum:bilevelVI_sm} \em
Consider problem~\eqref{eqn:bilevelVI}. Let the following statements hold. 

\noindent (a) Set $X$ is nonempty, closed, and convex.

\noindent (b) Mapping $F:X \to \mathbb{R}^n$ is $L_F$-Lipschitz and \far{monotone} on $X$.

\noindent  (c) Mapping $H:X \to \mathbb{R}^n$ is  $L_H$-Lipschitz and $\mu_H$-strongly monotone on $X$.

\noindent (d) The set $\mbox{SOL}(X,F)$ is nonempty.
\end{assumption}
\begin{remark}
Let us briefly comment on the existence and uniqueness of solutions to \eqref{eqn:bilevelVI} under the above assumptions. In view of Remark~\ref{rem:assum_monotone}, Assumption~\ref{assum:bilevelVI_sm} implies that $\mbox{SOL}(X,F)$ is nonempty, closed, and convex. Then, from the strong monotonicity of the mapping $H$, it follows that a solution to problem~\eqref{eqn:bilevelVI} exists and is unique. 
\end{remark}

\subsection{Algorithm outline}
\far{We develop \far{IR-EG$_{{\texttt{s,m}}}$},} presented by Algorithm~\ref{alg:IR-EG-s}, to address two classes of problems as follows. (i) The bilevel VI problem \eqref{eqn:bilevelVI} when $H$ is $\mu_H$-strongly monotone and Lipschitz continuous, and $F$ is \far{monotone} and Lipschitz continuous. (ii) The VI-constrained optimization problem~\eqref{eqn:optVI} when $f$ is $L$-smooth and strongly convex, by setting $H(x):=\nabla f(x)$. Note that \far{IR-EG$_{{\texttt{s,m}}}$} is a variant of \ses{IR-EG$_{{\texttt{m,m}}}$} \far{where we employ} a weighted averaging sequence where the weights, $\theta_k$, are updated geometrically characterized by the stepsize $\gamma$, regularization parameter \far{$\eta_k$}, and $\mu_H$. \far{In this section, we first derive explicit convergence guarantees for this method under diminishing update rules. We will then present the convergence guarantees when a constant regularization parameter is used.} 

\begin{algorithm}[H]
\caption{\far{IR-EG$_{{\texttt{s,m}}}$} for~\eqref{eqn:bilevelVI} with strongly monotone $H$ and \far{monotone} $F$}
\label{alg:IR-EG-s}
\begin{algorithmic}[1]
\STATE \textbf{input:} Initial vectors $x_0, \far{\bar {y}_0 }\in X$, a stepsize $\gamma > 0$, \far{regularization parameter sequence $\{\eta_k\}$}, such that {${\gamma}^2L_F^2+{\gamma}\far{\eta_k}\mu_H+{\gamma}^2\far{\eta_k^2}L_H^2\leq 0.5$},  $\theta_0 = \frac{1}{1- \gamma \far{\eta_0} \mu_H }$, and $ \Gamma_0 =0$.
\FOR{$k = 0, 1, \ldots, K-1$}

\STATE $y_{k+1}  := \Pi_X\left[x_k - {\gamma}\left(F(x_k)+\far{\eta_k} H(x_k)\right)\right]   \label{line:update_y_S H}$

\STATE $x_{k+1} := \Pi_X\left[x_k - {\gamma}\left(F(y_{k+1}) +\far{\eta_k} H(y_{k+1})\right)\right]  \label{line:update_x_S H}$

\STATE $\bar{y}_{k+1} := \frac{\Gamma_k \bar{y}_{k} + \far{\eta_k}\theta_{k} y_{k+1}}{\Gamma_{k+1}}$
 
\STATE $ \Gamma_{k+1} := \Gamma_k + \far{\eta_k}\theta_k $ and $\theta_{k+1} :=\frac{\theta_k}{1 - \gamma \far{\eta_{k+1}} \mu_H} $
\ENDFOR
\STATE {\bf return} $\bar{y}_K$
\end{algorithmic}
\end{algorithm} 

\subsection{Convergence analysis for \far{IR-EG$_{{\texttt{s,m}}}$}}
The \far{following} result provides a recursive relation that will be utilized in the rate analysis. \far{The proof of this result is presented in section~\ref{sec:proof_lemma_lem:s_monotone_lemma_ineq}.}
\begin{lemma}\label{lem:s_monotone_lemma_ineq}\em
Consider problem~\eqref{eqn:bilevelVI}. Let the sequence $\{\bar{y}_k\}$ be generated by Algorithm~\ref{alg:IR-EG-s}. Let Assumption~\ref{assum:bilevelVI_sm} hold and \far{$x\in X$ be an arbitrary vector}. 

\noindent (i) Let {${\gamma}^2L_F^2+{\gamma}\far{\eta_k}\mu_H+{\gamma}^2\far{\eta_k^2}L_H^2\leq 0.5$} \far{for all $k\geq 0$. Then,} for all $k\geq 0$, we have
\begin{align}\label{ineq:EG_lemma_smonotone}
2{\gamma}\left(F(x)+\far{\eta_k} H(x)\right)\fyy{^\top}\left(y_{k+1}-x\right)   & \leq (1-{\gamma}\far{\eta_k}\mu_H)\|x_k - x\|^2   -\|x_{k+1} - x\|^2.
\end{align}

\noindent (ii) Let $H(x) := \nabla f(x)$ where $f$ is a $\mu$-strongly convex and $L$-smooth function. Suppose ${\gamma}^2L_F^2+0.5{\gamma}\far{\eta_k}\mu+{\gamma}^2\far{\eta_k^2}L^2\leq 0.5$ \far{for all $k\geq 0$. Then,} for all $k\geq 0$, we have
\begin{align}\label{ineq:EG_lemma_smonotone2}
2{\gamma} F(x)\fyy{^\top}\left(y_{k+1}-x\right) +2{\gamma}\far{\eta_k}(f(y_{k+1})-f(x))  & \leq (1-\tfrac{\gamma\far{\eta_k}\mu}{2})\|x_k - x\|^2   -\|x_{k+1} - x\|^2.
\end{align}
\end{lemma}

In the following result, we show that the generated iterate by \far{IR-EG$_{{\texttt{s,m}}}$} is indeed a weighted average sequence. \far{The proof is presented in section~\ref{sec:proof_lemma_lem:ave_REG}.} 
\begin{lemma}[Weighted averaging in \far{IR-EG$_{{\texttt{s,m}}}$}]\label{lem:ave_REG}\em 
Let the sequence $\{\bar{y}_k\}$ be generated by Algorithm~\ref{alg:IR-EG-s} where $\theta_k \triangleq \far{\frac{1}{\prod_{t=0}^{k}(1- \gamma \eta_t \mu_H )}}$ for $k\geq 0$. Let us define the weights $\lambda_{k,K} \triangleq \frac{\far{\eta_k}\theta_k}{\sum_{j=0}^{K-1}\far{\eta_j}\theta_j}$ for $k \in \{0,\ldots, K-1\}$ and $K\geq 1$. Then, for any $K\geq 1$, we have $\bar{y}_{K} = \sum_{k=0}^{K-1} \lambda_{k,K} y_{k+1}$. Also, when $X$ is a convex set, we have $\bar y_K \in X$.
\end{lemma}
\far{Next, we derive error bounds and show asymptotic convergence for IR-EG$_{{\texttt{s,m}}}$.
 \begin{tcolorbox}[colback=blue!5!white,colframe=blue!55!black]
\far{\begin{theorem}[Error bounds and asymptotic convergence of IR-EG$_{{\texttt{s,m}}}$]\label{prop:unify_IR-EG_sm}\em
Consider problem~\eqref{eqn:bilevelVI}. Let the sequence $\{\bar{y}_k\}$ be generated by Algorithm~\ref{alg:IR-EG-s}, Assumption~\ref{assum:bilevelVI_sm} hold, and the set $X$ be bounded. Suppose $\{\eta_k\}$ is nonincreasing and ${\gamma}^2L_F^2+{\gamma}\eta_k\mu_H+{\gamma}^2\eta_k^2L_H^2\leq 0.5$ for $k\geq 0$.  

\noindent (i) For all $K\geq 1$, we have $$- B_H\, \mbox{dist}\left(\bar{y}_K,\mbox{SOL}(X,F)\right) \leq\mbox{Gap}\left(\bar{y}_K,\mbox{SOL}(X,F),H\right)      \leq   \gamma^{-1} D_X^2/ \textstyle\sum_{j=0}^{K-1} \eta_j\theta_j  .$$ 
 
 \noindent (ii) For all $K\geq 1$, we have $$ 0 \leq \mbox{Gap}\left(\bar{y}_K,X,F\right)       \leq \left({\gamma }^{-1}\eta_0 D_X^2 +  C_H \sqrt{2} D_X\textstyle\sum_{j=0}^{K-1} \eta_j^2\theta_j \right)/ \textstyle\sum_{j=0}^{K-1} \eta_j\theta_j .$$ 
 
\noindent (iii) Suppose $\sum_{j=0}^{\infty} \eta_j\theta_j = \infty$ and $\lim_{k\to \infty}\left(\sum_{j=0}^{k} \eta_j^2\theta_j  / \textstyle\sum_{j=0}^{k} \eta_j\theta_j\right)=0$. Then, $\{ \bar y_k\}$ converges to the unique solution of problem~\eqref{eqn:bilevelVI}.
\end{theorem}}
\end{tcolorbox}
\begin{proof}
\noindent (i)  The lower bound holds from Lemma~\ref{lem:bilevelVI_gap_ws}. Next, we show that the upper bound holds. Substituting $x:= x^*_F \in \mbox{SOL}(X,F)$ in \eqref{ineq:EG_lemma_smonotone}, for any $k\geq 0$ we have 
\begin{align*}
\left(F(x^*_F)+\eta_k H(x^*_F)\right)\fyy{^\top}\left(y_{k+1}-x^*_F\right) &\leq \frac{1}{2\gamma} (1-\gamma\eta_k\mu_H)\|x_k - x^*_F\|^2 - \frac{1}{2\gamma}\|x_{k+1} -  x^*_F\|^2.
\end{align*}
Since $x^*_F \in \mbox{SOL}(X,F)$ and $y_{k+1} \in X$, we have $F(x^*_F)\fyy{^\top}\left(y_{k+1}-x^*_F\right) \geq 0$. We obtain
\begin{align}\label{eqn:prop_unif_1}
\eta_k H(x^*_F)\fyy{^\top}\left(y_{k+1}-x^*_F\right) &\leq  \frac{1}{2\gamma} (1-\gamma\eta_k\mu_H)\|x_k - x^*_F\|^2 - \frac{1}{2\gamma}\|x_{k+1} -  x^*_F\|^2.
\end{align}
Consider \eqref{eqn:prop_unif_1} for $k=0$. Multiplying the both sides by $\theta_0=\far{\frac{1}{1- \gamma \eta_0 \mu_H }}$, we obtain 
\begin{align}\label{eqn:prop_unif_2}
 H(x^*_F)\fyy{^\top}\left(\eta_0\theta_0y_{1}-\eta_0\theta_0x^*_F\right) &\leq  \frac{1}{2\gamma}\|x_0 - x^*_F\|^2 - \frac{1}{2\gamma}\theta_0\|x_{1} -  x^*_F\|^2.
\end{align} 
If $K=1$, then taking the supremum on both sides with respect to $x^*_F \in\mbox{SOL}( X, F)$ and recalling the definition of dual gap function, we obtain (i). Suppose $K\geq 2$. Consider \eqref{eqn:prop_unif_1} for $k\geq 1$. Recall that $\theta_k \triangleq \far{\frac{1}{\prod_{t=0}^{k}(1- \gamma \eta_t \mu_H )}}$ for $k\geq 0$. Multiplying both sides of the preceding inequality by $\theta_k$, we obtain, for any $k\geq 1$,
\begin{align*}
 H(x^*_F)\fyy{^\top}\left(\eta_k\theta_k(y_{k+1}-x^*_F)\right) &\leq  \frac{1}{2\gamma} \left( \frac{\|x_k - x^*_F\|^2}{\prod_{t=0}^{k-1}(1- \gamma \eta_t \mu_H )}  -  \frac{\|x_{k+1} - x^*_F\|^2}{\prod_{t=0}^{k}(1- \gamma \eta_t \mu_H )}\right).
\end{align*}
Summing both sides of the preceding relation over $k=1,\ldots, K-1$, and then summing the resulting relation with \eqref{eqn:prop_unif_2}, we obtain 
\begin{align*}
H(x^*_F)\fyy{^\top}\left(\textstyle\sum_{k=0}^{K-1}\eta_k \theta_k(y_{k+1}-x^*_F)\right) &\leq  \frac{1}{2\gamma} \left(  \|x_0 - x^*_F\|^2  -  \theta_{K-1}\|x_{K} - x^*_F\|^2 \right).
\end{align*}
By dropping the nonpositive term and invoking Lemma~\ref{lem:ave_REG}, we obtain  
\begin{align*}
 & H(x^*_F)\fyy{^\top}(\bar y_{K} - x^*_F  ) \leq    \gamma^{-1} D_X^2/\left(\textstyle\sum_{j=0}^{K-1} \eta_j \theta_j\right),
\end{align*}
where we used the definition of $D_X$. Taking the supremum on both sides with respect to $x^*_F \in\mbox{SOL}( X, F)$ and recalling the definition of dual gap function, we obtain (i).

\noindent (ii) Consider \eqref{ineq:EG_lemma_smonotone}. For any $x \in X$, for $k\geq 0$, we have
\begin{align*}
2{\gamma}\left(F(x)+\eta_k H(x)\right)\fyy{^\top}\left(y_{k+1}-x\right)   & \leq (1-{\gamma}\eta_k \mu_H)\|x_k - x\|^2   -\|x_{k+1} - x\|^2.
\end{align*}
Using the Cauchy-Schwarz inequality, and invoking Definition~\ref{def:terms_bounded}, we obtain
\begin{align}\label{eqn:prop_unif_3}
 F(x) \fyy{^\top}\left(y_{k+1}-x\right)   
 & \leq   \frac{1}{2\gamma}\left((1-\gamma\eta_k\mu_H)\|x_k - x\|^2-  \|x_{k+1} -  \far{x}\|^2\right) +  \eta_k C_H \sqrt{2} D_X.
\end{align}
Consider \eqref{eqn:prop_unif_3} for $k=0$. Multiplying the both sides by $\eta_0\theta_0=\far{\frac{\eta_0}{1- \gamma \eta_0 \mu_H }}$, we obtain
 \begin{align}\label{eqn:prop_unif_4}
F(x) \fyy{^\top}\left( \eta_0\theta_0y_{1}- \eta_0\theta_0x\right)   
 & \leq  \frac{1}{2{\gamma}} \left(\eta_0\|x_0 - x\|^2   -\eta_0\theta_0\|x_{1} - x\|^2\right)+\eta_0^2\theta_0  C_H \sqrt{2} D_X\notag\\
 & \leq  \frac{1}{2{\gamma}} \left(\eta_0\|x_0 - x\|^2   -\eta_1\theta_0\|x_{1} - x\|^2\right)+\eta_0^2\theta_0  C_H \sqrt{2} D_X.
\end{align}
Multiplying both sides of \eqref{eqn:prop_unif_3} by $\eta_k\theta_k$ and using $\eta_k \geq \eta_{k+1}$, for any $k\geq 1$ we obtain 
\begin{align}\label{eqn:prop_unif_5}
 F(x) \fyy{^\top}\left(\eta_k\theta_k(y_{k+1}-x)\right)  
 &\leq 
  \frac{1}{2\gamma}\left(\eta_k\theta_{k-1}\|x_k - x\|^2-  \eta_{k+1}\theta_k\|x_{k+1} -  \far{x}\|^2\right) \notag\\
  &+  \eta_k^2\theta_k C_H \sqrt{2} D_X.
\end{align}
Summing both sides of the preceding relation over $k=1,\ldots, K-1$, and then summing the resulting relation with \eqref{eqn:prop_unif_4}, we obtain  
\begin{align*}
    F(x)\fyy{^\top}\textstyle\sum_{k=0}^{K-1}\eta_k\theta_k\left(y_{k+1} - x\right) &\leq \frac{1}{2\gamma} \left(  \eta_0\|x_0 - x  \|^2  - \eta_K \theta_{K-1}\|x_{K} - x \|^2 \right)\\
    &+ C_H \sqrt{2} D_X \textstyle\sum_{k=0}^{K-1} \eta_k^2\theta_k. 
\end{align*}
By dropping the nonpositive term and invoking Lemma~\ref{lem:ave_REG}, we obtain  
\begin{align*}
 & F(x )\fyy{^\top}(\bar y_{K} - x  ) \leq      \left({\gamma }^{-1}\eta_0 D_X^2  +  C_H \sqrt{2} D_X\textstyle\sum_{j=0}^{K-1} \eta_j^2\theta_j\right)/\left(\textstyle\sum_{j=0}^{K-1} \eta_j\theta_j\right),
\end{align*}
where we used the definition of $D_X$. Taking the supremum on both sides with respect to $x  \in X$, we obtain the bound in part (ii).


\noindent (iii) The proof can be done in a similar fashion to that of Theorem~\ref{Thm:asymptotic-m.m} (iii) and in view of the uniqueness of the solution of the bilevel VI. Hence, it is omitted.  
\end{proof}
}
We now present the error bounds for \far{IR-EG$_{{\texttt{s,m}}}$}. \far{The proof of this result is presented in section~\ref{sec_app_theorem: H is SC}.} 

 \begin{tcolorbox}[colback=blue!5!white,colframe=blue!55!black]
\begin{theorem}[\far{IR-EG$_{{\texttt{s,m}}}$}'s rate statements for bilevel VIs]\label{theorem: H is SC}\em 
\far{Consider problem~\eqref{eqn:bilevelVI}. Let $\bar{y}_K$ be generated by Algorithm~\ref{alg:IR-EG-s}, let Assumption~\ref{assum:bilevelVI_sm} hold, and let the set $X$ be bounded. Suppose $\gamma \leq \frac{1}{2L_F}$.

\medskip

{\bf [Case 1. Diminishing regularization]}
Suppose  $\eta_{k}=\frac{\eta_{0,u}}{k+\eta_{0,l}}$ for $k\geq 0$, where $\eta_{0,u}=(\gamma \mu_H)^{-1}$ and $\eta_{0,l}=5L_H\mu_H^{-1}$.  \ssrtwo{Then, $\{ \bar y_k\}$ converges to the unique solution of problem~\cref{eqn:bilevelVI}. Further} for all $K\geq 1$, 

\noindent {\bf (1-i)} $  \ssrtwo{- B_H\, \mbox{dist}\left(\bar{y}_K,\mbox{SOL}(X,F)\right) \leq } \mbox{Gap}\left(\bar{y}_K,\mbox{SOL}(X,F),H\right)      \leq      D_X^2(5L_H-\mu_H)  \frac{1}{K}$. 
 
 \noindent {\bf (1-ii)} The inner-level VI's gap function is bounded as 
 
 $ 0 \leq \mbox{Gap}\left(\bar{y}_K,X,F\right)       \leq \left(D_X^2 +\frac{\sqrt{2}C_HD_X}{\mu_H}\left(\frac{\mu_H}{5L_H}+\ln\left(\frac{K+5L_H\mu_H^{-1}-1}{5L_H\mu_H^{-1}}\right)\right)\right)\frac{\gamma^{-1}}{K}$. 
 
\noindent {\bf (1-iii)}  Further, if $\mbox{SOL}(X,F)$ is $\alpha$-weakly sharp  {of order $\mathcal{M}\geq 1$}, then  $\mbox{Gap}\left(\bar{y}_K,\mbox{SOL}(X,F),H\right)$ is bounded below by 

\noindent $-B_H \sqrt[\mathcal{M}]{ \alpha^{-1}\left(D_X^2 +\frac{\sqrt{2}C_HD_X}{\mu_H}\left(\frac{\mu_H}{5L_H}+\ln\left(\frac{K+5L_H\mu_H^{-1}-1}{5L_H\mu_H^{-1}}\right)\right)\right)\frac{\gamma^{-1}}{K}}.$ 

\medskip

{\bf [Case 2. Constant regularization]} Suppose $\eta := \frac{(p+1)\ln(K)}{ \gamma\mu_H K}$ for some arbitrary $p \geq 1$. Then, the following results hold for all integers $K$ such that $\frac{K}{\ln(K)} \geq 5(p+1)\frac{L_H}{\mu_H}$.

\noindent {\bf (2-i)} $  \ssrtwo{- B_H\, \mbox{dist}\left(\bar{y}_K,\mbox{SOL}(X,F)\right) \leq } \mbox{Gap}\left(\bar{y}_K,\mbox{SOL}(X,F),H\right)      \leq   \left(\frac{\mu_H D_X^2}{ p+1 }\right)\frac{1}{\ln(K)K^p}  $. 
 
 \noindent {\bf (2-ii)} $ 0 \leq \mbox{Gap}\left(\bar{y}_K,X,F\right)       \leq  \left(\frac{D_X^2}{\gamma}\right)\frac{1}{ K^{p+1}}+\left(\frac{(p+1)\sqrt{2} C_H  D_X}{ \gamma\mu_H } \right)\frac{\ln(K)}{K} $. 
 
 \noindent {\bf (2-iii)}  Further, if $\mbox{SOL}(X,F)$ is $\alpha$-weakly sharp {of order $\mathcal{M}\geq 1$}, then the error function  $\mbox{Gap}\left(\bar{y}_K,\mbox{SOL}(X,F),H\right)$ is bounded below by  
 
\noindent  $-B_H \sqrt[{\mathcal{M}}]{ \alpha^{-1} \left(\left( \frac{D_X^2}{\gamma} \right) \frac{1}{K^{p+1}}+ \left(\frac{\sqrt{2} C_H D_X (p+1)}{\gamma \mu_H}\right) \frac{\ln(K)}{K}\right) }.$

 \medskip
 
 {\bf [Case 3. Constant regularization with a priori known threshold \fyy{if} $\mathcal{M} =1$]}
 \noindent Let $\mbox{SOL}(X,F)$ be $\alpha$-weakly sharp of order $\mathcal{M}=1$. Suppose $\eta \leq \frac{\alpha}{2\|H(x^*)\|}$ and {${\gamma}^2L_F^2+{\gamma}\far{\eta}\mu_H+{\gamma}^2\far{\eta}^2L_H^2\leq 0.5$}, where $x^*$ is the unique solution to problem~\eqref{eqn:bilevelVI}.  Then, the following results hold for all $K\geq 1$. 
 
 \noindent {\bf (3-i)} $    \mbox{dist}(\bar{y}_K,X^*_F) \leq  \tfrac{  \|x_0 - x^*\|^2}{{\gamma}\alpha}(1-\gamma \eta \mu_H)^K.$
 
 \noindent {\bf (3-ii)}  $| \mbox{Gap}\left(\bar{y}_K, X^*_F,H\right)|   \leq \max\left\{ \tfrac{D_X^2}{ \gamma \eta},  \tfrac{B_H \|x_0 - x^*\|^2}{{\gamma}\alpha} \right\}(1-\gamma \eta \mu_H)^K.$
 }
\end{theorem}
   \end{tcolorbox}


We now provide the main results for addressing VI-constrained strongly optimization. \far{Due to some distinctions with the results in Theorem~\ref{theorem: H is SC}, we provide the detailed results as follows. These results will be utilized in section~\ref{sec:ncvx}. The proof of this result is presented in section~\ref{sec:proof_lemma_theorem: H is SC2}.}
 \begin{tcolorbox}[colback=blue!5!white,colframe=blue!55!black]
\begin{corollary}[IR-EG$_{{\texttt{s,m}}}$'s rate statements for VI-constrained strongly convex optimization]\label{theorem: H is SC2}\em 
\far{Consider problem \eqref{eqn:optVI}. Let Assumption~\ref{assum:bilevelVI_sm} hold, where $H$ is the gradient map of a $\mu$-strongly convex and $L$-smooth function $f$. Let $X$ be bounded and $x^*$ denote the unique optimal solution to \eqref{eqn:optVI}. Let $\bar{y}_K$ be generated by Algorithm~\ref{alg:IR-EG-s} with $\mu_H:=0.5\mu$ and $L_H:=L$.  Suppose $\gamma \leq \frac{1}{2L_F}$. 

 {\bf [Case 1. Diminishing regularization]}
Suppose  $\eta_{k}=\frac{\eta_{0,u}}{k+\eta_{0,l}}$ for $k\geq 0$, where $\eta_{0,u}=(0.5\gamma \mu)^{-1}$ and $\eta_{0,l}=10L\mu^{-1}$. Then, for all $K\geq 1$, 

\noindent {\bf (1-i)} $f(\bar{y}_K) -f(x^*)      \leq      \left(\frac{(5L-0.5\mu)\|x_0-x^*\|^2}{2}\right)  \frac{1}{K}$. 
 
 \noindent {\bf (1-ii)} The inner-level VI's gap function is bounded as 
 
 $ 0 \leq \mbox{Gap}\left(\bar{y}_K,X,F\right)       \leq \left(D_X^2 +\frac{2\sqrt{2}C_HD_X}{\mu}\left(\frac{\mu}{10L}+\ln\left(\frac{K+10L\mu^{-1}-1}{10L\mu^{-1}}\right)\right)\right)\frac{\gamma^{-1}}{K}$. 
 
\noindent {\bf (1-iii)}  Further, if $\mbox{SOL}(X,F)$ is $\alpha$-weakly sharp  {of order $\mathcal{M}\geq 1$}, then   
\begin{align*}
\tfrac{\mu}{2}\|\bar{y}_K-x^*\|^2 &\leq  \left(\tfrac{(5L-0.5\mu)\|x_0-x^*\|^2}{2}\right)  \tfrac{1}{K} \\ &+ \|\nabla f(x^*)\|\,\sqrt[\mathcal{M}]{\tfrac{\left(D_X^2 +\tfrac{2\sqrt{2}C_HD_X}{\mu}\left(\tfrac{\mu}{10L}+\ln\left(\tfrac{K+10L\mu^{-1}-1}{10L\mu^{-1}}\right)\right)\right)(\alpha\gamma)^{-1}}{K}}. 
\end{align*}

\medskip

{\bf [Case 2. Constant regularization]} Suppose $\eta := \frac{2(p+1)\ln(K)}{ \gamma\mu K}$ for some arbitrary $p \geq 1$. Then, the following results hold for all integers $K$ such that $\frac{K}{\ln(K)} \geq 10(p+1)\frac{L}{\mu}$.

\noindent {\bf (2-i)} $f(\bar{y}_K) -f(x^*)      \leq   \left(\frac{\mu \|x_0-x^*\|^2}{4( p+1) }\right)\frac{1}{\ln(K)K^p}  $. 
 
 \noindent {\bf (2-ii)} $ 0 \leq \mbox{Gap}\left(\bar{y}_K,X,F\right)       \leq  \left(\frac{D_X^2}{\gamma}\right)\frac{1}{ K^{p+1}}+\left(\frac{2(p+1)\sqrt{2} C_H  D_X}{ \gamma\mu} \right)\frac{\ln(K)}{K} $. 
 
 \noindent {\bf (2-iii)}  Further, if $\mbox{SOL}(X,F)$ is $\alpha$-weakly sharp {of order $\mathcal{M}\geq 1$}, then  
 \begin{align*}
\tfrac{\mu}{2}\|\bar{y}_K-x^*\|^2 &\leq  \left(\tfrac{\mu \|x_0-x^*\|^2}{4( p+1) }\right)\tfrac{1}{\ln(K)K^p} \\ &+ \|\nabla f(x^*)\|\,\sqrt[\mathcal{M}]{\left(\tfrac{D_X^2}{\alpha\gamma}\right)\tfrac{1}{ K^{p+1}}+\left(\tfrac{2(p+1)\sqrt{2} C_H  D_X}{ \alpha\gamma\mu} \right)\tfrac{\ln(K)}{K}}. 
\end{align*}

 \medskip
 
 {\bf [Case 3. Constant regularization with a priori known threshold \fyy{if} $\mathcal{M} =1$]}
 \noindent Let $\mbox{SOL}(X,F)$ be $\alpha$-weakly sharp of order $\mathcal{M}=1$. Suppose $\eta \leq \frac{\alpha}{2\|\nabla f(x^*)\|}$ and {${\gamma}^2L_F^2+0.5{\gamma}\far{\eta}\mu+{\gamma}^2\far{\eta}^2L^2\leq 0.5$}, where $x^*$ is the unique solution to problem~\eqref{eqn:bilevelVI}. Then, the following results hold for all $K\geq 1$. 
 
\noindent {\bf (3-i)} $    \mbox{dist}(\bar{y}_K,X^*_F) \leq  \tfrac{  \|x_0 - x^*\|^2}{{\gamma}\alpha}(1-0.5\gamma \eta \mu )^K.$
 

   \noindent \bf{(3-ii)} $\|\bar{y}_K-x^*\|^2 \leq \frac{2 \|x_0-x^*\|^2}{\mu \gamma \eta}(1- 0.5\gamma \eta \mu)^K.$
 
 }
\end{corollary}
   \end{tcolorbox}
\ssrtwo{\begin{remark}
Note that to define $\eta_{0,l}$ in \cref{theorem: H is SC}, it is not necessary to know the exact value of $L_H$, which may not be available. If the mapping $H$ is $L_H$-smooth, then it is also $\ell L_H$-smooth for any  $\ell \geq 1$. Therefore, in Case 1 of \cref{theorem: H is SC}, we can set $\eta_{0,l} = 5 \ell L_H \mu_H^{-1}$, and the results will still hold for $\ell L_H$. This implies that choosing $\eta_{0,l} \geq 5 L_H \mu_H^{-1}$ is sufficient to guarantee the results. Similarly, by the same reasoning, we conclude that \( \eta_{0,l} \geq 10 L \mu^{-1} \) is suitable in Case 1 of \cref{theorem: H is SC2}.
\end{remark}
}
 
\far{

\begin{remark}By Theorem~\ref{theorem: H is SC}, under the weak sharpness assumption with $\mathcal{M}=1$ and availability of the threshold $\tfrac{\alpha}{\|H(x^*)\|}$,  a linear convergence rate is guaranteed for bilevel VIs in  terms of  $| \mbox{Gap}\left(\bar{y}_K, X^*_F,H\right)| $.  Similar \far{guarantees are provided in Corollary~\ref{theorem: H is SC2} for VI-constrained  strongly convex optimization in terms of $\|\bar{y}_K-x^*\|^2$. }
\end{remark}}

\section{VI-constrained nonconvex optimization}\label{sec:ncvx}
Our goal in this section is to derive rate statements for Algorithm~\ref{alg:ncvx-IR-GD} in computing a stationary point to the VI-constrained nonconvex optimization problem~\eqref{eqn:optVI}.
\begin{assumption}\label{assum:bilevelVI_nc} \em
Consider problem~\eqref{eqn:optVI}. Let the following statements hold. 

\noindent (a) Set $X$ is nonempty, compact, and convex.

\noindent (b) Mapping $F:X \to \mathbb{R}^n$ is $L_F$-Lipschitz continuous and \far{monotone} on $X$.

\noindent  (c) Function $f:X \to \mathbb{R}$ is real-valued, possibly nonconvex, and $L$-smooth on $X$.

\noindent (d) The set $\mbox{SOL}(X,F)$ is $\alpha$-weakly sharp \ses{of order $\mathcal{M} \geq 1$}. 
\end{assumption}
\subsection{Algorithm outline}\label{sec:IPREG_desc}
 We propose the IPR-EG method, presented by Algorithm~\ref{alg:ncvx-IR-GD}, for solving~the VI-constrained optimization problem~\eqref{eqn:optVI} when $f$ is $L$-smooth and nonconvex. This is an inexactly-projected gradient method that employs IR-EG$_{{\texttt{s,m}}}$, at each iteration, with prescribed adaptive termination criterion and regularization parameters. \far{To elaborate on some key ideas employed in this method, consider} problem~\eqref{eqn:optVI}, rewritten as $\min_{x \in X^*_F} \ f(x)$ where $X^*_F\triangleq \mbox{SOL}(X,F)$. First, we may naively consider the standard projected gradient method as $\hat{x}_{k+1}:=\Pi_{X_F^*}\left[\hat{x}_k -\hat{\gamma}\nabla f(\hat{x}_k)\right]$. However, the set $X^*_F$ is unknown. Interestingly, at any iteration $k$, the \far{IR-EG$_{{\texttt{s,m}}}$} method can be employed to inexactly compute $\Pi_{X^*_F}[z_k]$ given  $z_k:=\hat{x}_k -\hat{\gamma}\nabla f(\hat{x}_k)$. To this end, let us assume that $k$ is fixed and define $H(x):=x-z_k$. Note that $H$ is strongly monotone and Lipschitz continuous with $\mu_H=L_H=1$. Observing that $H(x) =\nabla_x \left(\tfrac{1}{2}\|x-z_k\|^2\right)$, the unique solution to the bilevel VI problem~\eqref{eqn:bilevelVI} is equal to $\far{\Pi_{X^*_F}[z_k]} = \hbox{arg}\min_{x \in X^*_F} \ \tfrac{1}{2}\|x-z_k\|^2.$
Motivated by this observation, we employ \far{IR-EG$_{{\texttt{s,m}}}$} to compute $\far{\Pi_{X^*_F}[z_k]}$ inexactly. However, it is crucial to control this inexactness for establishing the convergence and deriving rate statements for computing a stationary point to problem~\eqref{eqn:optVI}. Indeed, we may obtain a bound on the inexactness, as it will be shown in Corollary~\ref{theorem: H is SC2}. A key question in the design of the IPR-EG method lies in finding out that at any given iteration $k$ of the underlying gradient  method, how many iterations of \far{IR-EG$_{{\texttt{s,m}}}$} are needed. Notably, because of performing inexact projections, $\hat{x}_k$ may not be feasible to problem~\eqref{eqn:optVI}. This infeasibility needs to be carefully treated in the analysis to establish the convergence to a stationary point of the original VI-constrained problem. 

\begin{algorithm}[H]
  \caption{IPR-EG for VI-constrained nonconvex optimization~\eqref{eqn:optVI}}
  \label{alg:ncvx-IR-GD}
  \begin{algorithmic}[1]
    \STATE \textbf{Input:} Initial vectors \far{$\hat x_{0}, x_{0,0}, \bar{y}_{0,0} \in X$}, outer loop stepsize $\hat{\gamma}:=\frac{1}{\sqrt{K}} \leq {\frac{1}{ 2L  }}$, and inner loop stepsize $ \gamma  \leq \frac{1}{2L_F}$.  \far{If $\mathcal{M}=1$ and $\frac{\alpha L}{2\sqrt{2}D_X L+  C_f}$ is known a priori, choose a constant regularization parameter $\eta_k \equiv \eta \leq    \frac{\alpha L}{2\sqrt{2}D_X L+  C_f}$ such that $0.5{\gamma}\eta+{\gamma}^2\eta^2\leq 0.25$ and let $T_k:=\tau\ln(k+1)$ where $\tau\geq \tfrac{-2}{\ln(1-0.5\eta\gamma)}$; otherwise,  let $\eta_k :=  {6\ln(T_k)} /( \gamma T_k)$ and \ses{$T_k := \max\{k^{1.5 \mathcal{M}},151\}$, for all $k\geq 0$}.}
   \FOR{$k=0,1,\ldots, K-1$}
     \STATE $z_k := \hat x_k - \hat{\gamma} \nabla f(\hat x_k)$
     \STATE  $ \Gamma_{k,0} :=0$, and $\theta_{k,0} := \frac{1}{1- 0.5\gamma \far{\eta_k} }$
 
      \FOR{$t = 0, 1, \ldots, T_k -1$}

	\STATE $y_{k, t+1}  := \Pi_X\left[x_{k,t} - {\gamma}\left(F(x_{k,t})+\eta_k \left( x_{k,t} - z_k\right)\right)\right]  $

	\STATE $x_{k, t+1} := \Pi_X\left[x_{k,t} - {\gamma}\left(F(y_{k, t+1}) +\eta_k \left( y_{k,t+1} - z_k\right)\right)\right] $

	\STATE $\bar{y}_{k, t+1} := \left( {\Gamma_{k,t} \bar{y}_{k,t} + \theta_{k,t} y_{k, t+1}}\right)/{\Gamma_{k, t+1}}$
 
	\STATE $ \Gamma_{k, t+1} := \Gamma_{k, t} + \theta_{k, t} $  and $\theta_{k, t+1} :=\frac{\theta_{k,t}}{1 - 0.5\gamma \eta_k } $
	\ENDFOR
         \STATE $\hat x_{k+1} := \bar{y}_{k,T_k } $, $\far{\bar{y}_{k+1,0}:=\bar{y}_{k,T_k}} $, and $\far{x_{k+1,0}:=\bar{y}_{k,T_k}} $
    \ENDFOR
    \STATE {\bf return} $\hat{x}_K$
  \end{algorithmic}
\end{algorithm}

\subsection{Convergence analysis for IPR-EG}
\far{Let us} define a metric for qualifying stationary points to problem~\eqref{eqn:optVI}.
\begin{definition}[Residual mapping]\label{def:res_map}\em
Let Assumption~\ref{assum:bilevelVI_nc} hold and $0<\hat{\gamma} \leq \frac{1}{L}$. For all $x \in X$, let $G_{1/\hat{\gamma}}:X \to \mathbb{R}^n$ be defined as 
$$G_{1/\hat{\gamma}}(x) \triangleq \tfrac{1}{\hat{\gamma}}\left(x-\Pi_{\footnotesize \mbox{SOL}(X,F)}\left[x-\hat{\gamma}\nabla f(x)\right]\right).$$
\end{definition}
\begin{remark}
Under Assumption~\ref{assum:bilevelVI_nc}, $\mbox{SOL}(X,F)$ is nonempty, compact, and convex. Also, $f^*\triangleq \inf_{x\, \in \, {\footnotesize \mbox{SOL}}(X,F)} f(x)>-\infty$. 
In view of Theorem~10.7 in~\cite{book:AmirB}, $x^* \in \mbox{SOL}(X,F)$ is a stationary point to problem~\eqref{eqn:optVI} if and only if $G_{1/\hat{\gamma}}(x^*)=0$. Utilizing this result, we consider $\|G_{1/\hat{\gamma}}(\bullet)\|^2$ as a well-defined error metric. 
\end{remark}
\begin{definition}\label{def:IPREG_terms}\em Let us define the following terms for $k\geq 0$. 
\begin{align*}
 & \hat{x}_{k+1} \triangleq  \bar y_{k,T_k}, \qquad z_k \triangleq \hat x_k - \hat{\gamma} \nabla f(\hat x_k),\\  
  & \delta _k\triangleq\hat x_{k+1}-\Pi_{X^*_F}[z_k] , \qquad e_k \triangleq \hat x_k - \Pi_{X^*_F}\left[\hat x_k\right], \qquad \hbox{where }X^*_F \triangleq \mbox{SOL}(X,F).
\end{align*}
\end{definition}
\begin{remark}\label{rem:e_k_sol_rel}
As explained earlier in subsection~\ref{sec:IPREG_desc}, at each iteration of IPR-EG, we employ \far{IR-EG$_{{\texttt{s,m}}}$} to compute an inexact but increasingly accurate projection onto the unknown set $X^*_F$. Accordingly, $\delta_k$ measures the inexactness of projection at iteration $k$ of the underlying gradient descent method. Also, $e_k$ measures the infeasibility of the generated iterate $\hat{x}_k$ at iteration $k$ with respect to the unknown constraint set $X^*_F$. Indeed, by definition, $\|e_k\| = \mbox{dist}(\hat{x}_k,\mbox{SOL}(X,F))$ for $k\geq 0$.
\end{remark}
To establish the convergence of the generated iterate by IPR-EG to a stationary point, the term $\|\hat{x}_{k+1}-\hat{x}_k\|^2$ may seem relevant. The following result builds a relation between this term and the residual mapping. \far{The proof is presented in section~\ref{sec:proof_lemma_subsection:error-metric}.}
\begin{lemma}\label{subsection:error-metric}\em
Let $\{{\hat x}_{k}\}$ be generated by Algorithm~\ref{alg:ncvx-IR-GD}. Then, for all $k \geq 0$,
\begin{align}\tfrac{\hat\gamma^2}{2}  \left\| G_{{1}/{\hat\gamma}}(\hat x_k) \right\|^2 \leq    \|\hat  x_{k+1} - \hat x_k\|^2 +  \| \delta_k \|^2.
\end{align}
\end{lemma}
In the following, we obtain an error bound on $\|G_{1/\hat{\gamma}}(\bullet)\|^2$ characterized by $\delta_k$ and $e_k$. \far{The proof is presented in section~\ref{sec:proof_lemma_prop:upper_bound_gradient}.} 
\begin{proposition}\label{prop:upper_bound_gradient}\em
Consider problem~\eqref{eqn:optVI}. Let Assumption~\ref{assum:bilevelVI_nc} hold and let $\{{\hat x}_{k}\}$ be generated by Algorithm~\ref{alg:ncvx-IR-GD}. Suppose  $\hat{\gamma} \leq \tfrac{1}{2L}$. Then, for all $k\geq 0$ we have 
\begin{align}\label{prop: uper G before sum}
  \| G_{{1}/{\hat\gamma}}(\hat x_k) \|^2  \leq \frac{4}{\hat\gamma} \left(f(\hat x_k) - f(\hat x_{k+1})\right)+  L\hat\gamma C_f^2  +\frac{80}{L\hat{\gamma}^3} ( \| \delta _k\|^2  + \| e_k \|^2).
\end{align}
\end{proposition}
In the next step, we derive bounds on the feasibility error term $e_k$ and the inexactness error term $\delta_k$. Of these, the bound on $e_k$ is obtained by directly invoking the rate results we developed for IR-EG$_{{\texttt{s,m}}}$. The bound on $\delta_k$, however, is less straightforward. \far{The proof of the next result is presented in section~\ref{sec_app_Upper bound for e_k}.} 
\begin{proposition} \label{Upper bound for e_k}\em
Consider problem~\eqref{eqn:optVI}. Let Assumption~\ref{assum:bilevelVI_nc} hold and let the sequence $\{{\hat x}_{k}\}$ be generated by Algorithm~\ref{alg:ncvx-IR-GD}. 

\noindent \far{[Case 1. $\mathcal{M}\geq  1$ and $\alpha$ is unavailable]} Then, for all $k\geq 1$, the following results hold. 
\begin{flalign*} 
\far{\hbox{(1-i)}}\ &\|e_k\| \leq \ses{
 \sqrt[\mathcal{M}]{ \left(\tfrac{25 D_X^2+17C_fD_X\hat{\gamma}}{\alpha\gamma}\right) \tfrac{\ln(T_{k})}{T_{k}} } 
} .\\
\far{\hbox{(1-ii)}}\ &\| \delta_k \|^2  \leq   \tfrac{2 D_X^2}{ 3  }  \tfrac{(\ln(T_{k}))^2}{T_{k}^2}+ 2\far{\sqrt[\mathcal{M}]{
 \left(\tfrac{25 D_X^2+17C_fD_X\hat{\gamma}}{\alpha\gamma}\right)^2 \tfrac{\ln(T_{k})^2}{T_{k}^2} 
} }\\
&+
 {2\hat{\gamma} C_f}\ses{\sqrt[\mathcal{M}]{
 \left(\tfrac{25 D_X^2+17C_fD_X\hat{\gamma}}{\alpha\gamma}\right) \tfrac{\ln(T_{k})}{T_{k}} 
} }.\end{flalign*}
\noindent \far{[Case 2. $\mathcal{M}=  1$ and $\alpha$ is available] Then, for all $k\geq 1$, the following results hold.
\begin{flalign*} 
\far{\hbox{(2-i)}}\     \|e_k\| \leq  \tfrac{ 2D_X^2}{{\gamma}\alpha}(1-0.5\gamma \eta  )^{T_{k-1}}.  \qquad 
\far{\hbox{(2-ii)}}\    \| \delta_k \|^2  \leq   \tfrac{4 D_X^2}{  \gamma \eta}(1- 0.5\gamma \eta  )^{T_k}. 
\end{flalign*}
 }
\end{proposition}
Next, we derive iteration complexity guarantees for the IPR-EG method in computing a stationary point to problem~\eqref{eqn:optVI} when the objective function is nonconvex.  
\begin{tcolorbox}[colback=blue!5!white,colframe=blue!55!black]
\begin{theorem}[IPR-EG's guarantees for VI-constrained nonconvex optimization]\label{theorem:NC}\em 
Consider problem~\eqref{eqn:optVI}. Let Assumption~\ref{assum:bilevelVI_nc} hold. Let the sequence $\{\hat {x}_k\}$ be generated by \cref{alg:ncvx-IR-GD}. Define $\hat x_K^*$ such that $ \| G_{{1}/{\hat\gamma}}(\hat x_K^*) \|^2= \min_{k=\lfloor {\frac{K}{2}}\rfloor, \ldots, K-1} \| G_{{1}/{\hat\gamma}}(\hat x_k) \|^2.$ Then, the following results hold.

\medskip 

\noindent {\bf [\far{Case 1. $\mathcal{M}\geq  1$ and $\alpha$ is unavailable}]}

\noindent {\bf\far{(1-i)}} [Infeasibility bound] For all $1 \leq k \leq K$, we have 
$$ \mbox{dist} (\hat x_k, \mbox{SOL}(X,F)) \leq \ses{
 \frac{\sqrt[\mathcal{M}]{ {25 D_X^2+17C_fD_X\hat{\gamma}}{(\alpha\gamma)^{-1}} }\ln(\max\{k^{1.5},\sqrt[\mathcal{M}]{151}\})}{k^{1.5}}
}.$$

\noindent {\bf\far{(1-ii)}} [Residual mapping bound] Let $\hat \gamma = \frac{1}{\sqrt{K}}$ where  {$K\geq \max\{2,4L^2\}$}. Then, 
\begin{align*}
    \| G_{{1}/{\hat\gamma}}(\hat x_K^*) \|^2  &\leq   \tfrac{\theta_1}{\sqrt {K}}+\tfrac{\theta_2 \fyy{{(\ln (\fyy{  \max\{K^{1.5 \mathcal{M}},151\}  }))} }^2}{K\sqrt {K}}   \\
    &+ \tfrac{\theta_3{\fyy{{(\ln (\fyy{  \max\{K^{1.5 \mathcal{M}},151\}  }))} }}^{\tfrac{2}{\mathcal{M}}} }{K\sqrt{K}} +  \tfrac{\theta_4{\fyy{{(\ln (\fyy{  \max\{K^{1.5 \mathcal{M}},151\}  }))} }}^{\tfrac{1}{\mathcal{M}}} }{\sqrt{K}},
\end{align*}
\ses{where $D_f \triangleq \sup_{x,y \in X} (f(x)-f(y))$, $\theta_1\triangleq  8D_f+2 L\, C_f^2 $, $\theta_2\triangleq  \tfrac{\fyy{1920} D_X^2}{L }    $, $\theta_3 \triangleq  \tfrac{8640}{L } \sqrt[\mathcal{M}]{{
\left(\tfrac{25 D_X^2+17C_fD_X\hat{\gamma}}{\alpha\gamma}\right)}^2 
}  $, and $\theta_4\triangleq   \tfrac{\fyy{960}\sqrt{3}}{L } { C_f}\sqrt[\mathcal{M}]{{
\left(\tfrac{25 D_X^2+17C_fD_X\hat{\gamma}}{\alpha\gamma}\right) 
}}  $.}

\noindent {\bf \far{(1-iii)}} [Overall iteration complexity] Let $\epsilon>0$ be an arbitrary scalar. To achieve $ \| G_{{1}/{\hat\gamma}}(\hat x_K^*) \|^2 \leq \epsilon$, the total number of projections onto the set $X$ is $ {\tilde{\mathcal{O}}}(\epsilon^{-\far{3\mathcal{M}-2}})$ where $\tilde{\mathcal{O}}$ ignores polylogarithmic and constant numerical factors.

\medskip 

\far{
\noindent {\bf [\far{Case 2. $\mathcal{M}=1$ and $\alpha$ is available]}}

\noindent {\bf \far{(2-i)}} [Infeasibility bound] For all $1 \leq k \leq K$, 
 $$ \mbox{dist} (\hat x_k, \mbox{SOL}(X,F)) \leq     \tfrac{{ 2D_X^2}{(\alpha\gamma)^{-1}}}{k^2}.$$ 

\noindent {\bf \far{(2-ii)}} [Residual mapping bound] Let $\hat \gamma = \frac{1}{\sqrt{K}}$ where $K\geq \max\{2,4L^2\}$. Then, 
\ses{\begin{align*}
    \| G_{{1}/{\hat\gamma}}(\hat x_K^*) \|^2  &\leq  \left(8D_f  +  2 L  C_f^2  +\tfrac{12160 D_X^2}{\gamma L}(\tfrac{D_X^2}{\gamma\alpha^2}+\tfrac{1}{\eta})\right)\tfrac{1}{\sqrt{K}} .
\end{align*}}

\noindent {\bf \far{(2-iii)}} [Overall iteration complexity] The total number of projections onto the set $X$ improves to $ {\tilde{\mathcal{O}}}(\epsilon^{-2})$.}

\end{theorem}
\end{tcolorbox}

\begin{proof}
\noindent \far{(1-i)} This result follows from \cref{Upper bound for e_k} and \cref{rem:e_k_sol_rel}.

\noindent \far{(1-ii)} Taking the sum on the both sides of \cref{prop: uper G before sum} for $k=\lfloor {\frac{K}{2}}\rfloor, \ldots, K-1$, we obtain 
\begin{align*} 
 \textstyle\sum_{k= \lfloor {\frac{K}{2}}\rfloor}^{K-1}  \| G_{{1}/{\hat\gamma}}(\hat x_k) \|^2  \leq \tfrac{4\left(f(\hat x_{\lfloor {\frac{K}{2}}\rfloor})- f(\hat x_{K})\right)}{\hat\gamma}  +  K L\hat\gamma C_f^2  +\frac{80}{L\hat{\gamma}^3}\sum_{k= \lfloor {\frac{K}{2}}\rfloor}^{K-1}  ( \| \delta _k\|^2  + \| e_k \|^2).
\end{align*}
From the definition of $\hat x_K^*$ and $D_f$ and  $\tfrac{K}{2} \leq K-\lfloor\tfrac{K}{2}\rfloor$, we obtain 
\begin{align}\label{eqn:thm_G_1}
\tfrac{K}{2}    \| G_{{1}/{\hat\gamma}}(\hat x_K^*) \|^2  \leq \tfrac{4D_f}{\hat\gamma} +  K L\hat\gamma C_f^2  +\tfrac{80}{L\hat{\gamma}^3} \textstyle\sum_{k= \lfloor {\frac{K}{2}}\rfloor}^{K-1}  ( \| \delta _k\|^2  + \| e_k \|^2).
\end{align}
From \cref{Upper bound for e_k}, we have
\begin{align}\label{eqn:thm_G_2} \| e_k \|^2+   \| \delta_k \|^2  &\leq 
 \tfrac{2 D_X^2}{ 3  }  \tfrac{(\ln(T_{k}))^2}{T_{k}^2}+ 3\far{\sqrt[\mathcal{M}]{
 \left(\tfrac{25 D_X^2+17C_fD_X\hat{\gamma}}{\alpha\gamma}\right)^2 \tfrac{\ln(T_{k})^2}{T_{k}^2} 
} }\\
&+
 {2\hat{\gamma} C_f}\ses{\sqrt[\mathcal{M}]{
 \left(\tfrac{25 D_X^2+17C_fD_X\hat{\gamma}}{\alpha\gamma}\right) \tfrac{\ln(T_{k})}{T_{k}} 
} }. \notag
\end{align}
Recall that {$T_k = \max\{k^{1.5 \mathcal{M}},151\}$}. Invoking \cite[Lemma~9]{yousefian2017smoothing} \ses{and using $\lfloor {K/2}\rfloor \geq K/3>K/4$ for $K\geq 2$, we obtain \cref{2M}, \cref{1M}, and \cref{eqn:thm_G_3} as follows.}
\ses{\begin{align}\label{2M}
& \textstyle\sum_{k= \lfloor {\tfrac{K}{2}}\rfloor}^{K-1}  \tfrac{(\ln(T_{k}))^{\tfrac{2}{\mathcal{M}}}}{T_{k}^{\tfrac{2}{\mathcal{M}}}}  \leq {(\ln (\fyy{  \max\{K^{1.5 \mathcal{M}},151\}  }))}^{\tfrac{2}{\mathcal{M}}}\textstyle\sum_{k= \lfloor {\tfrac{K}{2}}\rfloor}^{K-1} \tfrac{1}{{k^3}} \\ \notag 
& \leq  {(\ln (\fyy{  \max\{K^{1.5 \mathcal{M}},151\}  }))}^{\tfrac{2}{\mathcal{M}}} \left( \lfloor {\tfrac{K}{2}}\rfloor ^{-2}  + \tfrac{K^{-2} - \lfloor {\tfrac{K}{2}}\rfloor ^{-2}}{-2}\right)\\\notag& \leq  {(\ln (\fyy{  \max\{K^{1.5 \mathcal{M}},151\}  }))}^{\tfrac{2}{\mathcal{M}}} \tfrac{18}{K^2} 
.
\end{align}
}
\ses{\begin{align}\label{1M}
& \textstyle\sum_{k= \lfloor {\tfrac{K}{2}}\rfloor}^{K-1}  \tfrac{(\ln(T_{k}))^{\tfrac{1}{\mathcal{M}}}}{T_{k}^{\tfrac{1}{\mathcal{M}}}}  \leq {(\ln (\fyy{  \max\{K^{1.5 \mathcal{M}},151\}  }))}^{\tfrac{1}{\mathcal{M}}}\textstyle\sum_{k= \lfloor {\tfrac{K}{2}}\rfloor}^{K-1} \tfrac{1}{{k^{1.5}}} \\\notag 
& \leq  {(\ln (\fyy{  \max\{K^{1.5 \mathcal{M}},151\}  }))}^{\tfrac{1}{\mathcal{M}}} \left( \lfloor {\tfrac{K}{2}}\rfloor ^{-0.5}  + \tfrac{K^{-0.5} - \lfloor {\tfrac{K}{2}}\rfloor ^{-0.5}}{-0.5}\right)\\\notag& \leq  {(\ln (\fyy{  \max\{K^{1.5 \mathcal{M}},151\}  }))}^{\tfrac{1}{\mathcal{M}}} \tfrac{\fyy{3} \sqrt{3}}{\sqrt{K}}  .
\end{align}
}
\begin{align}\label{eqn:thm_G_3}
 \textstyle\sum_{k= \lfloor {\tfrac{K}{2}}\rfloor}^{K-1}  \tfrac{(\ln(T_{k}))^2}{T_{k}^2}  &\leq (\ln (\fyy{  \max\{K^{1.5 \mathcal{M}},151\}  }))^2 \textstyle\sum_{k= \lfloor {\tfrac{K}{2}}\rfloor}^{K-1}  \tfrac{1}{{k}^3} \\
& =(\ln (\fyy{  \max\{K^{1.5 \mathcal{M}},151\}  }))^2 \textstyle\sum_{k= \lfloor {\tfrac{K}{2}}\rfloor -1}^{K-2}  \tfrac{1}{{(k+1)}^3} \notag\\
& \leq  (\ln (\fyy{  \max\{K^{1.5 \mathcal{M}},151\}  }))^2 \left(\lfloor K/2\rfloor^{-3} +\lfloor{K/2}\rfloor^{-2}-K^{-2} \right) \notag\\
& \leq \fyy{2(\ln (\fyy{  \max\{K^{1.5 \mathcal{M}},151\}  }))^2} \lfloor {K/2}\rfloor^{-2} \leq \fyy{\tfrac{18(\ln (\fyy{  \max\{K^{1.5 \mathcal{M}},151\}  }))^2}{K^2}}.\notag
\end{align}
 \ses{Now, from \cref{eqn:thm_G_1}, \cref{eqn:thm_G_2}, \cref{2M}, \cref{1M}, and \cref{eqn:thm_G_3}, we obtain }
\begin{align*}
    \| G_{{1}/{\hat\gamma}}(\hat x_K^*) \|^2  & \leq \tfrac{8D_f}{\hat\gamma K} +   2 L\hat\gamma C_f^2  \ses{ +\tfrac{160}{L \hat{\gamma}^3} \left(\tfrac{2 D_X^2}{ 3  }\tfrac{\fyy{18(\ln (\fyy{  \max\{K^{1.5 \mathcal{M}},151\}  }))^2}}{K^3} \right) } 
\\ & \ses{ +\tfrac{480}{L \hat{\gamma}^3} \left( {
\left(\tfrac{25 D_X^2+17C_fD_X\hat{\gamma}}{\alpha\gamma}\right)
}^{\tfrac{2}{\mathcal{M}}}
{18} \tfrac{{(\ln (\fyy{  \max\{K^{1.5 \mathcal{M}},151\}  }))}^{\tfrac{2}{\mathcal{M}}}}{K^3}
 \right) } \\ &
 \ses{ +\tfrac{160}{L \hat{\gamma}^3} \left( {2\hat{\gamma} C_f}{
\left(\tfrac{25 D_X^2+17C_fD_X\hat{\gamma}}{\alpha\gamma}\right)}^{\tfrac{1}{\mathcal{M}}}  \fyy{3} \sqrt{3} \tfrac{\fyy{{(\ln (\fyy{  \max\{K^{1.5 \mathcal{M}},151\}  }))}^{\tfrac{1}{\mathcal{M}}} }}{K\sqrt{K}}\right) }.
\end{align*}
Now, substituting \ses{$\hat \gamma :=\tfrac{1}{ \sqrt{K}}$} for $K\geq \max\{4L^2,2\}$, we obtain 
 \ses{\begin{align}  \label{r.h.s.G}
    \| G_{{1}/{\hat\gamma}}(\hat x_K^*) \|^2  & \leq \tfrac{8D_f+2 L\, C_f^2}{\sqrt {K}}  
 { +\tfrac{\fyy{1920} D_X^2}{L } \left(\tfrac{\fyy{(\ln (\fyy{  \max\{K^{1.5 \mathcal{M}},151\}  }))^2}}{K\sqrt {K}} \right) } 
\\\notag & { +\left(\tfrac{8640}{L } {\left(\tfrac{25 D_X^2+17C_fD_X\hat{\gamma}}{\alpha\gamma}\right)}^{\tfrac{2}{\mathcal{M}}}\right)  \left(\tfrac{{(\ln (\fyy{  \max\{K^{1.5 \mathcal{M}},151\}  }))}^{\tfrac{2}{\mathcal{M}}}}{K\sqrt{K}}\right) } 
\\\notag & { +\left( \tfrac{\fyy{960}\sqrt{3}}{L } { C_f}{\left(\tfrac{25 D_X^2+17C_fD_X\hat{\gamma}}{\alpha\gamma}\right) }^{\tfrac{1}{\mathcal{M}}}\right)   \left(\tfrac{\fyy{{(\ln (\fyy{  \max\{K^{1.5 \mathcal{M}},151\}  }))}^{\tfrac{1}{\mathcal{M}}} } }{\sqrt{K}}\right) }.
\end{align}}

 
\noindent \far{(1-iii)} Note that from \far{(1-ii)}, we have
\ses{$  \| G_{{1}/{\hat\gamma}}(\hat x_K^*) \|^2 \le \omega(K)$, where $\omega(K)$ is the right-hand side bound in  \cref{r.h.s.G} and is $\mathcal{O} \left( \tfrac{\ln(K)}{\sqrt{K}}\right) =\tilde{\mathcal{O}} \left( \tfrac{1}{\sqrt{K}}\right)$.}
 Consider \cref{alg:ncvx-IR-GD}. At each iteration in the outer loop, $2T_k$, i.e, \ses{$\max\{2k^{1.5\mathcal{M}},302\}$}, projections are performed onto the set $X$. Invoking~\cite[Lemma~9]{yousefian2017smoothing}, we conclude that the total number of projections onto the set $X$ is \ses{$\tilde{\mathcal{O}}(\epsilon^{-3\mathcal{M}-2})$}. 

\noindent \far{(2-i) From \cref{Upper bound for e_k} (case 2), substituting $T_k:= \tau\ln(k+1)$ we have 
\begin{align}
\mbox{dist} (\hat x_k, \mbox{SOL}(X,F)) &  \leq  \tfrac{ 2D_X^2}{{\gamma}\alpha}(1-0.5\gamma \eta  )^{\tau\ln(k)} \notag \\
 &\ses{=     \tfrac{ 2D_X^2}{{\gamma}\alpha}\left((1-0.5\gamma \eta  )^{\tau/2} \exp(1)\right)^{\ln(k^2)} \tfrac{1}{k^2} }   \leq \tfrac{ 2D_X^2}{{\gamma}\alpha} \tfrac{1}{k^2}, \notag
\end{align}
where in the last inequality we used $\tau\geq \tfrac{-2}{\ln(1-0.5\eta\gamma)} $.

\noindent  (2-ii) From \cref{Upper bound for e_k}, in a similar vein to the proof of (2-i), we have
\ses{\begin{align*} \| e_k \|^2+   \| \delta_k \|^2  &\leq  \tfrac{ 4D_X^4}{{\gamma^2}\alpha^2}(1-0.5\gamma \eta  )^{2\tau\ln(k)}+\tfrac{4 D_X^2}{  \gamma \eta}(1- 0.5\gamma \eta  )^{\tau\ln(k+1)} \notag\\
& \leq \tfrac{ 4D_X^4}{{\gamma^2}\alpha^2} \tfrac{1}{k^4} +\tfrac{4 D_X^2}{  \gamma \eta} \tfrac{1}{{(k+1)}^2} \leq \tfrac{4 D_X^2}{\gamma}(\tfrac{D_X^2}{\gamma \alpha^2}+\tfrac{1}{\eta})\tfrac{1}{k^2}. 
\end{align*}}
Thus, from \cite[Lemma~9]{yousefian2017smoothing} we obtain 
\ses{$$\textstyle\sum_{k= \lfloor {\tfrac{K}{2}}\rfloor}^{K-1}  ( \| \delta _k\|^2  + \| e_k \|^2)   \leq \tfrac{76D_X^2}{\gamma}(\tfrac{D_X^2}{\gamma \alpha^2}+\tfrac{1}{\eta}) \tfrac{1}{K},$$ }
where we used $\lfloor {K/2}\rfloor \geq K/4$ for $K\geq 2$. Consider \eqref{eqn:thm_G_1}. From the preceding relation, we obtain $
    \| G_{{1}/{\hat\gamma}}(\hat x_K^*) \|^2  \leq \tfrac{8D_f}{\hat\gamma K} +  2 L\hat\gamma C_f^2  +\tfrac{160}{L\hat{\gamma}^3K}\left(\ses{\tfrac{76D_X^2}{\gamma}(\tfrac{D_X^2}{\gamma \alpha^2}+\tfrac{1}{\eta}) \tfrac{1}{K}}\right).$ Substituting $\hat{\gamma}=\tfrac{1}{\sqrt{K}}$ we obtain the bound. 
    
    \noindent \far{(2-iii)}  Consider \cref{alg:ncvx-IR-GD} under the setting of a constant regularization parameter. At each iteration in the outer loop, $2T_k$, i.e, $2\tau\ln(k+1)$, projections are performed onto the set $X$. Invoking~\cite[Lemma~9]{yousefian2017smoothing}, we conclude that the total number of projections onto the set $X$ is $\sum_{k=0}^{\mathcal{O}(\epsilon^{-2})}\tau\ln(k+1) = {\mathcal{O}}(\epsilon^{-2}\ln(\epsilon^{-1}))$. 
}
\end{proof}

\begin{remark}
Theorem~\ref{theorem:NC} provides new complexity guarantees in computing a stationary point to VI-constrained nonconvex optimization. To our knowledge, such guarantees for computing the optimal NE in nonconvex settings did not exist before.
\end{remark}
\section{Numerical experiments}\label{sec:num}
We provide preliminary empirical results to validate the performance of the proposed methods in computing optimal equilibria. 
\far{In subsection~\ref{sec:num1}, we} consider a simple illustrative example of a two-person zero-sum Nash game, provided in~\cite{jalilzadeh2024stochastic}, where the set $\mbox{SOL}(X,F)$ is explicitly available. \far{In subsection~\ref{sec:num2}, we consider the Nguyen and Dupuis traffic equilibrium problem~\cite{nguyen1984efficient} and compare the performance of the proposed methods with existing methods in~\cite{facchinei2014vi}.}

\subsection{\far{An illustrative example}}\label{sec:num1}
\far{Consider a two-person zero-sum game given as}\\
\noindent\begin{minipage}{.47\linewidth}
\begin{equation*}
\Bigg\{\begin{aligned}
 &\min_{x_1} \   f_1(x_1,x_2) \triangleq 20-0.1x_1x_2+x_1 \\
& \hbox{s.t.}  \   x_1 \in X_1 \triangleq \{x_1 \mid 11 \leq x_1\leq 60\}, 
\end{aligned}
\end{equation*}
\end{minipage} 
\begin{minipage}{.47\linewidth}
\begin{equation*}
\Bigg\{\begin{aligned}
 &\min_{x_2} \   f_2(x_1,x_2) \triangleq -20+0.1x_1x_2-x_1 \\
& \hbox{s.t.}  \  x_2 \in X_2 \triangleq \{x_2 \mid 10 \leq x_2\leq 50\}. 
\end{aligned}
\end{equation*}
\end{minipage}
\vspace{.1in}

Note that the above game can be formulated as $\mbox{VI}(X,F)$ where $F(x)\triangleq Ax+b$, with $A=[ 0,-0.1  ;0.1,0 ]$ and $b=[1;0]$, and the set $X\triangleq X_1\times X_2$. The mapping $F$ is \far{monotone}, in view of $u\fyy{^\top}Au=0$ for all $u \in \mathbb{R}^2$. Also, $F$ is $L_F$-Lipschitz with $L_F = \|A\|=0.1$, where $\|\bullet\|$ denotes the spectral norm. Further, this game has a spectrum of equilibria given by $\mbox{SOL}(X,F) = \{(x_1,x_2)\mid 11 \leq x_1 \leq 60, \ x_2=10\}$.
 
In the experiments, we seek the best and the worst NE with respect to the metric $\psi(x) \triangleq \frac{1}{2}\|x\|^2$. The \fyy{best equilibrium seeking} problem is $\min_{x \in {\tiny \mbox{SOL}(X,F)}} \ f(x)\triangleq \psi(x)$ with the unique best NE as $x^*_{{\tiny \mbox{B-NE}} }= (11,10)$,  while the \fyy{worst equilibrium seeking} problem is $\min_{x \in {\tiny \mbox{SOL}(X,F)}} \ f(x)\triangleq -\psi(x)$ with the unique worst NE as $x^*_{{\tiny \mbox{W-NE}}} = (60,10)$. This is the unique stationary point to $\min_{x \in {\tiny \mbox{SOL}(X,F)}} \ -\psi(x)$, in view of the condition $0 \in -\nabla \psi(x^*_{{\tiny \mbox{W-NE}}} ) + \mathcal{N}_{\tiny \mbox{SOL}(X,F)}(x^*_{{\tiny \mbox{W-NE}}})$, where $\mathcal{N}$ denotes the normal cone. 

\far{We implement the proposed methods in the convex, strongly convex, and nonconvex cases. We implement IR-EG$_{{\texttt{m,m}}}$ and IR-EG$_{{\texttt{s,m}}}$ for computing the best NE. For these methods, we use the same constant stepsize $\gamma:=1/(2\|A\|_{\mathcal{F}})$ where $\|\bullet\|_{\mathcal{F}}$ denotes the Frobenius norm. For IR-EG$_{{\texttt{m,m}}}$, we let $\eta_0=0.01$ and for IR-EG$_{{\texttt{s,m}}}$, we use the self-tuned constant regularization parameter in Corollary~\ref{theorem: H is SC2}. We implement IPR-EG for computing the worst NE where we follow the parameter choices outlined in Algorithm~\ref{alg:ncvx-IR-GD} in the setting when $\alpha$ and $\mathcal{M}$ are not available.}
\begin{figure}[htbp]
  \centering
  \begin{subfigure}[b]{0.45\textwidth}
    \centering
    \includegraphics[width=\textwidth, height=5cm]{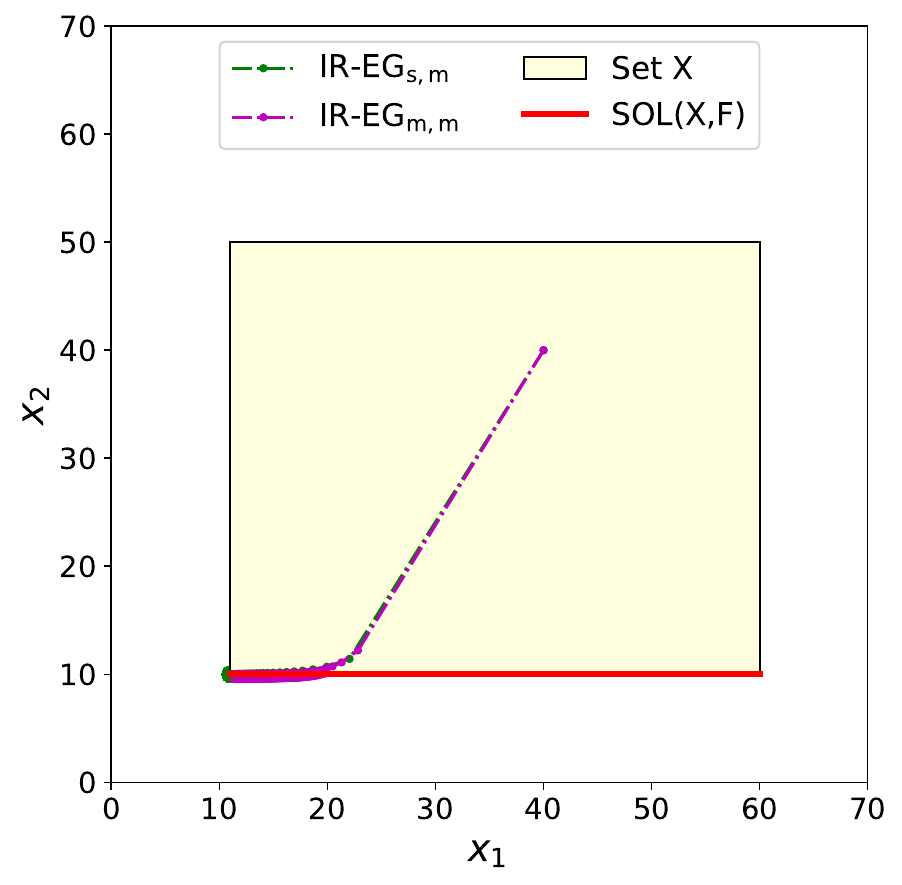}
  \end{subfigure}
  \hfill
  \begin{subfigure}[b]{0.45\textwidth}
    \centering
    \includegraphics[width=\textwidth, height=5cm]{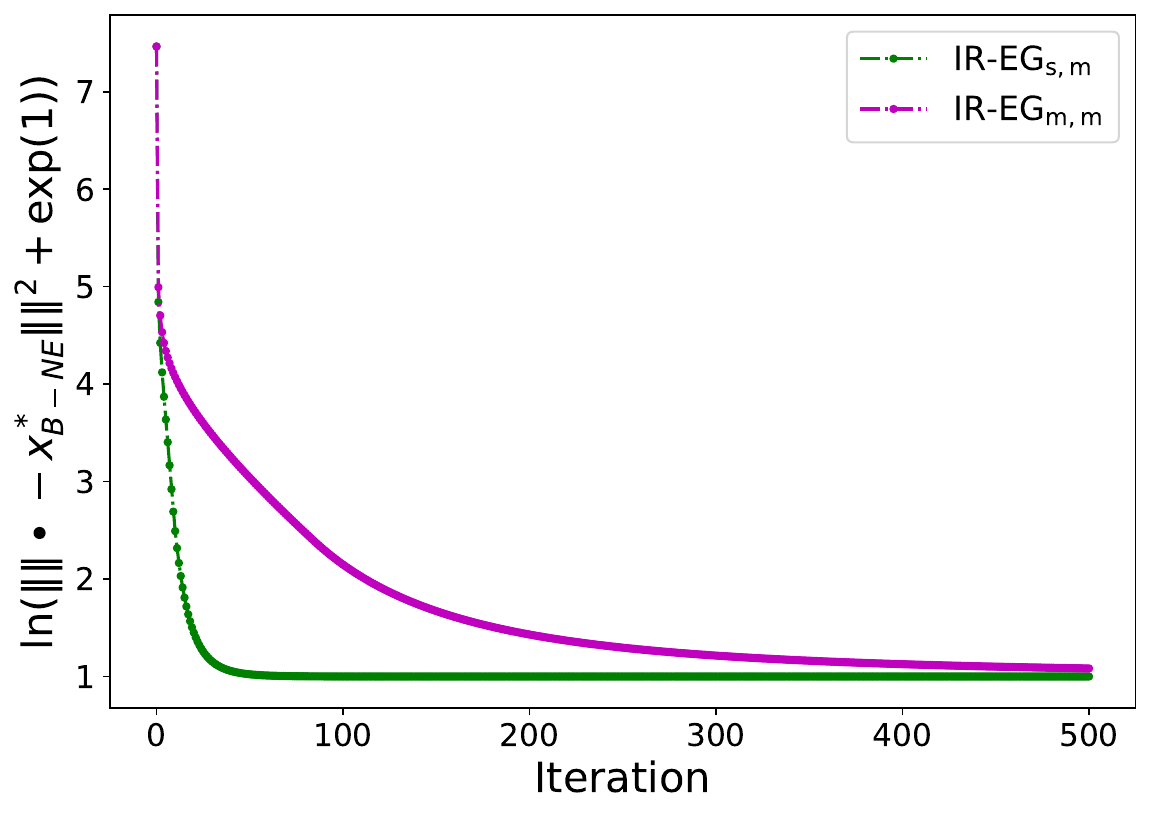}
  \end{subfigure}
  \vspace{-0.1in}
  \caption{Computing the best NE using IR-EG$_{{\texttt{m,m}}}$ and IR-EG$_{{\texttt{s,m}}}$}
  \label{fig:E1}
\end{figure}

\begin{figure}[htbp]
  \centering
  \begin{subfigure}[b]{0.45\textwidth}
    \centering
    \includegraphics[width=\textwidth, height=5cm]{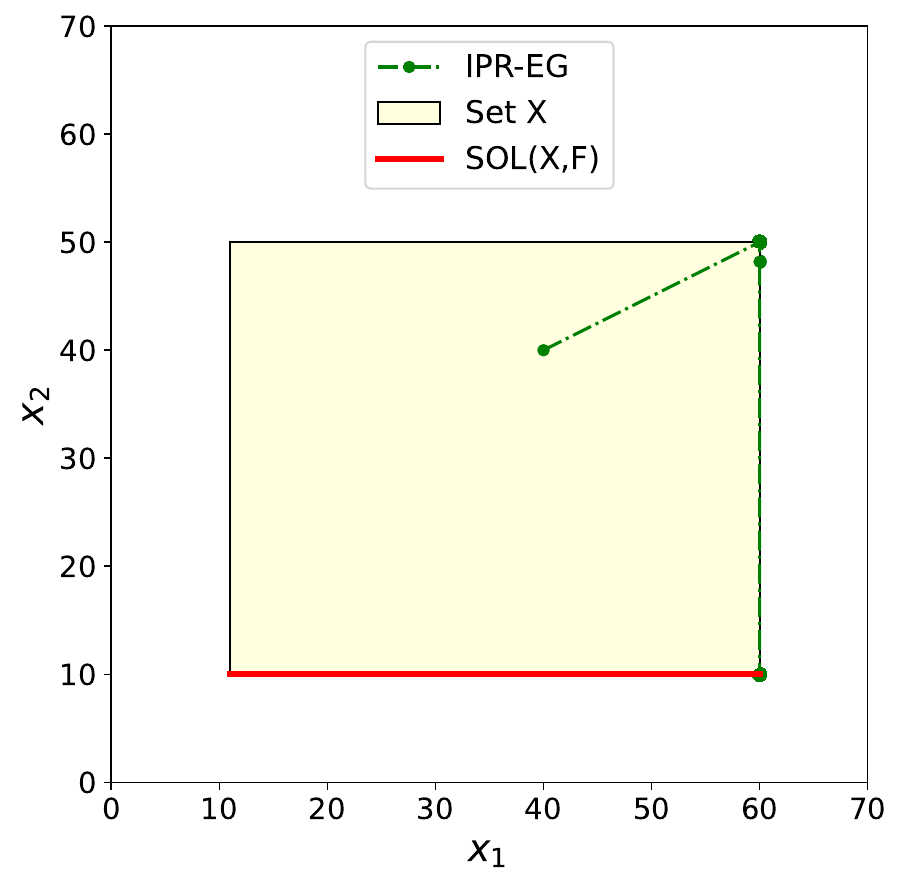}
  \end{subfigure}
  \hfill
  \begin{subfigure}[b]{0.45\textwidth}
    \centering
    \includegraphics[width=\textwidth, height=5cm]{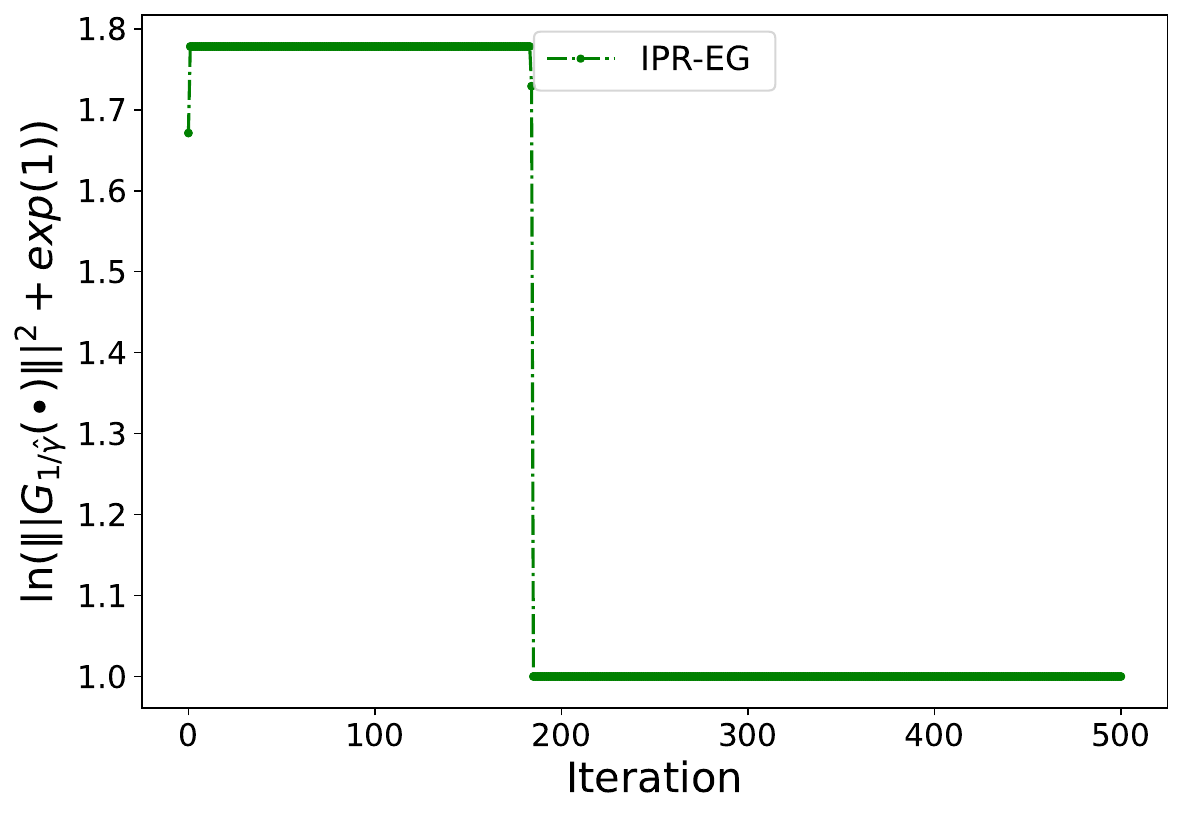}
  \end{subfigure}
  \vspace{-0.1 in}
  \caption{Computing the worst NE using IPR-EG}
  \label{fig:E2}
\end{figure}

 In computing the best NE in \ref{fig:E1}, IR-EG$_{{\texttt{s,m}}}$ outperforms IR-EG$_{{\texttt{m,m}}}$. This observation is indeed aligned with the accelerated rate statements we obtained for IR-EG$_{{\texttt{s,m}}}$. This is because unlike in IR-EG$_{{\texttt{s,m}}}$ where we utilize the strong convexity of the objective function $f$, IR-EG$_{{\texttt{m,m}}}$ is devised primarily for addressing \ssrtwo{ convex} objectives and it does not leverage the strong convexity property of $f$.  In computing the worst NE in \ref{fig:E2}, the sequence of iterates generated by IPR-EG appears to converge to the unique stationary point $x^*_{{\tiny \mbox{W-NE}}} = (60,10)$. This is indeed aligned with the theoretical guarantees in Theorem~\ref{theorem:NC}. 

\far{\subsection{Nguyen and Dupuis traffic equilibrium problem}\label{sec:num2}
We consider the Nguyen and Dupuis traffic equilibrium problem~\cite{nguyen1984efficient,chen2012stochastic}. The traffic flow network is shown in Figure~\ref{fig:traffic} containing 13 nodes, 19 arcs, 25 paths, and 4 origin-destination (OD) pairs between the origin nodes $\{n_1, n_4\}$ and destination nodes $\{n_2, n_3\}$. The arcs are indexed by $a=1,\ldots, 19$ and the paths are indexed by $p=1,\ldots,25$. Let $d=[d_1;\ldots; d_4]$ denote the travel demand vector between OD pairs where $d_1$, $d_2$, $d_3$, and $d_4$ are associated with demand between OD pairs $(n_1,n_2)$, $(n_1,n_3)$, $(n_4,n_2)$, and $(n_4,n_3)$, respectively. Next, we formulate this traffic network as a nonlinear complementarity problem (NCP) by adopting the path formulation~\cite[Ch. 1]{facchinei2003finite}. 
Let $h = [h_1; \ldots; h_{25}]$ and $ \mathcal{F}= [\mathcal{F}_1; \ldots; \mathcal{F}_{19}]$ denote the traffic flow vector associated with the paths and the arcs, respectively. Let matrix $\Delta$ denote the arc-path incidence matrix with entries $\delta_{a,p}$ where $\delta_{a,p}=1$ if path $p$ traverses arc $a$, and  $\delta_{a,p}=0$ otherwise. Let the arc travel time function be denoted by $c_a({\mathcal {F}}_a)$. }

\begin{wrapfigure}{r}{5.4cm}
\includegraphics[width=5.3cm]{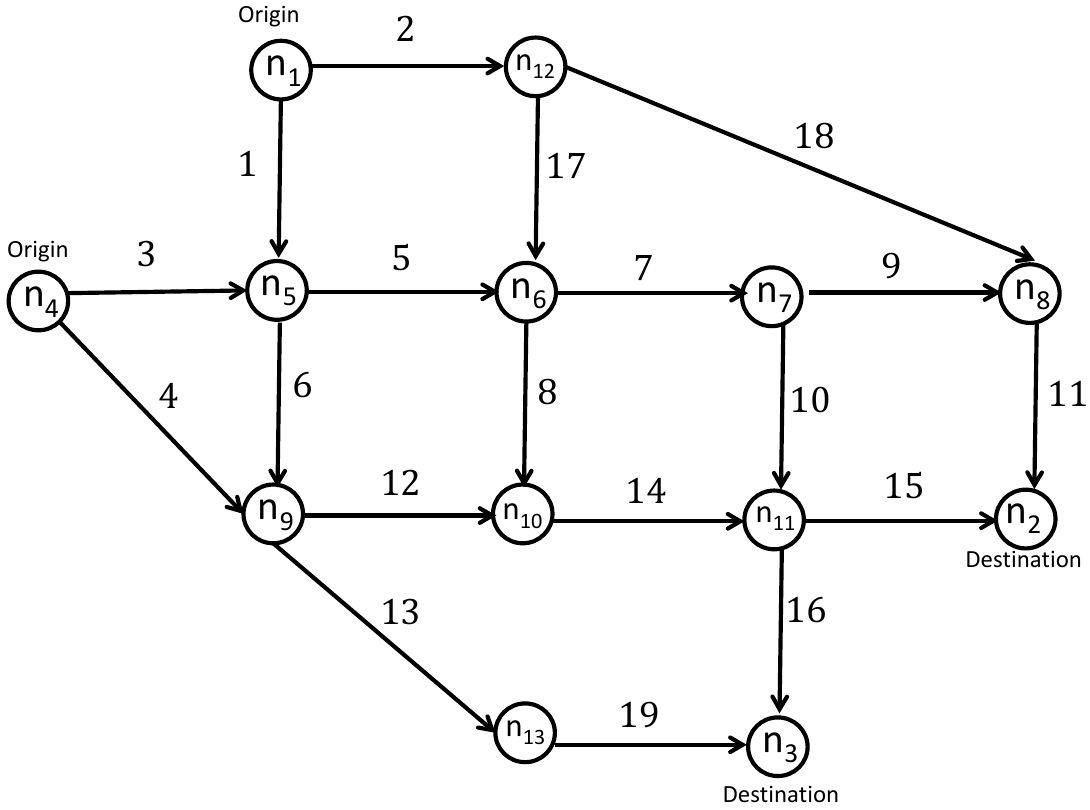}
 \caption{Nguyen and Dupuis network}\label{fig:traffic}
\vspace*{-2mm}
\end{wrapfigure}

\far{In the experiments, we consider the Bureau of Public Road arc travel time function~\cite{yin2009robust} given by $c_a({\mathcal {F}}_a) = t_a^0 (1+ 0.15 (\frac{{\mathcal {F}}_a}{cap_a})^{n_a})$  where $n_a \geq 1$, and $t_a^0$ and $cap_a$ denote free-flow travel time and capacity  of arc $a$, respectively. Note that $\mathcal {F} =\Delta h $ which implies that the path cost map $C:\mathbb{R}^{25} \to \mathbb{R}^{19}$ is given by $C(h) = \Delta \fyy{^\top} c(\Delta h)$~\cite[Ch. 1]{facchinei2003finite}, where $c:\mathbb{R}^{19} \to \mathbb{R}^{19}$ denotes the arc cost mapping given by $c(\mathcal{F}) = [c_1(\mathcal{F}_1);\ldots;c_{19}(\mathcal{F}_{19})]$. Let $u =[u_1; u_2; u_3;u_4 ] $ denote the unknown vector of minimum travel costs between OD pairs. Let $\Omega$ denote the (OD pair, path)-incidence matrix. Invoking the Wardrop user equilibrium principle, the traffic user equilibrium pair $x=[h;u]$ solves the $\mbox{NCP}(F)$, formulated succinctly as $0 \leq x \perp F(x) \geq 0$, where $F(x) = \begin{bmatrix} C(h)-\Omega\fyy{^\top}u; \Omega h-d\end{bmatrix} $. Note that when $n_a=1$ for all arcs, the equilibrium problem is cast as a linear complementarity problem $\mbox{LCP}(F)$. In the following, we show that when $n_a \geq 0$, the mapping $F$ is monotone. The proof of this result is presented in section~\ref{sec:proof_lemma_num-proof-mon-F}.
\begin{lemma}\label{num-proof-mon-F}\em 
Consider the aforementioned Nguyen and Dupuis traffic equilibrium problem with $n_a\geq 0$ for all $a$. Then, the mapping $F(x)$ is monotone on $\mathbb{R}^{29}_+$. 
\end{lemma}

In the experiments in this section, we consider computing the best and the worst traffic equilibrium with respect to the aggregated travel cost over the network given by $f(x) = \mathbf{1}_{25}\fyy{^\top} C(h) $, where $\mathbf{1}_{25}$ is a column vector with all $25$ entries equal to $1$. Next, we show that $f$ is convex. \far{The proof is presented in section~\ref{sec:proof_lemma_num-proof-cvx-f}.}
\begin{lemma}\label{num-proof-cvx-f} \em
Consider the aforementioned Nguyen and Dupuis traffic equilibrium problem with $n_a\geq 1$ for all $a$. The function $f(x) = \mathbf{1}_{25}\fyy{^\top} C(h) $ is convex. 
\end{lemma}
}


\far{
\subsection*{Experiments and setup} {\bf (E1)} In the first experiment, we let $n_a=1$ for all arcs and consider the minimization of the total travel cost function $f(x) = \mathbf{1}_{25}\fyy{^\top} C(h) $ over all traffic equilibria associated with $\mbox{VI}(\mathbb{R}^{29}_+,F)$ where $F(x) = \begin{bmatrix} C(h)-\Omega\fyy{^\top}u; \Omega h-d\end{bmatrix}$. This is a VI-constrained convex minimization. We implement Algorithm~\ref{alg:IR-EG} for solving this problem and compare its performance with Algorithm~1 in~\cite{facchinei2014vi}, which we refer to by ISR-cvx. The ISR-cvx method is suitable for addressing problem~\eqref{eqn:optVI} is essentially a two-loop regularized gradient scheme where at the $k$th iteration in the outer loop, the regularized $\mbox{VI}(X, F_k)$ is inexactly solved, where $\ssrtwo{F_k(\bullet)} = F(\bullet) + \eta_k  \nabla f (\bullet) + \tilde{\alpha} (\bullet - x_k)$ where $\tilde{\alpha}>0$ is a scalar. In the outer loop, the regularization parameter $\eta_k$ is updated, while in the inner loop, a gradient method is employed for computing and an $e_k$-inexact solution to the regularized VI. {\bf (E2)} In the second experiment, we repeat (E1) for the case when $n_a=1.2$. {\bf (E3)} Here, we assume that $n_a=1.2$ for all arcs. We are interested in computing the worst equilibrium by considering the maximization of the total travel cost function $f$ over $\mbox{SOL}(\mathbb{R}^{29}_+,F)$. We implement Algorithm~\ref{alg:ncvx-IR-GD} for solving this problem and compare its performance with Algorithm~3 in~\cite{facchinei2014vi}, which we refer to by ISR-ncvx. The ISR-ncvx method is a three-loop method and employs ISR-cvx at each iteration to inexactly compute a projection onto the set of equilibria. ISR-ncvx is characterized by two parameters $\delta$ and $\beta$ and leverages a line search technique to find a suitable stepsize at each iteration. In all the experiments, we let $d= [400, 800, 600, 450]$ and let the arc capacity be given by $$ cap = 100\times [7.5, 3.5, 1.5, 7.5, 3, 3.5, 7.5, 2, 2, 2.5, 5, 5, 5.5, 6.5, 4.5, 2.5,1.5, 3.5, 5.5].$$ We let $t^0 =[7, 9, 9, 12, 3, 9, 5, 13, 5, 9, 9, 10, 9, 5, 9,8,7,14,11 ] .$ We consider $\|z_{k+1} -z_k\|$ to measure suboptimality of the methods of interest, where $z_k$ denotes the iterate generated by the underlying method. We use the metric $\phi(x) = \| \max\{0, -x\}\|^2 + \| \max\{0, -F(x)\}\|^2 + |x\fyy{^\top} F(x)|$ to measure the infeasibility in the methods~\cite{kaushik2023incremental}. The ISR-cvx method is guaranteed to converge when $\sum_{k=0}^{\infty} \eta_k = \infty$ and $\left\{  {e_k}/{\eta_k} \right\} \to 0$. To meet these conditions, we use \ssrtwo{$\eta_k =\frac{\eta_0}{{(k+1)^{\tilde{b}}}}$, where $0 <\tilde{b} \leq 1 $} in ISR-cvx and terminate the inner-level gradient method after $T_k=k+1$ number of iterations.} 

\far{\subsection*{Results and insights} The implementation results for experiments (E1), (E2), and (E3) are presented in Figures~\ref{fig:cvx-lin-cap}, \ref{fig:cvx-non-cap}, and \ref{fig:ncvx-non-cap}, respectively. We make the following key observations. (i) In (E1), Algorithm~\ref{alg:IR-EG} appears to outperform ISR-cvx in all the 18 settings in terms of the suboptimality metric. Also, Algorithm~\ref{alg:IR-EG} displays a steady improvement in the infeasibility metric and reaches a smaller error than that of ISR-cvx in all the settings. (ii) In (E2), the comparison between the two methods stays akin to that in (E1). Notably, ISR-cvx shows more sensitivity to the change in $n_a$ from $1$ in (E1) to $1.2$ in (E2). (iii) In the nonconvex case studied in (E3), we observe that Algorithm~\ref{alg:ncvx-IR-GD}  outperforms ISR-ncvx in terms of the infeasibility metric. We do observe that Algorithm~\ref{alg:ncvx-IR-GD} performs worse in terms of the suboptimality metric. However, this is perhaps because ISR-ncvx appears to be relatively slow in making improvements in terms of the infeasibility metric, and thus, the consecutive generated iterates by this method, even far from being feasible, stay relatively close to each other, which results in a smaller suboptimality error. (iv) Lastly to display the distinctions between the performance of ISR-ncvx in (E3) with respect to the different settings, we suppress Algorithm~\ref{alg:ncvx-IR-GD} in Figure~\ref{fig:ncvx-non-cap} and present the performance of ISR-ncvx separately in Figure~\ref{fig:ncvx-non-cap-isr}.
\begin{table*}[t]
	\renewcommand\thetable{2}
	\setlength{\tabcolsep}{0pt}
	\centering
	\begin{tabular}{c |  c  c  c c}
			\rotatebox[origin=c]{90}{{\scriptsize{ $\log$(suboptimality metric)}}}
			&
			\begin{minipage}{.3\textwidth}
				\includegraphics[scale=.197, angle=0]{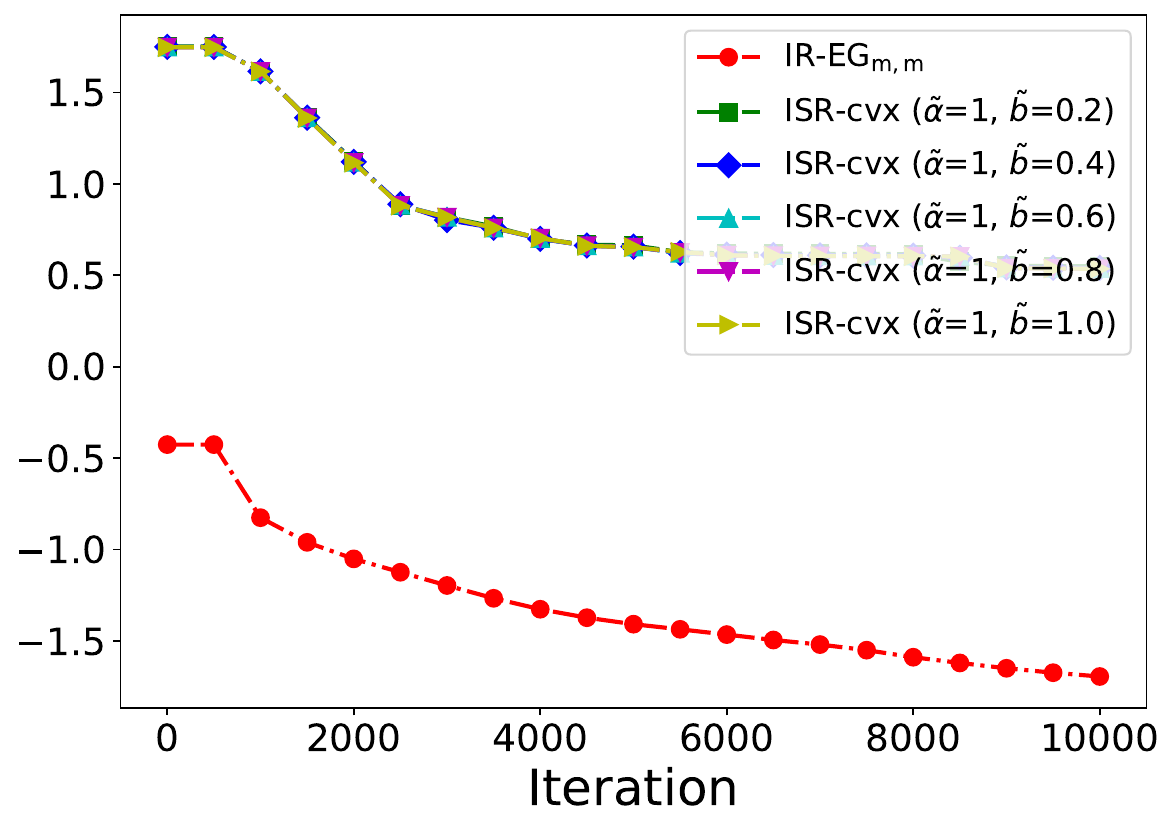}
			\end{minipage}
			&
			\begin{minipage}{.3\textwidth}
				\includegraphics[scale=.197, angle=0]{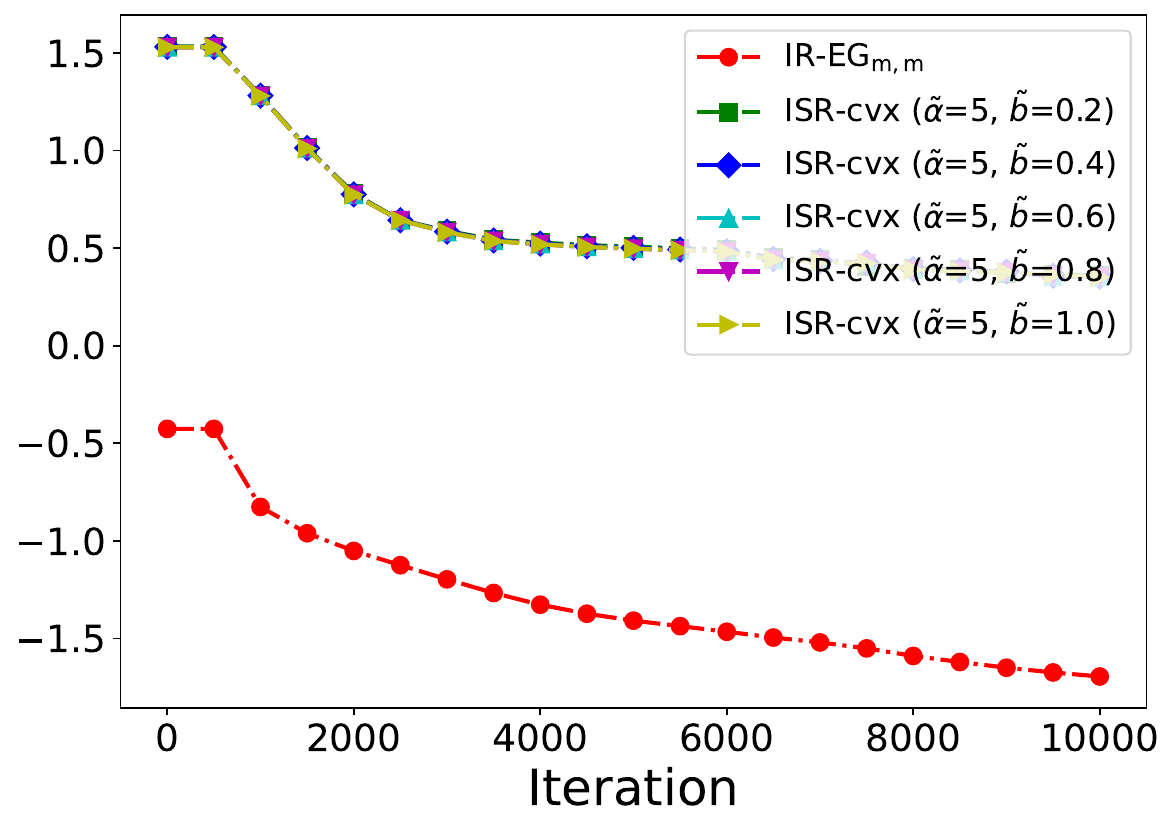}
			\end{minipage}
			&
			\begin{minipage}{.3\textwidth}
				\includegraphics[scale=.197, angle=0]{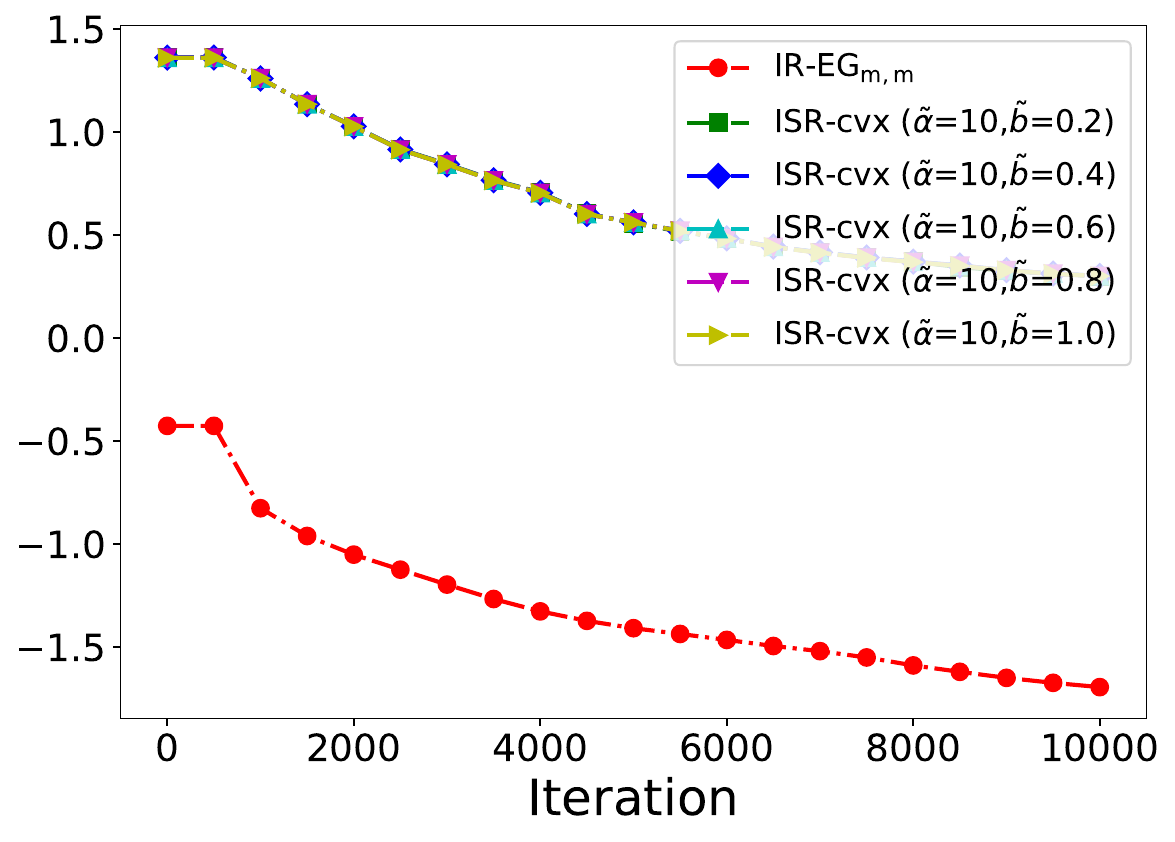}
			\end{minipage}
			\\
			\hline\\
			\rotatebox[origin=c]{90}{{\scriptsize{ $\log$(infeasibility metric)}}}
			&
			\begin{minipage}{.3\textwidth}
				\includegraphics[scale=.2, angle=0]{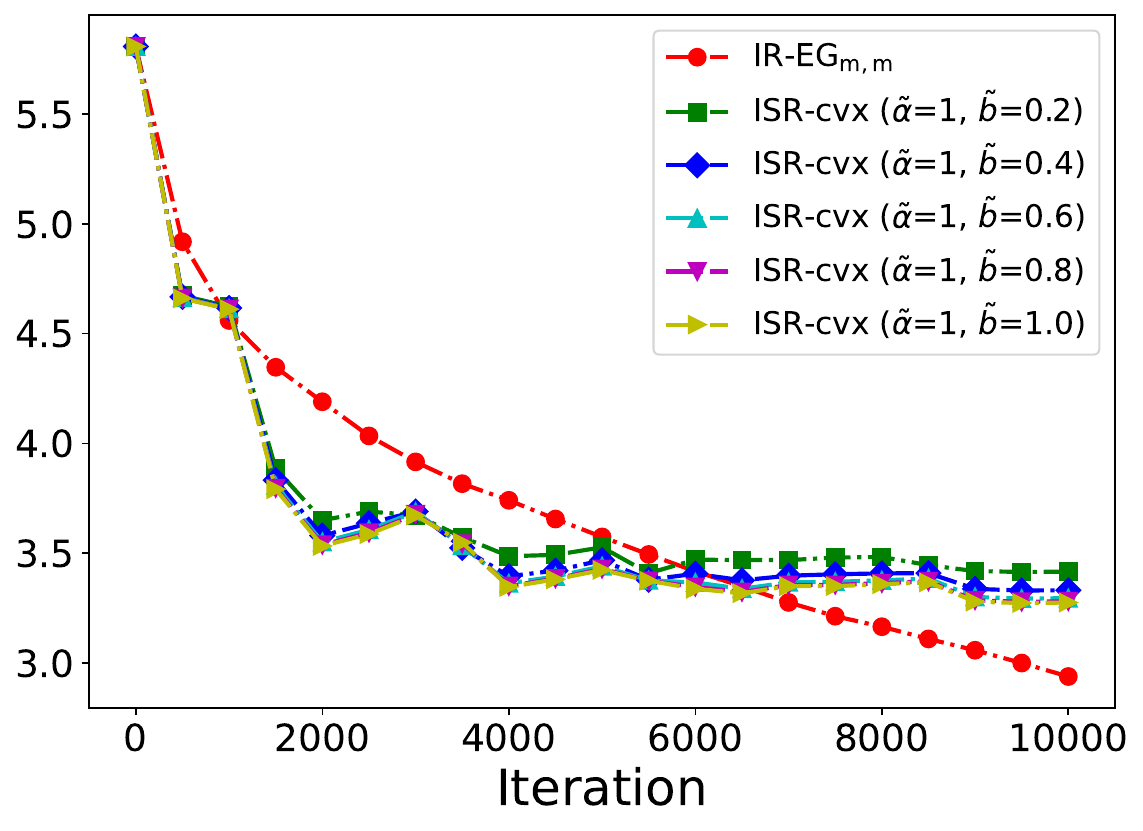}
			\end{minipage}
			&
			\begin{minipage}{.3\textwidth}
				\includegraphics[scale=.2, angle=0]{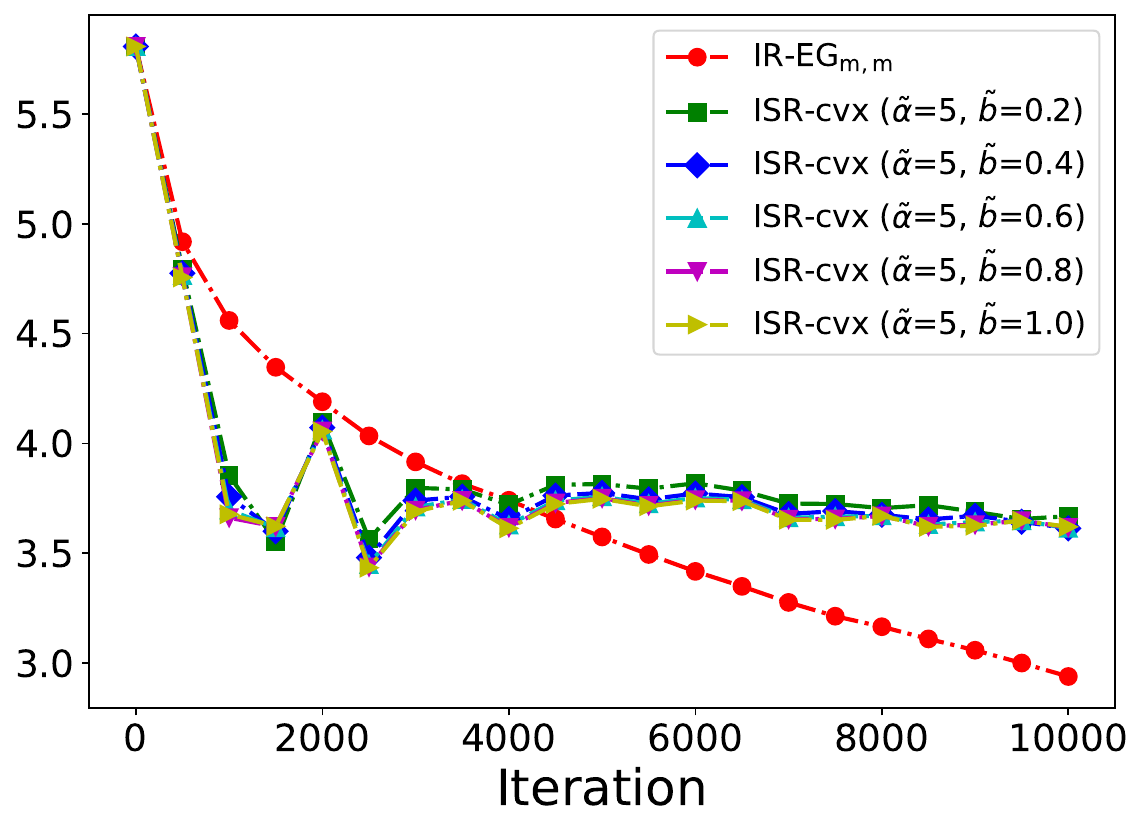}
			\end{minipage}
			&
			\begin{minipage}{.3\textwidth}
				\includegraphics[scale=.2, angle=0]{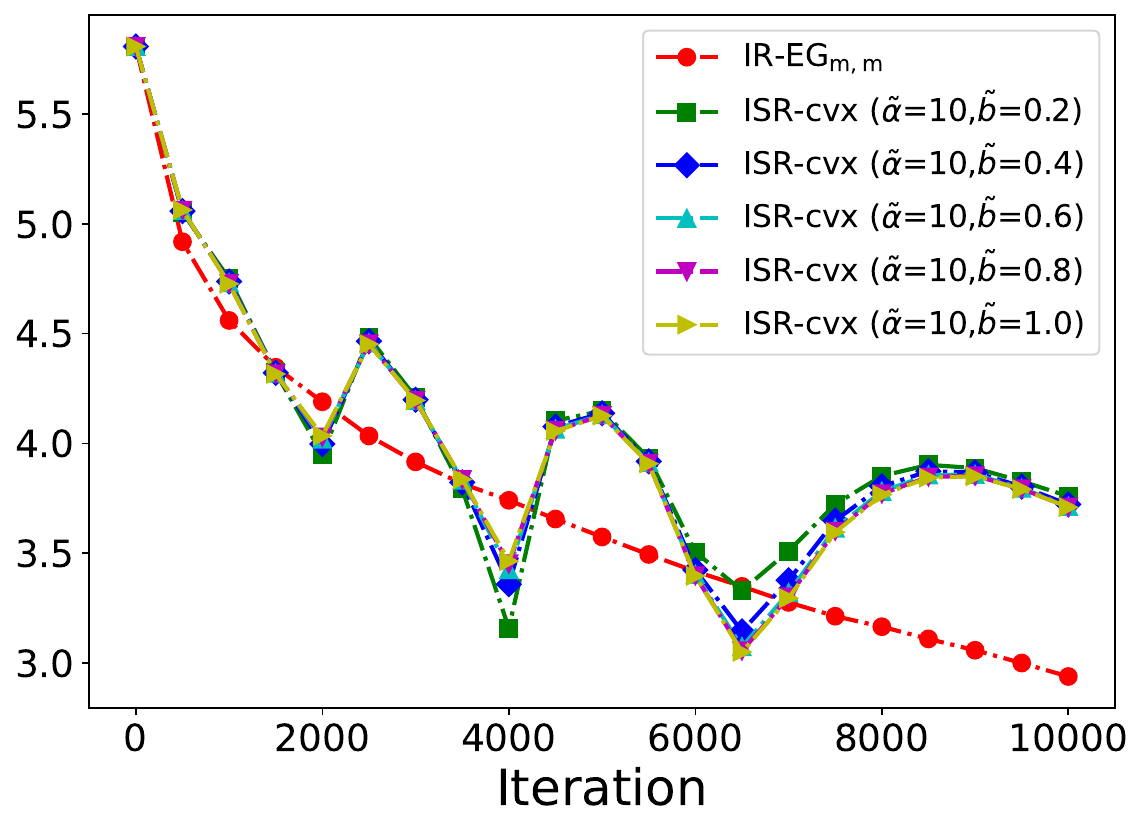}
			\end{minipage}
	\end{tabular}
	\captionof{figure}{Algorithm~\ref{alg:IR-EG} vs. ISR-cvx for computing the best NE when $n_a=1$ in all arcs.}\label{fig:cvx-lin-cap}
\end{table*}

\begin{table*}[t]
	\renewcommand\thetable{3}
	\setlength{\tabcolsep}{0pt}
	\centering{
		\begin{tabular}{c|  c  c  c c}
			\rotatebox[origin=c]{90}{{\scriptsize { $\log$(suboptimality metric)}}}
			&
			\begin{minipage}{.3\textwidth}
				\includegraphics[scale=.197, angle=0]{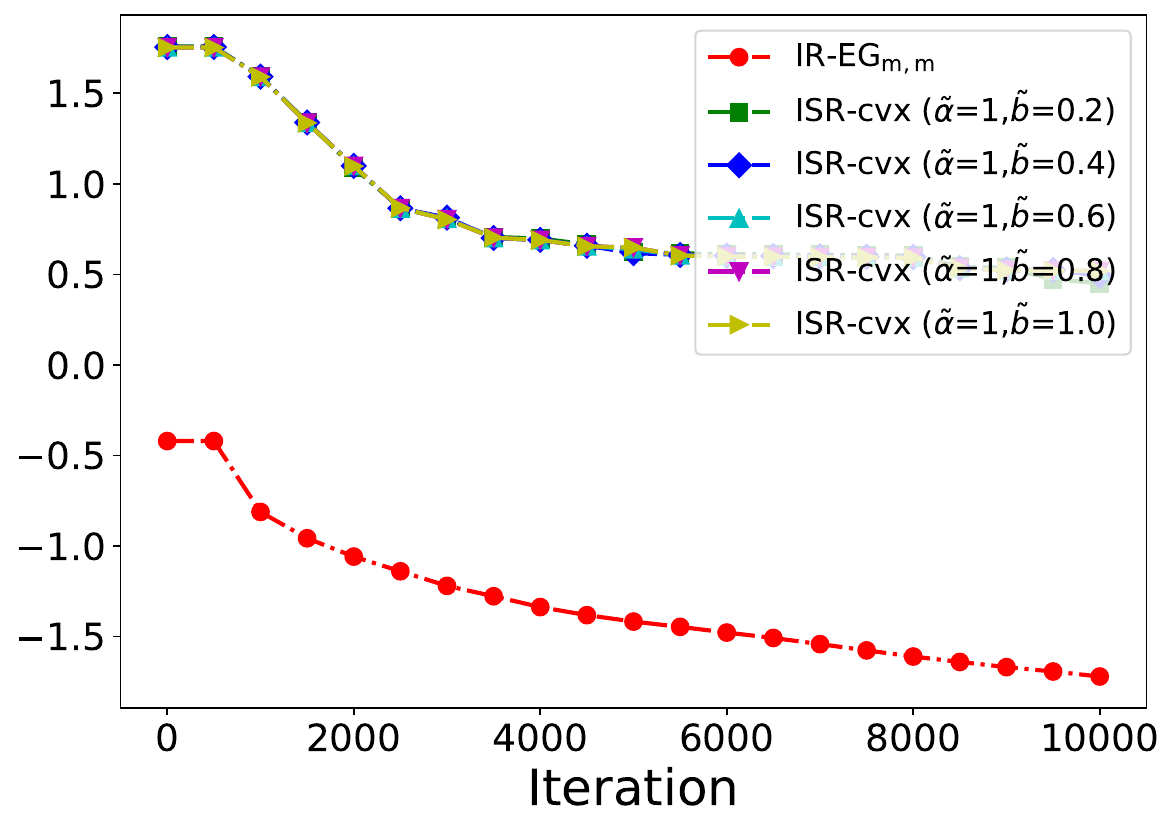}
			\end{minipage}
			&
			\begin{minipage}{.3\textwidth}
				\includegraphics[scale=.197, angle=0]{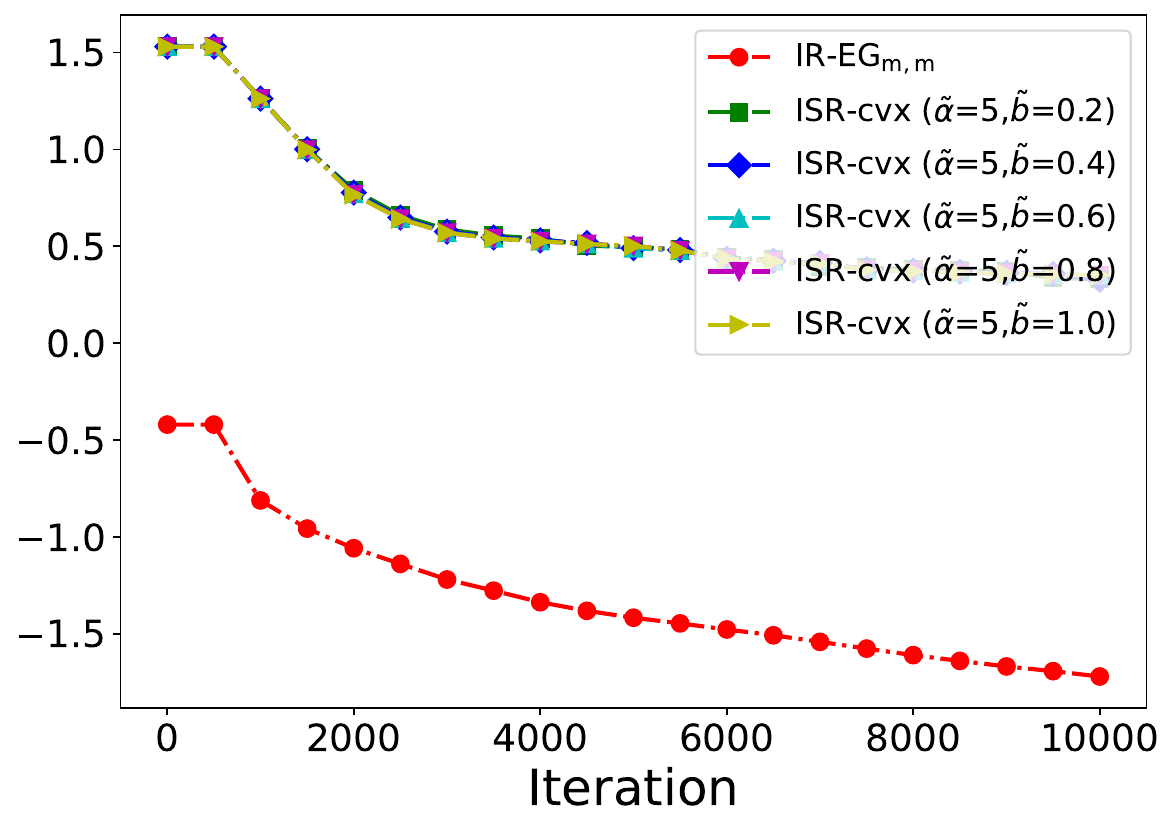}
			\end{minipage}
			&
			\begin{minipage}{.3\textwidth}
				\includegraphics[scale=.197, angle=0]{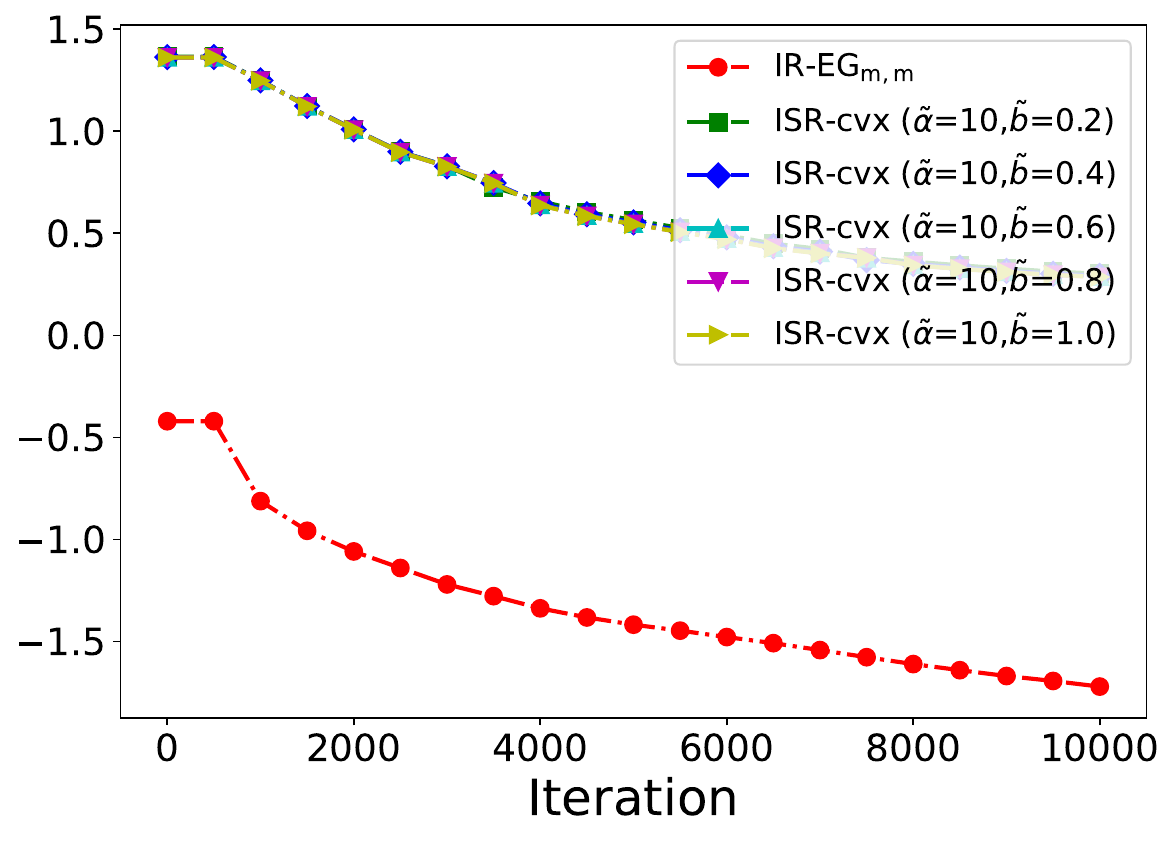}
			\end{minipage}
			\\
			\hline\\
			\rotatebox[origin=c]{90}{{ \scriptsize{ $\log$(infeasibility metric)}}}
			&
			\begin{minipage}{.3\textwidth}
				\includegraphics[scale=.2, angle=0]{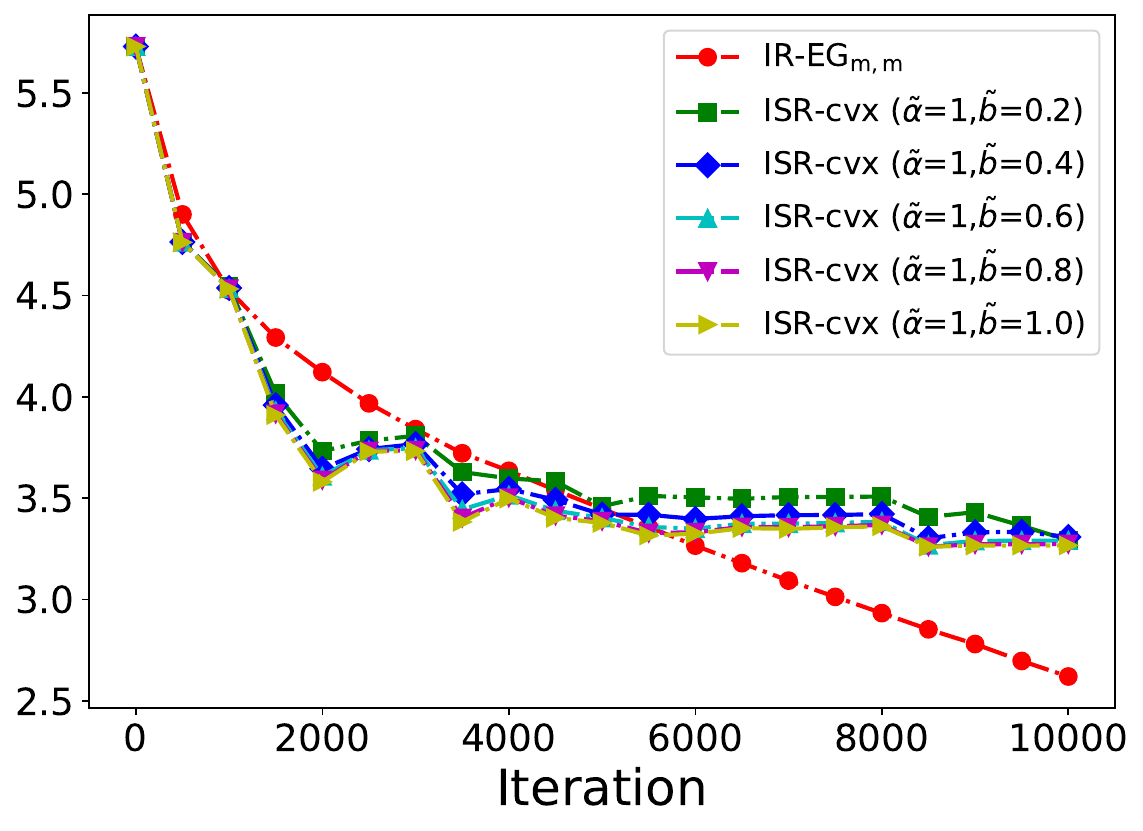}
			\end{minipage}
			&
			\begin{minipage}{.3\textwidth}
				\includegraphics[scale=.2, angle=0]{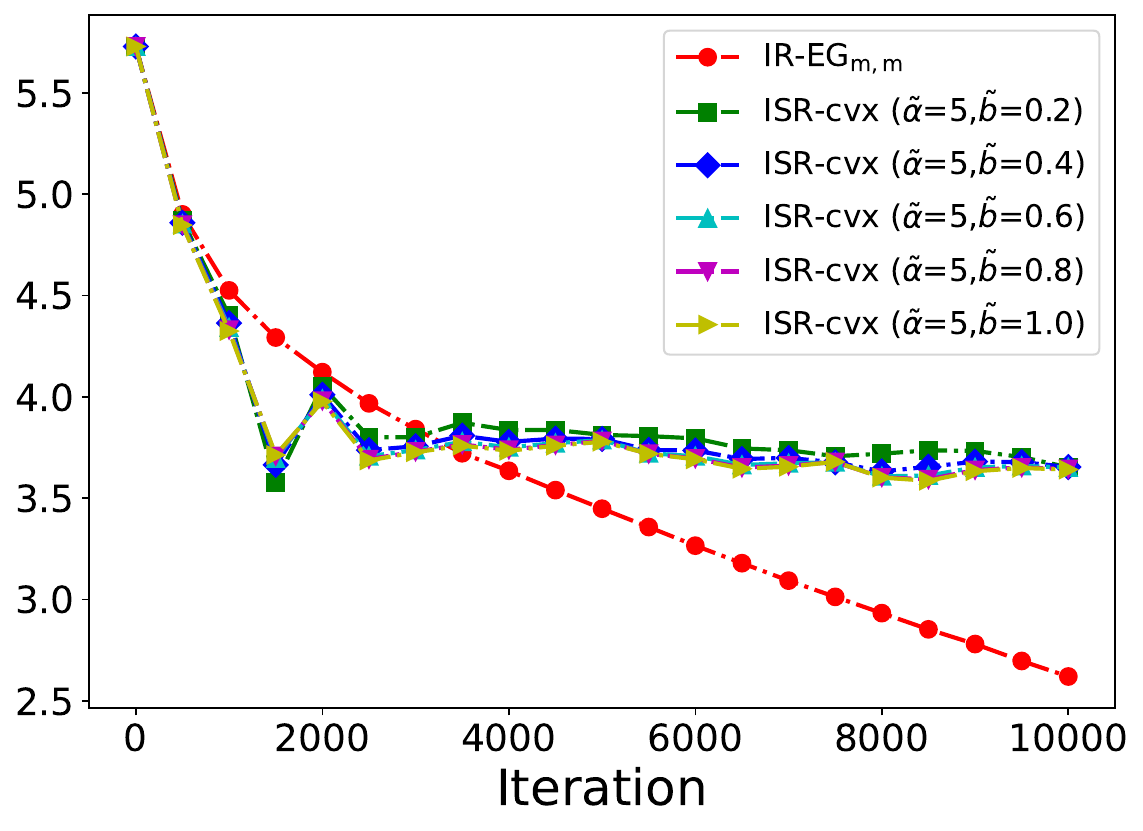}
			\end{minipage}
			&
			\begin{minipage}{.3\textwidth}
				\includegraphics[scale=.2, angle=0]{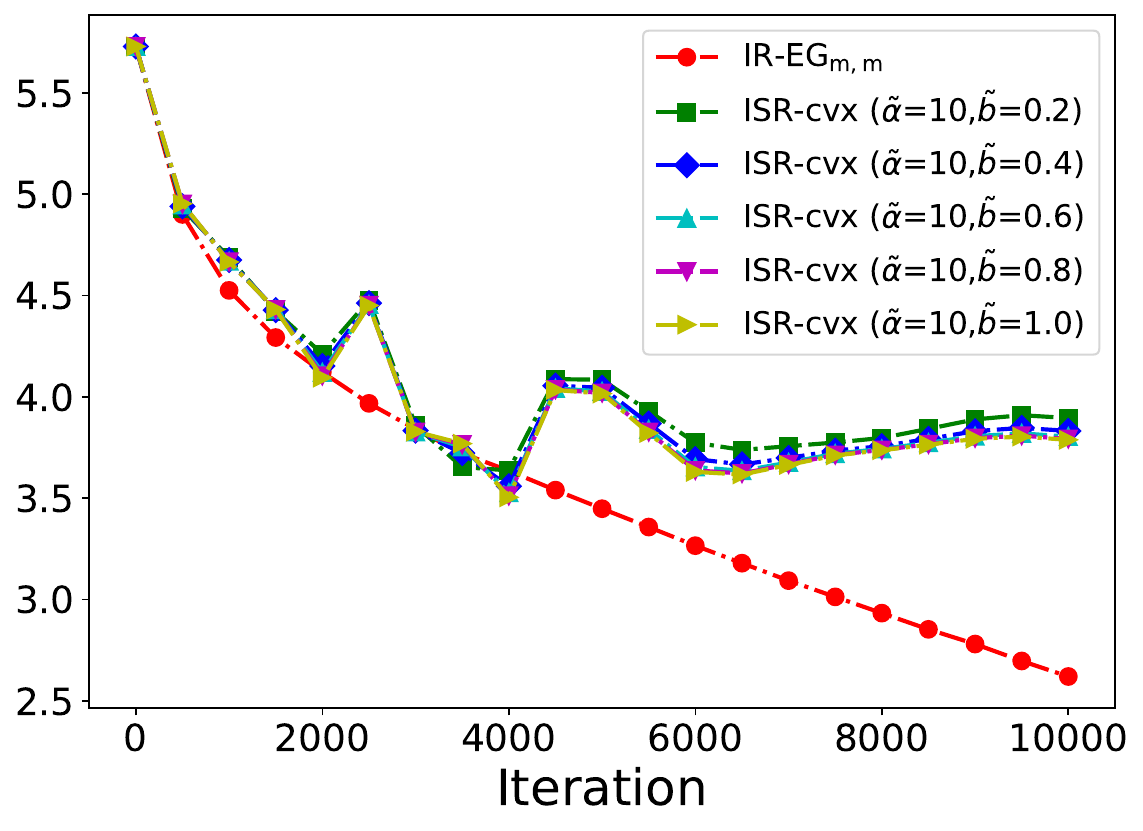}
			\end{minipage}
	\end{tabular}}
 \captionof{figure}{Algorithm~\ref{alg:IR-EG} vs. ISR-cvx for computing the best NE when $n_a=1.2$ in all arcs.}\label{fig:cvx-non-cap}
\end{table*}

\begin{table*}[t]
	\renewcommand\thetable{4}
	\setlength{\tabcolsep}{0pt}
	\centering 
		\begin{tabular}{c|  c  c  c c}
			\rotatebox[origin=c]{90}{{\scriptsize {suboptimality metric}}}
			&
			\begin{minipage}{.3\textwidth}
				\includegraphics[scale=.2, angle=0]{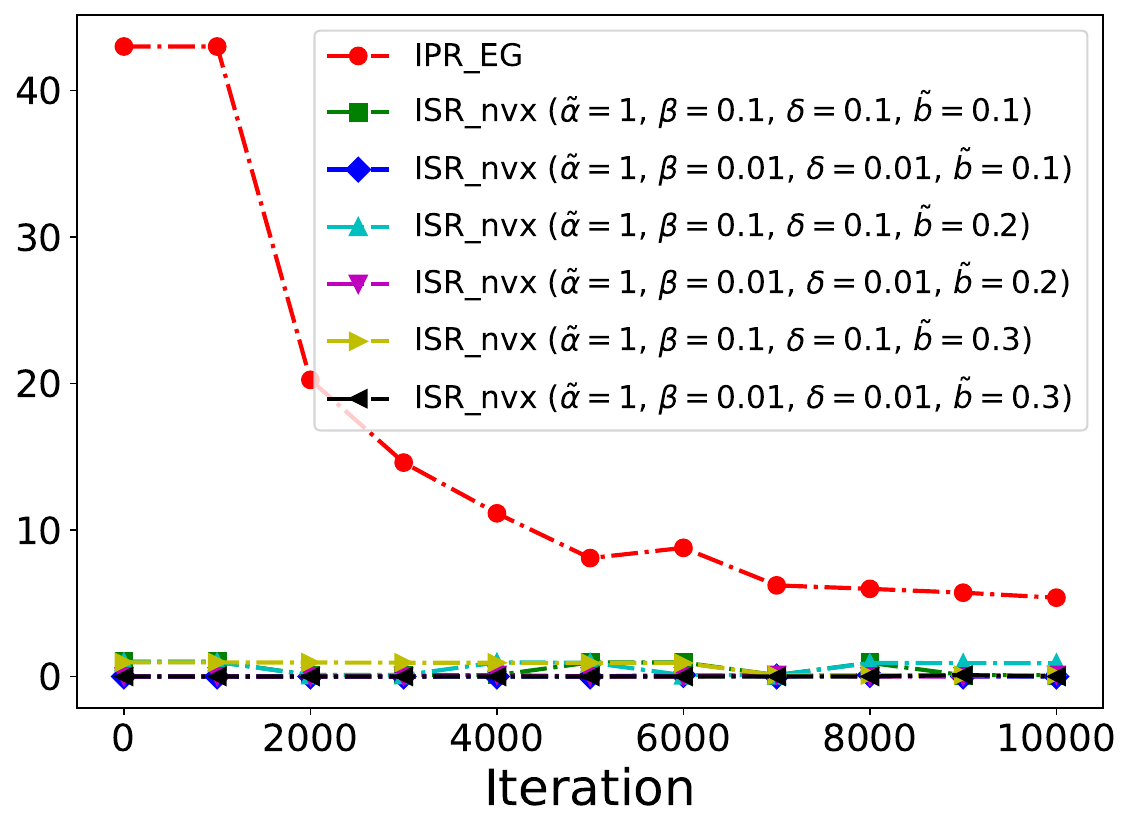}
			\end{minipage}
			&
			\begin{minipage}{.3\textwidth}
				\includegraphics[scale=.2, angle=0]{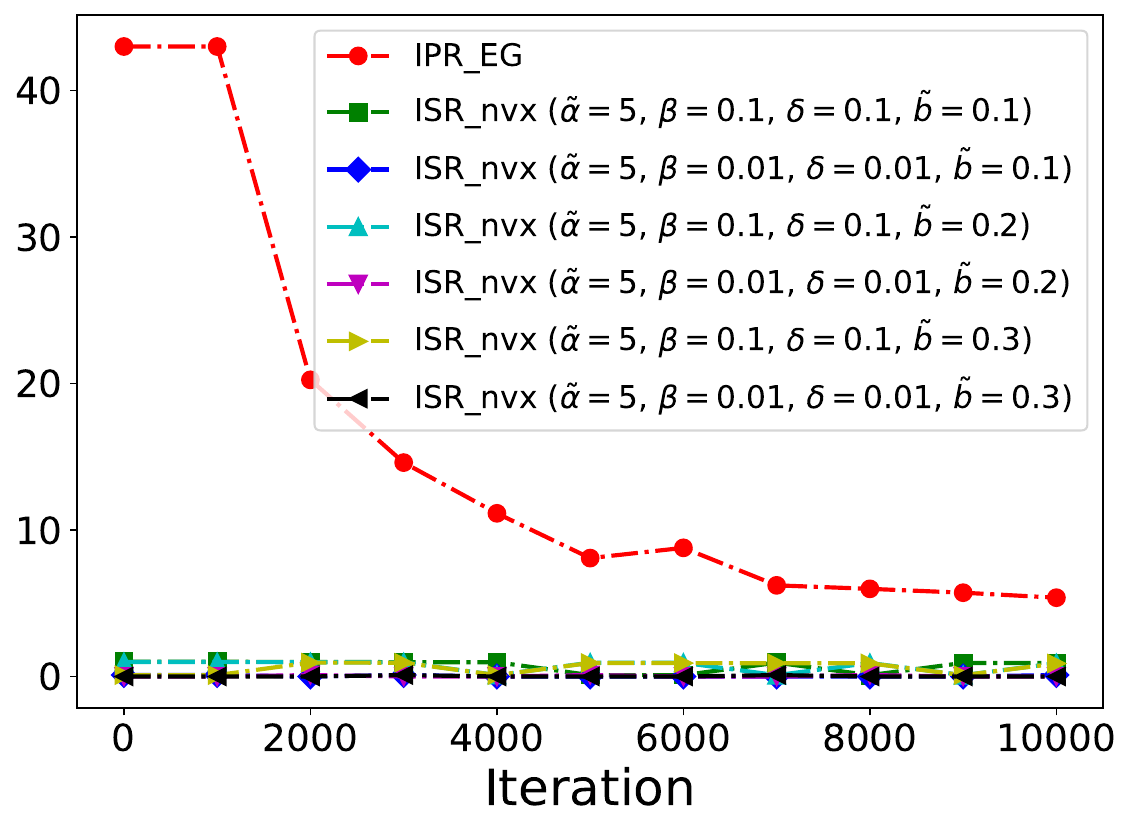}
			\end{minipage}
			&
			\begin{minipage}{.3\textwidth}
				\includegraphics[scale=.2, angle=0]{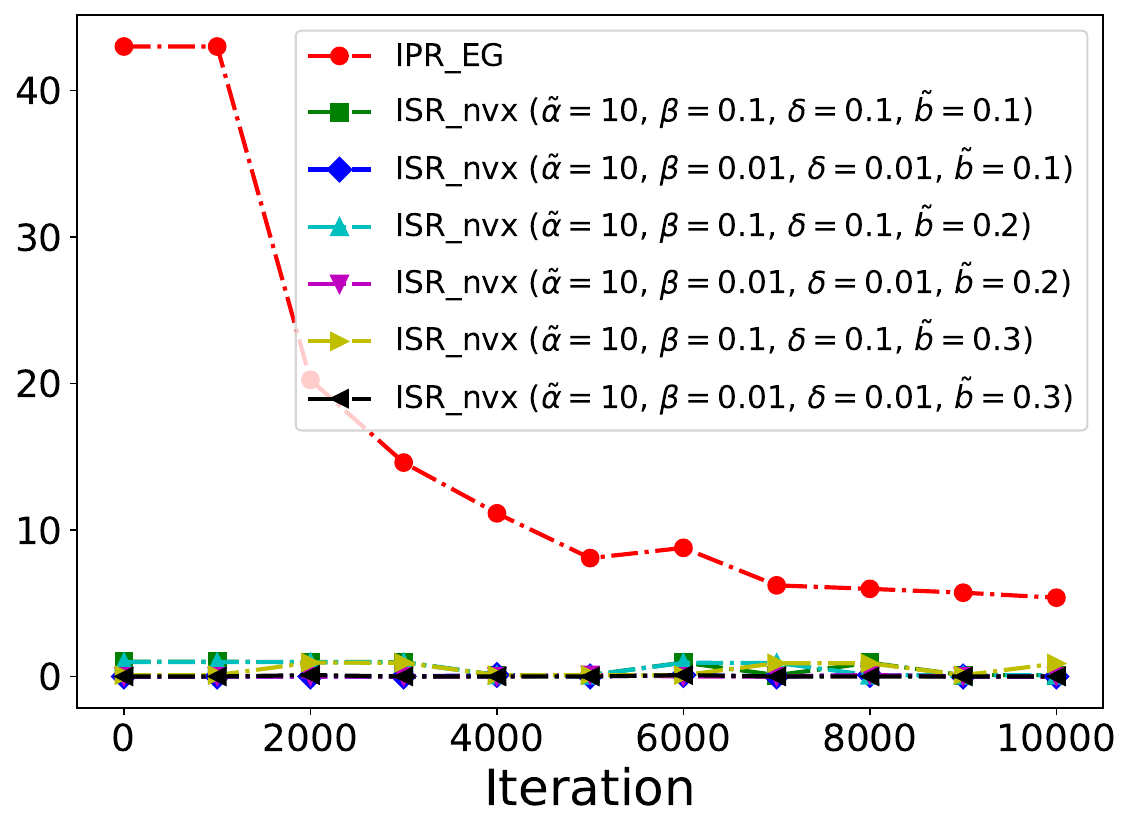}
			\end{minipage}
			\\
			\hline\\
			\rotatebox[origin=c]{90}{{\scriptsize{ $\log$(infeasibility metric)}}}
			&
			\begin{minipage}{.3\textwidth}
				\includegraphics[scale=.2, angle=0]{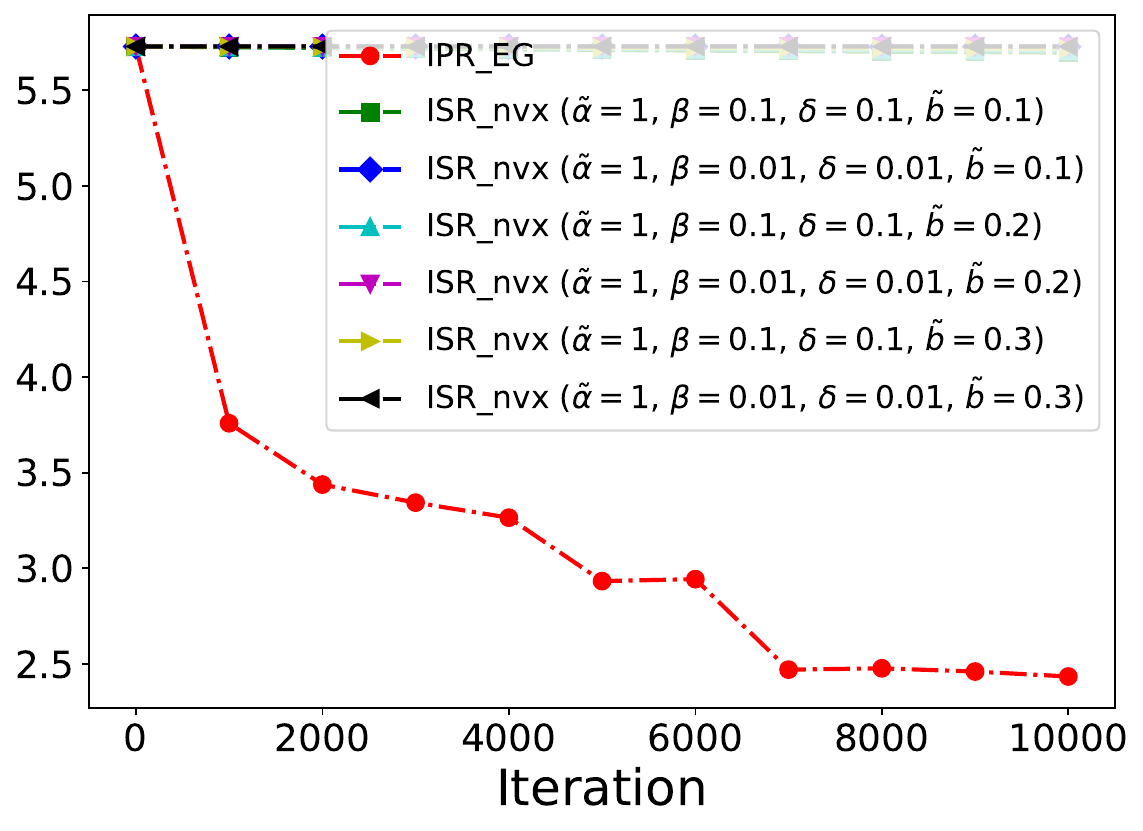}
			\end{minipage}
			&
			\begin{minipage}{.3\textwidth}
				\includegraphics[scale=.2, angle=0]{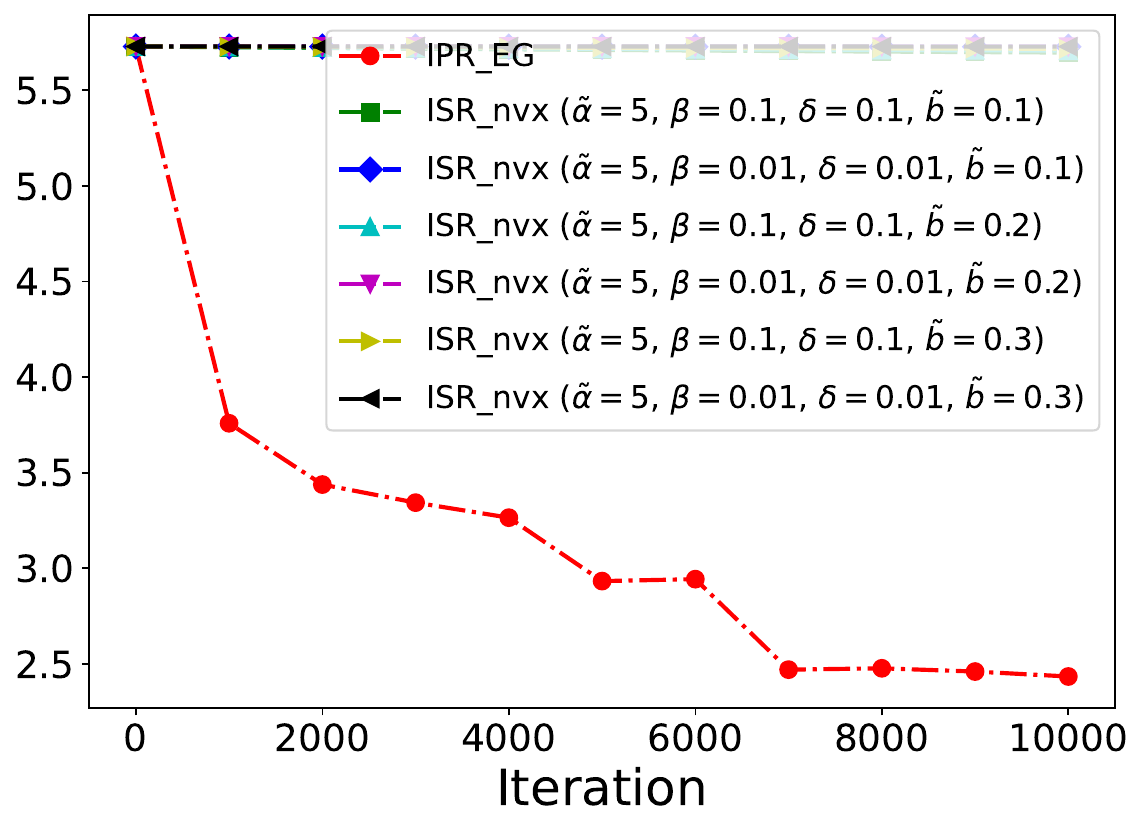}
			\end{minipage}
			&
			\begin{minipage}{.3\textwidth}
				\includegraphics[scale=.2, angle=0]{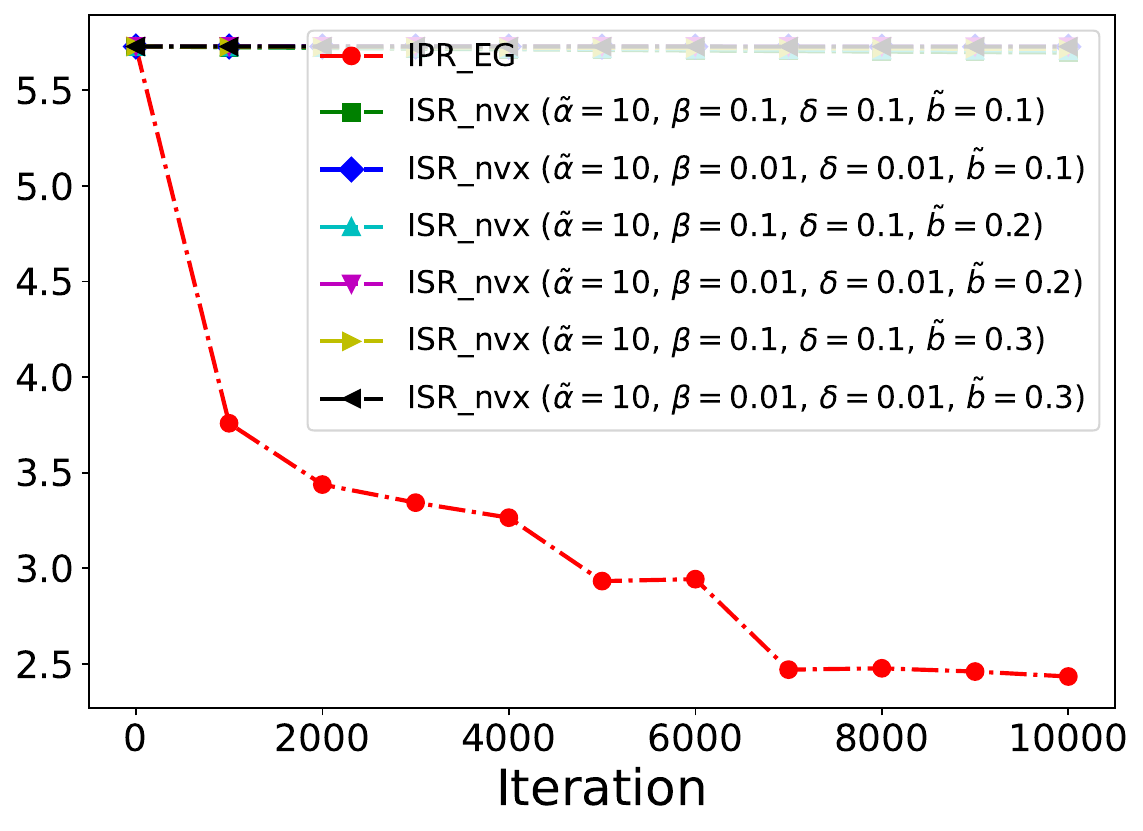}
			\end{minipage}
	\end{tabular}
 \captionof{figure}{Algorithm~\ref{alg:ncvx-IR-GD} vs. ISR-ncvx for computing the worst NE when $n_a=1.2$ in all arcs.}\label{fig:ncvx-non-cap}
\end{table*} 

\begin{table*}[t]
	\renewcommand\thetable{4}
	\setlength{\tabcolsep}{0pt}
	\centering
		\begin{tabular}{ c  c  c c}
			\begin{minipage}{.32\textwidth}
				\includegraphics[scale=.2, angle=0]{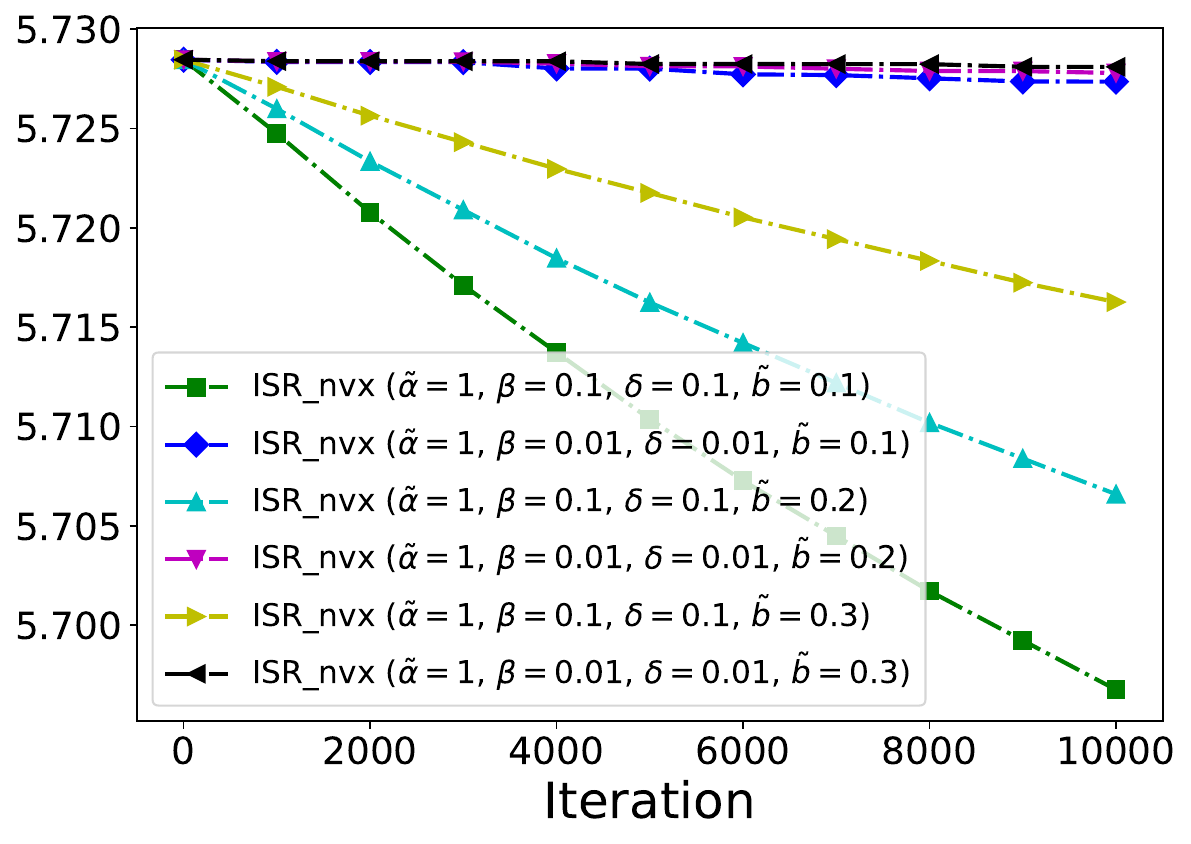}
			\end{minipage}
			&
			\begin{minipage}{.32\textwidth}
				\includegraphics[scale=.2, angle=0]{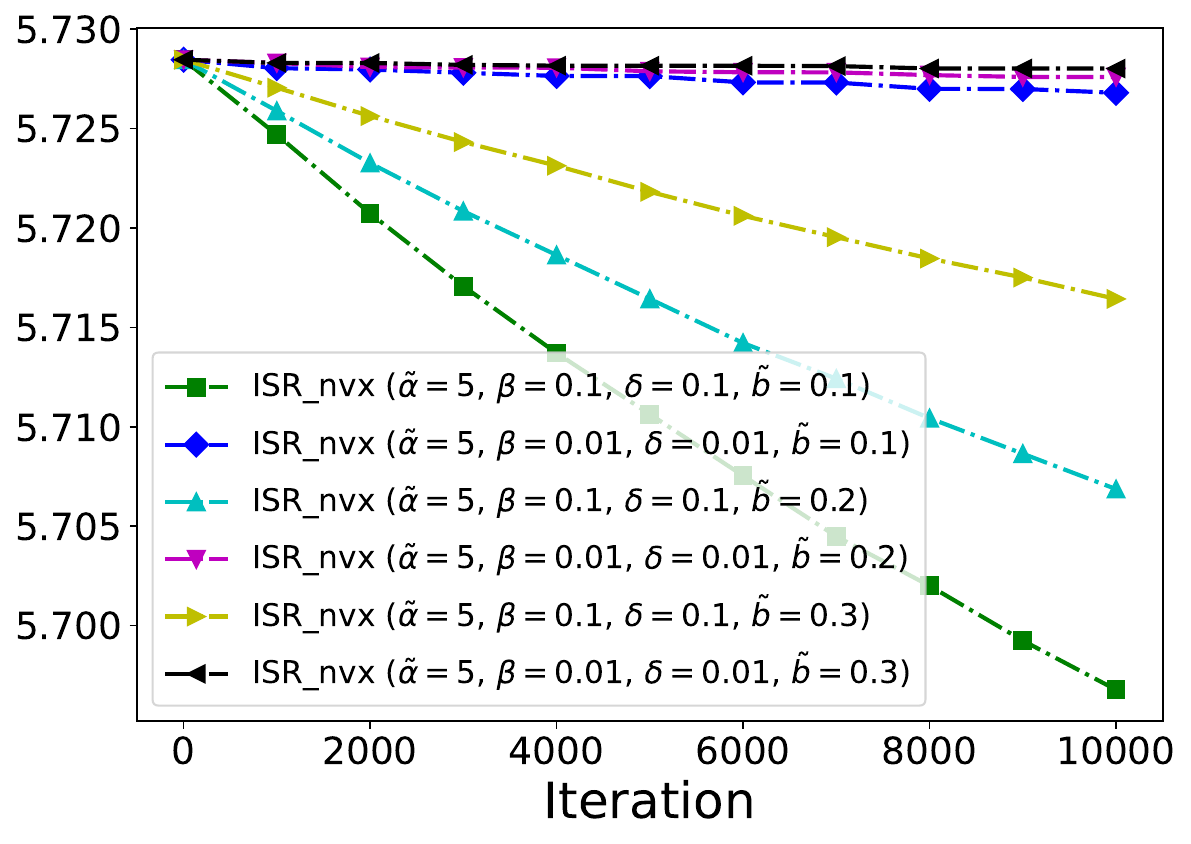}
			\end{minipage}
			&
			\begin{minipage}{.32\textwidth}
				\includegraphics[scale=.2, angle=0]{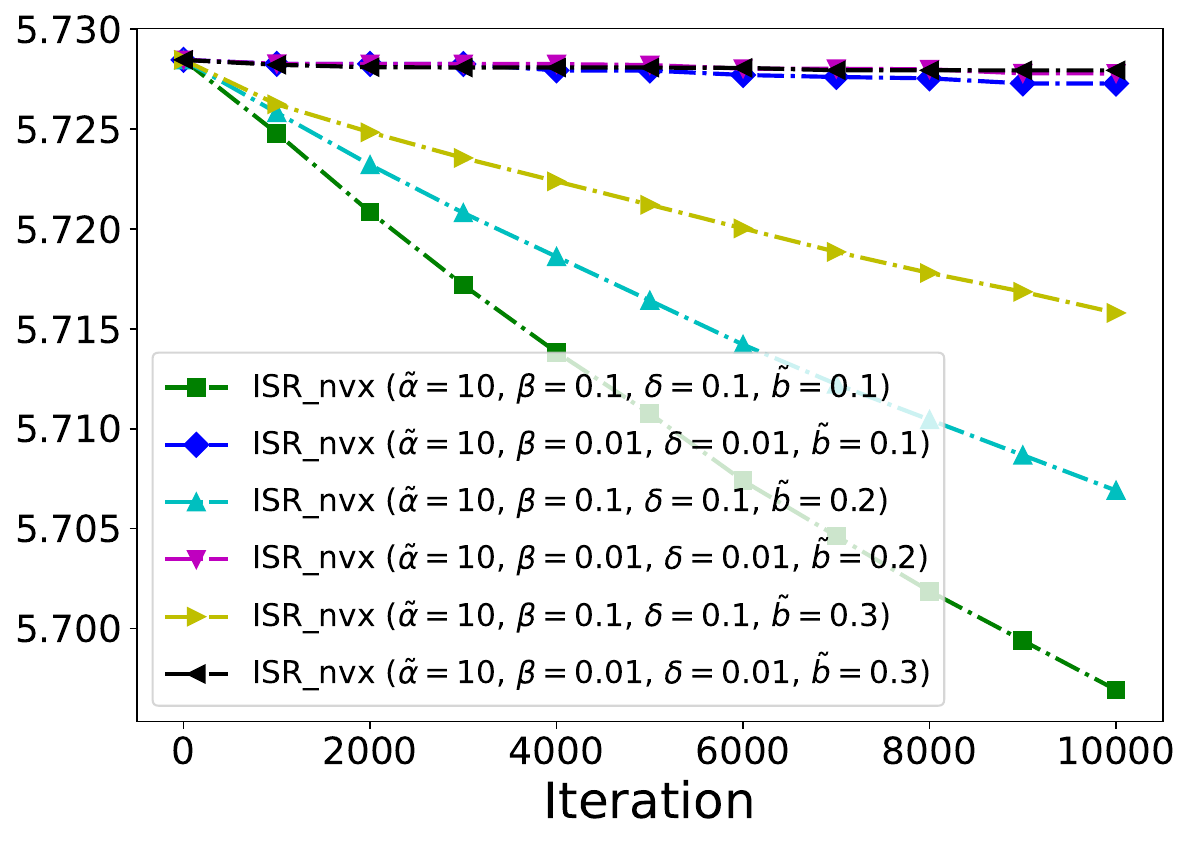}
			\end{minipage}
	\end{tabular} 
	\captionof{figure}{Detailed performance of ISR-ncvx in Fig.~\ref{fig:ncvx-non-cap} in terms of infeasibility error.}\label{fig:ncvx-non-cap-isr}
\end{table*}

}

\section{Concluding remarks}\label{sec:conclude}
In noncooperative game theory, understanding the quality of Nash equilibrium has been a storied area of research and has led to the emergence of popular metrics including the Price of Anarchy and the Price of Stability. The evaluation of these ratios is complicated by the need to compute the worst and the best equilibrium, respectively. In this paper, our goal is to devise a class of iteratively regularized extragradient methods with performance guarantees for computing the optimal equilibrium. To this end, we consider optimization problems with \far{monotone} variational inequality (VI) constraints when the objective function is either (i)  \ssrtwo{convex}, (ii) strongly convex, or (iii) nonconvex. In (i), we considerably improve the existing complexity guarantees. In  (ii) and (iii), we derive new complexity guarantees for solving this class of problems. Notably, this appears to be the first work where nonconvex optimization with monotone VI constraints is addressed with complexity guarantees. We further show that our results in (i) and (ii) can be generalized to address a class of bilevel VIs. Extensions of the results in this work to stochastic regimes are among interesting directions for future research.

\section{Acknowledgments}
\ssrtwo{We would like to thank the two anonymous referees for their constructive comments. We also thank Drs. Uday V. Shanbhag, Rasoul Etesami, and Shisheng Cui for the insightful feedback and suggestions about this work.}

%
%
%



\bibliographystyle{siamplain}
\bibliography{references}

\begin{thebibliography}{10}

\bibitem{alacaoglu2022stochastic}
{\sc A.~Alacaoglu and Y.~Malitsky}, {\em Stochastic variance reduction for
  variational inequality methods}, in Conference on Learning Theory, PMLR,
  2022, pp.~778--816.

\bibitem{alizadeh2024randomized}
{\sc Z.~Alizadeh, A.~Jalilzadeh, and F.~Yousefian}, {\em Randomized lagrangian
  stochastic approximation for large-scale constrained stochastic nash games},
  Optimization Letters, 18 (2024), pp.~377--401.

\bibitem{amini2019iterative}
{\sc M.~Amini and F.~Yousefian}, {\em An iterative regularized mirror descent
  method for ill-posed nondifferentiable stochastic optimization}, arXiv
  preprint arXiv:1901.09506,  (2019).

\bibitem{anshelevich2008price}
{\sc E.~Anshelevich, A.~Dasgupta, J.~Kleinberg, {\'E}.~Tardos, T.~Wexler, and
  T.~Roughgarden}, {\em The price of stability for network design with fair
  cost allocation}, SIAM Journal on Computing, 38 (2008), pp.~1602--1623.

\bibitem{book:AmirB}
{\sc A.~Beck}, {\em First-Order Methods in Optimization}, Society for
  Industrial and Applied Mathematics, Philadelphia, PA, 2017,
  \url{https://doi.org/10.1137/1.9781611974997}.

\bibitem{beck2014first}
{\sc A.~Beck and S.~Sabach}, {\em A first order method for finding minimal
  norm-like solutions of convex optimization problems}, Mathematical
  Programming, 147 (2014), pp.~25--46.

\bibitem{benenati2023optimal}
{\sc E.~Benenati, W.~Ananduta, and S.~Grammatico}, {\em Optimal selection and
  tracking of generalized {N}ash equilibria in monotone games}, IEEE
  Transactions on Automatic Control,  (2023).

\bibitem{benenati2023semi}
{\sc E.~Benenati, W.~Ananduta, and S.~Grammatico}, {\em A semi-decentralized
  tikhonov-based algorithm for optimal generalized nash equilibrium selection},
  in 2023 62nd IEEE Conference on Decision and Control (CDC), IEEE, 2023,
  pp.~4243--4248.

\bibitem{bertsekas2003convex}
{\sc D.~Bertsekas, A.~Nedi\'c, and A.~Ozdaglar}, {\em Convex analysis and
  optimization}, vol.~1, Athena Scientific, 2003.

\bibitem{burke1993weak}
{\sc J.~V. Burke and M.~C. Ferris}, {\em Weak sharp minima in mathematical
  programming}, SIAM Journal on Control and Optimization, 31 (1993),
  pp.~1340--1359.

\bibitem{chan1982generalized}
{\sc D.~Chan and J.~Pang}, {\em The generalized quasi-variational inequality
  problem}, Mathematics of Operations Research, 7 (1982), pp.~211--222.

\bibitem{chen2023bilevel}
{\sc L.~Chen, J.~Xu, and J.~Zhang}, {\em On bilevel optimization without
  lower-level strong convexity}, arXiv preprint arXiv:2301.00712,  (2023).

\bibitem{chen2012stochastic}
{\sc X.~Chen, R.~J.-B. Wets, and Y.~Zhang}, {\em Stochastic variational
  inequalities: residual minimization smoothing sample average approximations},
  SIAM Journal on Optimization, 22 (2012), pp.~649--673.

\bibitem{cottle2009linear}
{\sc R.~W. Cottle, J.-S. Pang, and R.~E. Stone}, {\em The linear
  complementarity problem}, SIAM, 2009.

\bibitem{facchinei2010generalized}
{\sc F.~Facchinei and C.~Kanzow}, {\em Generalized {N}ash equilibrium
  problems}, Annals of Operations Research, 175 (2010), pp.~177--211.

\bibitem{facchinei2003finite}
{\sc F.~Facchinei and J.-S. Pang}, {\em Finite-dimensional variational
  inequalities and complementarity problems}, Springer, 2003.

\bibitem{FacchineiPang2003}
{\sc F.~Facchinei and J.-S. Pang}, {\em Finite-dimensional Variational
  Inequalities and Complementarity Problems. {V}ols. {I,II}}, Springer Series
  in Operations Research, Springer-Verlag, New York, 2003.

\bibitem{facchinei2014vi}
{\sc F.~Facchinei, J.-S. Pang, G.~Scutari, and L.~Lampariello}, {\em
  {VI}-constrained hemivariational inequalities: distributed algorithms and
  power control in ad-hoc networks}, Mathematical Programming, 145 (2014),
  pp.~59--96.

\bibitem{feinstein2023characterizing}
{\sc Z.~Feinstein and B.~Rudloff}, {\em Characterizing and computing the set of
  {N}ash equilibria via vector optimization}, Operations Research,  (2023).

\bibitem{ferris1991finite}
{\sc M.~C. Ferris and O.~L. Mangasarian}, {\em Finite perturbation of convex
  programs}, Applied Mathematics and Optimization, 23 (1991), pp.~263--273.

\bibitem{friedlander2008exact}
{\sc M.~P. Friedlander and P.~Tseng}, {\em Exact regularization of convex
  programs}, SIAM Journal on Optimization, 18 (2008), pp.~1326--1350.

\bibitem{iusem2017extragradient}
{\sc A.~N. Iusem, A.~Jofr{\'e}, R.~I. Oliveira, and P.~Thompson}, {\em
  Extragradient method with variance reduction for stochastic variational
  inequalities}, SIAM Journal on Optimization, 27 (2017), pp.~686--724.

\bibitem{jalilzadeh2024stochastic}
{\sc A.~Jalilzadeh, F.~Yousefian, and M.~Ebrahimi}, {\em Stochastic
  approximation for estimating the price of stability in stochastic nash
  games}, ACM Transactions on Modeling and Computer Simulation, 34 (2024),
  pp.~1--24.

\bibitem{jiang2008stochastic}
{\sc H.~Jiang and H.~Xu}, {\em Stochastic approximation approaches to the
  stochastic variational inequality problem}, IEEE Transactions on Automatic
  Control, 53 (2008), pp.~1462--1475.

\bibitem{jiang2023conditional}
{\sc R.~Jiang, N.~Abolfazli, A.~Mokhtari, and E.~Y. Hamedani}, {\em A
  conditional gradient-based method for simple bilevel optimization with convex
  lower-level problem}, in International Conference on Artificial Intelligence
  and Statistics, PMLR, 2023, pp.~10305--10323.

\bibitem{juditsky2011solving}
{\sc A.~Juditsky, A.~Nemirovski, and C.~Tauvel}, {\em Solving variational
  inequalities with stochastic mirror-prox algorithm}, Stochastic Systems, 1
  (2011), pp.~17--58.

\bibitem{kannan2019optimal}
{\sc A.~Kannan and U.~V. Shanbhag}, {\em Optimal stochastic extragradient
  schemes for pseudomonotone stochastic variational inequality problems and
  their variants}, Computational Optimization and Applications, 74 (2019),
  pp.~779--820.

\bibitem{kaushik2023incremental}
{\sc H.~D. Kaushik, S.~Samadi, and F.~Yousefian}, {\em An incremental gradient
  method for optimization problems with variational inequality constraints},
  IEEE Transactions on Automatic Control,  (2023),
  \url{https://doi.org/10.1109/TAC.2023.3251851}.

\bibitem{doi:10.1137/20M1357378}
{\sc H.~D. Kaushik and F.~Yousefian}, {\em A method with convergence rates for
  optimization problems with variational inequality constraints}, SIAM Journal
  on Optimization, 31 (2021), pp.~2171--2198,
  \url{https://doi.org/10.1137/20M1357378}.

\bibitem{korpelevich1976extragradient}
{\sc G.~M. Korpelevich}, {\em The extragradient method for finding saddle
  points and other problems}, Matecon, 12 (1976), pp.~747--756.

\bibitem{kotsalis2022simple}
{\sc G.~Kotsalis, G.~Lan, and T.~Li}, {\em Simple and optimal methods for
  stochastic variational inequalities, {I}: operator extrapolation}, SIAM
  Journal on Optimization, 32 (2022), pp.~2041--2073.

\bibitem{koutsoupias1999worst}
{\sc E.~Koutsoupias and C.~Papadimitriou}, {\em Worst-case equilibria}, in
  Annual symposium on theoretical aspects of computer science, Springer, 1999,
  pp.~404--413.

\bibitem{lampariello2022solution}
{\sc L.~Lampariello, G.~Priori, and S.~Sagratella}, {\em On the solution of
  monotone nested variational inequalities}, Mathematical Methods of Operations
  Research, 96 (2022), pp.~421--446.

\bibitem{latafat2023adabim}
{\sc P.~Latafat, A.~Themelis, S.~Villa, and P.~Patrinos}, {\em {AdaBiM}: An
  adaptive proximal gradient method for structured convex bilevel
  optimization}, arXiv preprint arXiv:2305.03559,  (2023).

\bibitem{marcotte1998weak}
{\sc P.~Marcotte and D.~Zhu}, {\em Weak sharp solutions of variational
  inequalities}, SIAM Journal on Optimization, 9 (1998), pp.~179--189.

\bibitem{merchav2023convex}
{\sc R.~Merchav and S.~Sabach}, {\em Convex bi-level optimization problems with
  nonsmooth outer objective function}, SIAM Journal on Optimization, 33 (2023),
  pp.~3114--3142.

\bibitem{nemirovski2004prox}
{\sc A.~Nemirovski}, {\em Prox-method with rate of convergence {O}(1/t) for
  variational inequalities with lipschitz continuous monotone operators and
  smooth convex-concave saddle point problems}, SIAM Journal on Optimization,
  15 (2004), pp.~229--251.

\bibitem{nguyen1984efficient}
{\sc S.~Nguyen and C.~Dupuis}, {\em An efficient method for computing traffic
  equilibria in networks with asymmetric transportation costs}, Transportation
  science, 18 (1984), pp.~185--202.

\bibitem{papadimitriou2001algorithms}
{\sc C.~Papadimitriou}, {\em Algorithms, games, and the internet}, in
  Proceedings of the thirty-third annual ACM symposium on Theory of computing,
  2001, pp.~749--753.

\bibitem{rockafellar2009variational}
{\sc R.~T. Rockafellar and R.~J.-B. Wets}, {\em Variational analysis},
  vol.~317, Springer Science \& Business Media, 2009.

\bibitem{sabach2017first}
{\sc S.~Sabach and S.~Shtern}, {\em A first order method for solving convex
  bilevel optimization problems}, SIAM Journal on Optimization, 27 (2017),
  pp.~640--660.

\bibitem{samadi2024achieving}
{\sc S.~Samadi, D.~Burbano, and F.~Yousefian}, {\em Achieving optimal
  complexity guarantees for a class of bilevel convex optimization problems},
  in 2024 American Control Conference (ACC), IEEE, 2024, pp.~2206--2211.

\bibitem{shen2023online}
{\sc L.~Shen, N.~Ho-Nguyen, and F.~K{\i}l{\i}n{\c{c}}-Karzan}, {\em An online
  convex optimization-based framework for convex bilevel optimization},
  Mathematical Programming, 198 (2023), pp.~1519--1582.

\bibitem{solodov2007explicit}
{\sc M.~Solodov}, {\em An explicit descent method for bilevel convex
  optimization}, Journal of Convex Analysis, 14 (2007), p.~227.

\bibitem{solodov2007bundle}
{\sc M.~V. Solodov}, {\em A bundle method for a class of bilevel nonsmooth
  convex minimization problems}, SIAM Journal on Optimization, 18 (2007),
  pp.~242--259.

\bibitem{studniarski1999weak}
{\sc M.~Studniarski and D.~E. Ward}, {\em Weak sharp minima: characterizations
  and sufficient conditions}, SIAM Journal on Control and Optimization, 38
  (1999), pp.~219--236.

\bibitem{thong2020strong}
{\sc D.~V. Thong, N.~A. Triet, X.-H. Li, and Q.-L. Dong}, {\em Strong
  convergence of extragradient methods for solving bilevel pseudo-monotone
  variational inequality problems}, Numerical Algorithms, 83 (2020),
  pp.~1123--1143.

\bibitem{tikhonov1963solution}
{\sc A.~N. Tikhonov}, {\em On the solution of ill-posed problems and the method
  of regularization}, in Doklady akademii nauk, vol.~151, Russian Academy of
  Sciences, 1963, pp.~501--504.

\bibitem{van2021regularization}
{\sc D.~Van~Hieu and A.~Moudafi}, {\em Regularization projection method for
  solving bilevel variational inequality problem}, Optimization Letters, 15
  (2021), pp.~205--229.

\bibitem{yamada2005hybrid}
{\sc I.~Yamada and N.~Ogura}, {\em Hybrid steepest descent method for
  variational inequality problem over the fixed point set of certain
  quasi-nonexpansive mappings},  (2005).

\bibitem{yin2009robust}
{\sc Y.~Yin, S.~M. Madanat, and X.-Y. Lu}, {\em Robust improvement schemes for
  road networks under demand uncertainty}, European Journal of Operational
  Research, 198 (2009), pp.~470--479.

\bibitem{yousefian2021bilevel}
{\sc F.~Yousefian}, {\em Bilevel distributed optimization in directed
  networks}, in 2021 American Control Conference (ACC), IEEE, 2021,
  pp.~2230--2235.

\bibitem{yousefian2014optimal}
{\sc F.~Yousefian, A.~Nedi{\'c}, and U.~V. Shanbhag}, {\em Optimal robust
  smoothing extragradient algorithms for stochastic variational inequality
  problems}, in 53rd IEEE conference on decision and control, IEEE, 2014,
  pp.~5831--5836.

\bibitem{yousefian2017smoothing}
{\sc F.~Yousefian, A.~Nedi{\'c}, and U.~V. Shanbhag}, {\em On smoothing,
  regularization, and averaging in stochastic approximation methods for
  stochastic variational inequality problems}, Mathematical Programming, 165
  (2017), pp.~391--431.

\bibitem{yousefian2018stochastic}
{\sc F.~Yousefian, A.~Nedi{\'c}, and U.~V. Shanbhag}, {\em On stochastic
  mirror-prox algorithms for stochastic cartesian variational inequalities:
  Randomized block coordinate and optimal averaging schemes}, Set-Valued and
  Variational Analysis, 26 (2018), pp.~789--819.

\end{thebibliography}

\newpage
\appendix
\section{Supplementary results and proofs}
{\begin{lemma}\label{lem:example1_monotone_F}\em
Consider problem~\eqref{prob:opt_select} and let $X$ and $F$ be given by~\eqref{eqn:def_example1} where the set $Y\subseteq \mathbb{R}^p$ is closed and convex and functions $g$, $h_j$, $j=1,\ldots,q$ are continuously differentiable and convex. Then, the mapping $F$ is  \ssrtwo{monotone} on $X$. 
\end{lemma}
\begin{proof}
Suppose $x_1,x_2 \in X$ are arbitrary given vectors. From \eqref{eqn:def_example1}, we can write 
\begin{align*}
 (F(x_1)-F(x_2))\fyy{^\top}(x_1-x_2) &=  (\nabla g(y_1) -\nabla g(y_2) )\fyy{^\top}(y_1-y_2)\\
 &+ \sum_{j=1}^q(\nabla h_j(y_1)\lambda_{1,j}-\nabla h_j(y_2)\lambda_{2,j})\fyy{^\top}(y_1-y_2)\\
 & +\sum_{j=1}^q(h_j(y_2)-h_j(y_1))(\lambda_{1,j}-\lambda_{2,j}),
\end{align*} 
where $\lambda_{1,j}$ and $\lambda_{2,j}$ denote the $j$th dual variable in $x_1$ and $x_2$, respectively. Let us define $\mathcal{J}_{+} \triangleq \{j \in [q]\mid \lambda_{1,j}-\lambda_{2,j} \geq 0\}$ and $\mathcal{J}_{-} \triangleq \{j \in [q]\mid \lambda_{1,j}-\lambda_{2,j} < 0\}$. Then, adding and subtracting the two terms $\sum_{j \in \mathcal{J}_{+}}\nabla h_j(y_1)\lambda_{2,j} $ and  $\sum_{j \in \mathcal{J}_{-}}\nabla h_j(y_2)\lambda_{1,j} $, we obtain
\begin{align*}
& (F(x_1)-F(x_2))\fyy{^\top}(x_1-x_2) =  (\nabla g(y_1) -\nabla g(y_2) )\fyy{^\top}(y_1-y_2)\\
 &+ \sum_{j \in \mathcal{J}_{+}}^q(\nabla h_j(y_1)\lambda_{1,j}-\nabla h_j(y_1)\lambda_{2,j}+\nabla h_j(y_1)\lambda_{2,j}-\nabla h_j(y_2)\lambda_{2,j})\fyy{^\top}(y_1-y_2)\\
  &+ \sum_{j \in \mathcal{J}_{-}}^q(\nabla h_j(y_1)\lambda_{1,j}-\nabla h_j(y_2)\lambda_{1,j}+\nabla h_j(y_2)\lambda_{1,j}-\nabla h_j(y_2)\lambda_{2,j})\fyy{^\top}(y_1-y_2)\\
 & +\sum_{j=1}^q(h_j(y_2)-h_j(y_1))(\lambda_{1,j}-\lambda_{2,j}).
\end{align*}
Rearranging the terms, we obtain 
\begin{align*}
& (F(x_1)-F(x_2))\fyy{^\top}(x_1-x_2) =  (\nabla g(y_1) -\nabla g(y_2) )\fyy{^\top}(y_1-y_2)\\
 &+ \sum_{j \in \mathcal{J}_{+}}^q(\lambda_{1,j}-\lambda_{2,j})\nabla h_j(y_1)\fyy{^\top}(y_1-y_2)+ \sum_{j \in \mathcal{J}_{+}}^q\lambda_{2,j}(\nabla h_j(y_1)-\nabla h_j(y_2))\fyy{^\top}(y_1-y_2)\\
  &+ \sum_{j \in \mathcal{J}_{-}}^q\lambda_{1,j}(\nabla h_j(y_1)-\nabla h_j(y_2))\fyy{^\top}(y_1-y_2)+\sum_{j \in \mathcal{J}_{-}}^q(\lambda_{1,j}-\lambda_{2,j})\nabla h_j(y_2)\fyy{^\top}(y_1-y_2)\\
 & +\sum_{j=1}^q(h_j(y_2)-h_j(y_1))(\lambda_{1,j}-\lambda_{2,j}).
\end{align*}
From the nonnegativity of the Lagrange multipliers and convexity of functions $g$ and $h_j$, we have 
\begin{align*}
& (F(x_1)-F(x_2))\fyy{^\top}(x_1-x_2) \geq  \\
 &+ \sum_{j \in \mathcal{J}_{+}}^q(\lambda_{1,j}-\lambda_{2,j})\left(h_j(y_2)-h_j(y_1)+\nabla h_j(y_1)\fyy{^\top}(y_1-y_2)\right)\\
  &+\sum_{j \in \mathcal{J}_{-}}^q(\lambda_{1,j}-\lambda_{2,j})\left(h_j(y_2)-h_j(y_1)+\nabla h_j(y_2)\fyy{^\top}(y_1-y_2)\right).
\end{align*}
In view of convexity of $h_j$, we have that $h_j(y_2)-h_j(y_1)+\nabla h_j(y_1)\fyy{^\top}(y_1-y_2) \geq 0$ and $h_j(y_2)-h_j(y_1)+\nabla h_j(y_2)\fyy{^\top}(y_1-y_2) \leq 0$.  Invoking the definitions of $\mathcal{J}_{+}$  and $\mathcal{J}_{-} $, we have $(F(x_1)-F(x_2))\fyy{^\top}(x_1-x_2) \geq 0$ and thus, $F$ is  \ssrtwo{monotone} on $X$. 
\end{proof}}
%

\subsection{Proof of Lemma~\ref{lem:bilevelVI_gap_ws}}\label{sec:proof_lemma_ws}
\noindent (i) \far{Let $x \in X$ be a given vector.} From the definition of the dual gap function, for any $\hat{x} \in \mbox{SOL}(X,F)$, we have
\begin{align}\label{eqn:hatx_monotoneH} 
H(\hat{x} )\fyy{^\top}\left(x-\hat{x} \right) \leq  \sup_{ y\in  {\tiny \mbox{SOL}(X,F)}} H(y)\fyy{^\top}(x-y)= \mbox{Gap}\left(x,\mbox{SOL}(X,F),H\right)   .
\end{align}
Invoking the Cauchy-Schwarz inequality, from the preceding inequality we may write $
 -B_H \|x-\hat{x} \| \leq \mbox{Gap}\left(x,\mbox{SOL}(X,F),H\right)$. Now, choosing $\hat{x} :=\Pi_{\tiny \mbox{SOL}(X,F)}(x)$, we obtain
\begin{align*}
 -B_H \ \mbox{dist}\left(x,\mbox{SOL}(X,F)\right) \leq \mbox{Gap}\left(x,\mbox{SOL}(X,F),H\right)       .
\end{align*}
 
\noindent (ii) We show the result for both \far{convex} and strongly convex cases by letting $\mu \geq 0$. Let us define $\hat{x}\triangleq \Pi_{\tiny\mbox{SOL}(X,F)} (x)$. Let $x^*$ denote an optimal solution to problem~\eqref{eqn:optVI}. Since $\hat{x} \in \mbox{SOL}(X,F)$, in view of the optimality condition, we have $\nabla f(x^*)\fyy{^\top}(\hat{x}-x^*) \geq 0$. From (strong) convexity of $f$ and the Cauchy-Schwarz inequality, we may write 
\begin{align*}
f(x) - f(x^*) & \geq \nabla f(x^*)\fyy{^\top}(x-x^*) +\frac{\mu}{2}\|x-x^*\|^2
\\
&=\nabla f(x^*)\fyy{^\top}(x-\hat{x})+\nabla f(x^*)\fyy{^\top}(\hat{x}-x^*)+\frac{\mu}{2}\|x-x^*\|^2\\ 
&\geq -\|\nabla f(x^*)\|\, \mbox{dist}(x,\mbox{SOL}(X,F))+\frac{\mu}{2}\|x-x^*\|^2.
\end{align*}

\noindent (iii) Let $x \in X$ be given. Let $x^*_F \in \mbox{SOL}(X,F)$ be an arbitrary vector. Then, from Definition~\ref{def:weaksharp}, we have \ses{$F(x^*_F)\fyy{^\top}(x- x^*_F) \geq \alpha\,  \mbox{dist}^ \mathcal{M}\left(x,\mbox{SOL}(X,F)\right)$}. Invoking the definition of the dual gap function and the preceding relation, we obtain  
\begin{align*}\mbox{Gap}\left(x,X,F\right) = \sup_{ y\in  {X}} F(y)\fyy{^\top}(x-y) \geq F(x^*_F)\fyy{^\top}(x- x^*_F) \geq \ses{\alpha\,  \mbox{dist}^ \mathcal{M}\left(x,\mbox{SOL}(X,F)\right)}.
\end{align*} 
\subsection {Proof of Lemma~\ref{lem:monotone_lemma_ineq}}\label{sec:proof_lemma_ineq}
(i) Let $x \in X$ be an arbitrary vector. For $k \geq 0$, we have
\begin{align*}
 \|x_{k+1} - x\|^2 &= \|x_{k+1} -x_k + x_k - x\|^2 \\
  & =\|x_{k+1} - x_k\|^2 + \|x_k - x\|^2 + 2(x_{k+1} - x_k)\fyy{^\top}(x_k -x_{k+1} + x_{k+1} - x)  \\
 & =\|x_k - x\|^2 - \|x_{k+1} - x_k\|^2 + 2(x_{k+1} - x_k)\fyy{^\top}(x_{k+1} - x).
\end{align*}
Also, $\|x_{k+1} - x_k\|^2  =
\|x_{k+1} - y_{k+1}\|^2 + \|y_{k+1} - x_k\|^2 + 2(x_{k+1} - y_{k+1})\fyy{^\top}(y_{k+1} - x_k)$. From the preceding relations, we may write 
\begin{align}\label{eqn:ES_main1}
 \|x_{k+1} - x\|^2 & =\|x_k - x\|^2 - \|x_{k+1} - y_{k+1}\|^2 - \|y_{k+1} - x_k\|^2\nonumber\\
 & + 2(y_{k+1}-x_{k} )\fyy{^\top}( y_{k+1}-x_{k+1}) + 2(x_{k+1} - x_k)\fyy{^\top}(x_{k+1} - x).
\end{align}
Next, we invoke ~Lemma~\ref{lem:Projection theorm} for finding upper bounds on $(x_{k+1} - y_{k+1})\fyy{^\top}( x_k-y_{k+1} )$ and $(x_{k+1} - x_k)\fyy{^\top}(x_{k+1} - x)$. Recall that from the projection theorem, for all vectors $u\in\mathbb{R}^n$ and $\hat{x} \in X$ we have $\left(\Pi_X[u]-u\right)\fyy{^\top}\left(\hat{x}-\Pi_X[u]\right) \geq 0$. By substituting $u:=x_k-{\gamma}(F(y_{k+1})+\eta_kH(y_{k+1}))$, $\hat{x}:=x$, and noting that $x_{k+1}=\Pi_X[u]$, we obtain
$\left(x_{k+1}-x_k+{\gamma}(F(y_{k+1})+\eta_kH(y_{k+1}))\right)\fyy{^\top}\left(x-x_{k+1}\right) \geq 0.$ This implies that
\begin{align}\label{eqn:ES_main2}
(x_{k+1} - x_k)\fyy{^\top}(x_{k+1} - x)\leq {\gamma}\left(F(y_{k+1})+\eta_kH(y_{k+1})\right)\fyy{^\top}\left(x-x_{k+1}\right).
\end{align}
Let us invoke the projection theorem again by substituting $u:=x_k-{\gamma}(F(x_k)+\eta_kH(x_k))$, $\hat{x}:=x_{k+1}$, and noting that $y_{k+1}=\Pi_X[u]$. We obtain
\begin{align}\label{eqn:ES_main3}
(y_{k+1}-x_{k} )\fyy{^\top}( y_{k+1}-x_{k+1})\leq {\gamma}\left(F(x_k)+\eta_kH(x_k)\right)\fyy{^\top}\left(x_{k+1}-y_{k+1}\right).
\end{align}
From the equations~\eqref{eqn:ES_main1}, \eqref{eqn:ES_main2}, and \eqref{eqn:ES_main3}, we obtain 
\begin{align*}
 \|x_{k+1} - x\|^2 & \leq\|x_k - x\|^2 - \|x_{k+1} - y_{k+1}\|^2 - \|y_{k+1} - x_k\|^2\nonumber\\
 &+2{\gamma}\left(F(x_k)+\eta_kH(x_k)\right)\fyy{^\top}\left(x_{k+1}-y_{k+1}\right)\\
 & + 2{\gamma}\left(F(y_{k+1})+\eta_kH(y_{k+1})\right)\fyy{^\top}\left(x-x_{k+1}\right).
\end{align*}
Adding and subtracting $y_{k+1}$ inside $\left(x-x_{k+1}\right)$ in the preceding relation, we obtain
\begin{align*}
 \|x_{k+1} - x\|^2 & \leq\|x_k - x\|^2 - \|x_{k+1} - y_{k+1}\|^2 - \|y_{k+1} - x_k\|^2\nonumber\\
 &+2{\gamma}\left(F(x_k)+\eta_kH(x_k)-F(y_{k+1})-\eta_kH(y_{k+1})\right)\fyy{^\top}\left(x_{k+1}-y_{k+1}\right)\\
 & + 2{\gamma}\left(F(y_{k+1})+\eta_kH(y_{k+1})\right)\fyy{^\top}\left(x-y_{k+1}\right).
\end{align*}
Recall that for any $a,b \in \mathbb{R}^n$, $-\|a\|^2+2a\fyy{^\top}b \leq \|b\|^2$. Utilizing this relation, we obtain 
\begin{align*} 
 \|x_{k+1} - x\|^2 & \leq\|x_k - x\|^2   - \|y_{k+1} - x_k\|^2\nonumber\\
 &+ {\gamma}^2\left\|F(x_k)+\eta_kH(x_k)-F(y_{k+1})-\eta_kH(y_{k+1})\right\|^2 \\
 & + 2{\gamma}\left(F(y_{k+1})+\eta_kH(y_{k+1})\right)\fyy{^\top}\left(x-y_{k+1}\right).
\end{align*}
Invoking the Lipschitzian property of $F$ and $H$, we obtain \far{\eqref{eqn:EG_lemma_mm_last}.}

\noindent \far{(ii) Consider equation  \eqref{eqn:EG_lemma_mm_last}.} From ${\gamma}^2(L_F^2+\eta_k^2L_H^2) \leq 0.5$ and invoking the monotonicity of the mappings $F$ and $H$, we obtain $ 
{\gamma}\left(F(x)+\eta_kH(x)\right)\fyy{^\top}\left(y_{k+1}-x\right)    \leq 0.5\|x_k - x\|^2   -0.5\|x_{k+1} - x\|^2 .$
This completes the proof for part \far{(ii).}

\noindent \far{(iii)}  From the convexity of $f$, we have $ \nabla f(y_{k+1})\fyy{^\top}(x-y_{k+1}) \leq f(x) - f(y_{k+1}). $ Consider~\eqref{eqn:EG_lemma_mm_last} where $H(\bullet):=\nabla f(\bullet)$. Invoking \far{the preceding inequality} and the monotonicity of $F$, from \eqref{eqn:EG_lemma_mm_last} we obtain 
\begin{align*} 
 \|x_{k+1} - x\|^2 & \leq\|x_k - x\|^2   - \|y_{k+1} - x_k\|^2+ 2{\gamma}^2(L_F^2+\eta_k^2L_H^2)\|y_{k+1} - x_k\|^2 \\
 & + 2{\gamma} F(x)\fyy{^\top}\left(x-y_{k+1}\right)+2\gamma\eta_k (f(x) - f(y_{k+1})).\notag
\end{align*}
Rearranging the terms, we obtain the inequality in part \far{(iii).}
\subsection{Proof of Theorem~\ref{thm:bilevelVI}}\label{sec_app_thm:bilevelVI}
 \far{{\bf [Case 1]}} Consider the bounds in Theorem~\ref{Thm:asymptotic-m.m} (i) and (ii). Note that (1-i) and (1-ii) follow by $\eta_{K-1}=\frac{\eta_0}{{K}^b}$ and also that $\textstyle\sum_{k=0}^{K-1} \eta_k =\eta_0\textstyle\sum_{k=0}^{K-1} (k+1)^{-b} \leq \frac{\eta_0 K^{1-b}}{1-b}$, for $K\geq 2^{1/(1-b)},$ where the last inequality follows from \cite[Lemma~2.14]{doi:10.1137/20M1357378}. (1-iii) is implied from Lemma~ \ref{lem:bilevelVI_gap_ws} and (1-ii). 

\noindent \far{{\bf [Case 2]} To show (2-i) consider an arbitrary solution $x^* \in X^*$ where $X^*$ denotes the solution set of problem~\eqref{eqn:bilevelVI}. Consider \eqref{ineq:EG_lemma_merelymonotone}. Substituting $x:=x^* $, for a constant regularization parameter $\eta$, we have
\begin{align*} 
2{\gamma}\left(F(x^* )+\eta H(x^* )\right)\fyy{^\top}\left(y_{k+1}-x^* \right)   & \leq  \|x_k - x^* \|^2   -\|x_{k+1} - x^* \|^2 .
\end{align*}
Summing both sides over $k=0,\ldots,K-1$ and divining both sides by $K$, we obtain 
\begin{align}\label{eqn:FHx^*K} 
2{\gamma}\left(F(x^* )+\eta H(x^* )\right)\fyy{^\top}\left(\bar{y}_K-x^* \right)   & \leq  (\|x_0 - x^* \|^2   -\|x_{K} - x^* \|^2)/K .
\end{align}
In view of $x^* \in X^*_F$ and Definition~\ref{def:weaksharp}, we have $F(x^*)\fyy{^\top}(\bar{y}_K-x^*) \geq \alpha \mbox{dist}(\bar{y}_K ,X^*_F)$. Invoking the Cauchy-Schwarz inequality, we have
\begin{align*} 
H(x^* ) \fyy{^\top}\left(\bar{y}_K-x^* \right)   & =   H(x^* ) \fyy{^\top}\left(\bar{y}_K- \Pi_{X^*_F}[\bar{y}_K]+   \Pi_{X^*_F}[\bar{y}_K] -x^* \right)\\
& \geq  -\|H(x^*)\|\|\bar{y}_K- \Pi_{X^*_F}[\bar{y}_K]\|,
\end{align*}
where we used $ H(x^* ) \fyy{^\top}\left(   \Pi_{X^*_F}[\bar{y}_K] -x^* \right) \geq 0$ in view of $ \Pi_{X^*_F}[\bar{y}_K] \in X^*_F$ and that $x^*$ solves $\mbox{VI}(X^*_F,H)$.  Thus, from \eqref{eqn:FHx^*K}, we obtain 
\begin{align*} 
2{\gamma} \left( \alpha \mbox{dist}(\bar{y}_K ,X^*_F)- \eta\|H(x^*)\|\mbox{dist}(\bar{y}_K ,X^*_F)\right)   & \leq  (\|x_0 -x^*\|^2   -\|x_{K} - x^*\|^2)/K .
\end{align*}
This implies that $
2{\gamma}  (\alpha- \eta \|H(x^*)\|) \mbox{dist}(\bar{y}_K ,X^*_F)     \leq  \|x_0 - x^*\|^2 /K.$  
Recalling $\eta \leq \frac{\alpha}{2\|H(x^*)\|}$, the preceding relation implies that $\mbox{dist}(\bar{y}_K,X^*_F) \leq \frac{\|x_0-x^*\|^2}{{\gamma}\alpha K}$. 
To show  (2-ii), by invoking Lemma~\ref{lem:bilevelVI_gap_ws},  $ \mbox{Gap}\left(\bar{y}_K, X^*_F,H\right)   \geq - B_H\, \frac{\|x_0-x^*\|^2}{ {\gamma} \alpha K}$.  Combining this with the bound in part (1-i) for $b=0$, we obtain the result. }
\subsection{Proof of Corollary~\ref{cor:nested}}\label{sec:proof_lemma_cor:nested}
\far{
The proof can be done by invoking~\eqref{ineq:EG_lemma_merelymonotone2} and following similar steps to those in the proof of Theorem~\ref{thm:bilevelVI}. Hence, it is omitted.
}
\subsection{Proof of Lemma~\ref{lem:s_monotone_lemma_ineq}}\label{sec:proof_lemma_lem:s_monotone_lemma_ineq}
\noindent (i) From the monotonicity of $F$ and the strong monotonicity of $H$, we have
\begin{align*}
2{\gamma}\left(F(y_{k+1})+ {{\far{\eta_k}}}H(y_{k+1})\right)\fyy{^\top}\left(  x-y_{k+1}  \right) &\leq 2{\gamma}\left(F(x)+ {{\far{\eta_k}}}H(x)\right)\fyy{^\top}\left( x -y_{k+1}\right) \\
&- 2{\gamma}\far{\eta_k}\mu_H\|y_{k+1}-x\|^2.
\end{align*}
We can also write $-2\|y_{k+1}-x\|^2 \leq -\|x_k-x\|^2+ 2\|x_k-y_{k+1}\|^2.$ From \far{equation}~\eqref{eqn:EG_lemma_mm_last} and the preceding two relations, we obtain 
\begin{align*} 
2{\gamma}\left(F(x)+\far{\eta_k} H(x)\right)\fyy{^\top}\left(y_{k+1}-x\right)   & \leq (1-{\gamma}\far{\eta_k}\mu_H)\|x_k - x\|^2   -\|x_{k+1} - x\|^2 \\
&-(1-2{\gamma}\far{\eta_k}\mu_H-2{\gamma}^2(L_F^2+\far{\eta_k}^2L_H^2)) \|x_k-y_{k+1}\|^2.
\end{align*}
From the assumptions, $1-2{\gamma}\far{\eta_k}\mu_H-2{\gamma}^2(L_F^2+\far{\eta_k}^2L_H^2)\geq 0$. This completes the proof.  

\noindent (ii) From the strong convexity of $f$, we have 
$$ \nabla f(y_{k+1})\fyy{^\top}(x-y_{k+1}) +\frac{\mu}{2}\|y_{k+1}-x\|^2\leq f(x) - f(y_{k+1}). $$
From the monotonicity of $F$ and the preceding relation, we have
\begin{align*}
2{\gamma}\left(F(y_{k+1})+ {{\far{\eta_k}}}\nabla f(y_{k+1})\right)\fyy{^\top}\left(  x-y_{k+1}  \right) &\leq 2{\gamma} F(x)\fyy{^\top}\left( x -y_{k+1}\right)\\
& +2\gamma\far{\eta_k} (f(x)-f(y_{k+1})) - {\gamma}\far{\eta_k}\mu\|y_{k+1}-x\|^2.
\end{align*}
We can also write $-\|y_{k+1}-x\|^2 \leq -0.5\|x_k-x\|^2+ \|x_k-y_{k+1}\|^2.$ From \far{equation}~\eqref{eqn:EG_lemma_mm_last} and the preceding two relations, we obtain 
\begin{align*} 
2{\gamma}( F(x)\fyy{^\top}\left(y_{k+1}-x\right) +\far{\eta_k}(f(y_{k+1})-f(x)))    & \leq(1-0.5{\gamma}\far{\eta_k}\mu)\|x_k - x\|^2   -\|x_{k+1} - x\|^2 \\
-(1-{\gamma}\far{\eta_k}\mu-2{\gamma}^2&(L_F^2+\far{\eta_k}^2L^2)) \|x_k-y_{k+1}\|^2.
\end{align*}
But we assumed that $1-{\gamma}\far{\eta_k}\mu-2{\gamma}^2(L_F^2+\far{\eta_k}^2L^2)\geq 0$. This completes the proof.
\subsection{Proof of Lemma~\ref{lem:ave_REG}}\label{sec:proof_lemma_lem:ave_REG}
\far{
We use induction on $K\geq 1$. For $K=1$, we may write $\sum_{k=0}^0 {\lambda_{k,1}} y_{k+1} = \lambda_{0,1}y_{1} = y_{1},$ where we used $\lambda_{0,1}=1$. From Algorithm~\ref{alg:IR-EG-s} and the initialization $\Gamma_0 =0$, we have
$\bar y_{1}:= (\Gamma_0 \bar y_0+\eta_0\theta_0 y_{1})/\Gamma_{1}=(0+\eta_0\theta_0 y_{1})/(\Gamma_{0}+\eta_0\theta_0) = y_{1}.$ From the preceding two relations, the hypothesis statement holds for $K=1$. Next, suppose the relation holds for some $K\geq 1$. From $\Gamma_{K}=\sum_{k=0}^{K-1}\eta_k\theta_k$, we may write
\begin{align*}
	\bar{y}_{K+1} &= \frac{\Gamma_K\bar{y}_K + \eta_K\theta_K y_{K+1}}{\Gamma_{K+1}} 
	 = \frac{\left(\sum_{k=0}^{K-1}\eta_k\theta_k\right)\sum_{k=0}^{K-1} \lambda_{k,K} y_{k+1}+ \eta_K\theta_K y_{K+1}}{\Gamma_{K+1}}\\
	 &= \frac{\sum_{k=0}^{K}\eta_k\theta_k y_{k+1}}{\sum_{j = 0}^{K}\eta_j\theta_j} = \sum_{k=0}^{K}\left(\tfrac{\eta_k\theta_k }{\sum_{j = 0}^{K}\eta_j\theta_j}\right)y_{k+1}= \sum_{k=0}^{K} \lambda_{k,K+1} y_{k+1},
	\end{align*}
implying that the induction hypothesis holds for $K+1$. In view of $\sum_{k=0}^{K-1} \lambda_{k,K}=1$, under the convexity of the set $X$, we have $\bar y_K \in X$.  
}
\subsection{Proof of Theorem~\ref{theorem: H is SC}}\label{sec_app_theorem: H is SC}
\far{ {\bf [Case 1]} First, we show that $\{\eta_k\}$ and  $\gamma$ meet the condition of Theorem~\ref{prop:unify_IR-EG_sm}. For any $k \geq 0$, we have 
\begin{align*}
&{\gamma}^2L_F^2+{\gamma}\eta_k\mu_H+{\gamma}^2\eta_k^2L_H^2 \leq 
 0.25+{\gamma}\eta_0\mu_H+{\gamma}^2\eta_0^2L_H^2 \\
&\leq 0.25 + \frac{1}{\eta_{0,l}} + \frac{L_H^2}{\mu_H^2(5L_H\mu_H^{-1})^2}
\leq  0.25 + 0.2+0.04 \leq 0.5. 
\end{align*}
Next, we estimate the terms $\sum_{k=0}^{K-1}\eta_k\theta_k $ and $\sum_{k=0}^{K-1}\eta_k^2\theta_k $. From the update rule of $\eta_k$, we obtain  
\begin{align*}
 \prod_{t=0}^{k}(1- \gamma \eta_t \mu_H ) =  \prod_{t=0}^{k}\left(1- \frac{1}{t+\eta_{0,l}} \right) = \prod_{t=0}^{k}\left(\frac{t+\eta_{0,l}-1}{t+\eta_{0,l}} \right) = \frac{\eta_{0,l}-1}{k+\eta_{0,l}}. 
\end{align*}
This implies that $\theta_k \triangleq \far{\frac{1}{\prod_{t=0}^{k}(1- \gamma \eta_t \mu_H )}}=\frac{k+\eta_{0,l}}{\eta_{0,l}-1}$, for $k\geq 0$. We obtain 
\begin{align}
& \sum_{k=0}^{K-1}\eta_k \theta_k = \sum_{k=0}^{K-1}\frac{\eta_{0,u}}{\eta_{0,l}-1} = \frac{\eta_{0,u} K}{\eta_{0,l}-1} =\frac{K}{\gamma  (5L_H-\mu_H)} \label{eqn:theta_sums_IREG1}\\ 
\hbox{and} \quad  & \sum_{k=0}^{K-1}\eta_k^2 \theta_k = \sum_{k=0}^{K-1}\frac{\eta_{0,u}^2}{(\eta_{0,l}-1)(k+\eta_{0,l})}\leq \frac{\eta_{0,u}^2}{(\eta_{0,l}-1)}\left(\eta_{0,l}^{-1}+\ln\left(\frac{K+\eta_{0,l}-1}{\eta_{0,l}}\right)\right),\label{eqn:theta_sums_IREG2}
\end{align}
where we invoked~\cite[Lemma 9]{yousefian2017smoothing} in the preceding relation.  \fyy{Note that, in view of \eqref{eqn:theta_sums_IREG1} and \eqref{eqn:theta_sums_IREG2}, the conditions of \cref{prop:unify_IR-EG_sm} (iii) are met, implying that $\{ \bar y_k\}$ converges to the unique solution of problem~\cref{eqn:bilevelVI}.}

\noindent  {\bf (1-i)} This result can be obtained by invoking Theorem~\ref{prop:unify_IR-EG_sm} and utilizing \eqref{eqn:theta_sums_IREG1}. 
 
 \noindent {\bf (1-ii)} This result follows by invoking Theorem~\ref{prop:unify_IR-EG_sm} and utilizing \eqref{eqn:theta_sums_IREG1} and \eqref{eqn:theta_sums_IREG2}. 

\noindent {\bf (1-iii)} This result is obtained by applying Lemma~ \ref{lem:bilevelVI_gap_ws} and using the bound in (1-ii).

\noindent {\bf [Case 2]} First, we show that ${\gamma}^2L_F^2+{\gamma}\eta\mu_H+{\gamma}^2\eta^2L_H^2\leq 0.5$. From $\gamma \leq \frac{1}{2L_F}$ and $\eta := \frac{(p+1)\ln(K)}{ \gamma\mu_H K}$, we have 
\begin{align*}
 {\gamma}^2L_F^2+{\gamma}\eta\mu_H+{\gamma}^2\eta^2L_H^2 &\leq 0.25 + \tfrac{(p+1)\ln(K)}{  K} + \tfrac{(p+1)^2\ln(K)^2L_H^2}{ \mu_H^2 K^2} \leq 0.49 \leq 0.5,
 \end{align*}
where we used ${K}/{\ln(K)} \geq {5(p+1) L_H/\mu_H}$. This implies that the condition of Theorem~\ref{prop:unify_IR-EG_sm} is met.   

\noindent {\bf  (2-i)} Invoking Theorem~\ref{prop:unify_IR-EG_sm} (i), with a constant $\eta$, we have 
\begin{align*}
 \ssrtwo{- B_H\, \mbox{dist}\left(\bar{y}_K,\mbox{SOL}(X,F)\right) \leq } \mbox{Gap}\left(\bar{y}_K,\mbox{SOL}(X,F),H\right) \leq (\eta\gamma)^{-1} D_X^2/\left(\textstyle\sum_{j=0}^{K-1} \theta_j\right)\leq \frac{D_X^2}{ \gamma \eta}  (1- \gamma \eta \mu_H)^K\ssrtwo{,}
\end{align*}
where we used $\sum_{j=0}^{K-1} \theta_j \geq \theta_{K-1}=\frac{1}{(1-\gamma\eta\mu_H)^K}$. Invoking $\eta := \frac{(p+1)\ln(K)}{ \gamma\mu_H K}$, we have
\begin{align*}
&(1-\eta\gamma\mu_H)^K = \left(1-\tfrac{(p+1)\ln(K)}{K}\right)^K
=  \left(\left(1-{(p+1)\ln(K)}/{K}\right)^{\tfrac{K}{(p+1)\ln(K)}}\right)^{\ln(K^{(p+1)})}\\
&\leq \left(\lim_{K\to \infty}\left(1-{(p+1)\ln(K)}/{K}\right)^{\tfrac{K}{(p+1)\ln(K)}}\right)^{\ln(K^{(p+1)})} = e^{{-\ln(K^{(p+1)})}}
= \tfrac{1}{K^{(p+1)}}.
\end{align*}
  The \ssrtwo{bounds} in (2-i) follows from the two preceding relations.  
  
\noindent {\bf  (2-ii)}  Invoking Theorem~\ref{prop:unify_IR-EG_sm} (ii) and $\sum_{j=0}^{K-1} \theta_j \geq \theta_{K-1}=\frac{1}{(1-\gamma\eta\mu_H)^K}$, we obtain $ 0 \leq \mbox{Gap}\left(\bar{y}_K,X,F\right)       \leq \frac{D_X^2}{ \gamma}  (1- \gamma \eta \mu_H)^K+  \sqrt{2} C_HD_X\eta$. Substituting $(1-\eta\gamma\mu_H)^K$ by $\tfrac{1}{K^{(p+1)}}$ and $\eta$ by $\frac{(p+1)\ln(K)}{ \gamma\mu_H K}$, we arrive at (2-ii). 
 
 \noindent {\bf (2-iii)}  This result follows by applying Lemma~ \ref{lem:bilevelVI_gap_ws} and using the bound in (2-ii).



\noindent {\bf [Case 3]}   {\bf (3-i)}  Consider \eqref{ineq:EG_lemma_smonotone}. Substituting $x:=x^*$, for a constant regularization parameter $\eta$, we have
\begin{align*}
2\gamma\left(F(x^*)+\eta H(x^*)\right)\fyy{^\top}\left(y_{k+1}-x^*\right) &\leq   (1-\gamma\eta\mu_H)\|x_k - x^*\|^2 -  \|x_{k+1} -  x^*\|^2.
\end{align*}
In view of $x^* \in X^*_F$ and Definition~\ref{def:weaksharp}, we have $F(x^*)\fyy{^\top}(y_{k+1}-x^*) \geq \alpha\, \mbox{dist}(y_{k+1} ,X^*_F)$. Invoking the Cauchy-Schwarz inequality, we have
\begin{align*} 
H(x^* ) \fyy{^\top}\left(y_{k+1}-x^* \right)   & =   H(x^* ) \fyy{^\top}\left(y_{k+1}- \Pi_{X^*_F}[y_{k+1}]+   \Pi_{X^*_F}[y_{k+1}] -x^* \right)\\
& \geq  -\|H(x^*)\|\|y_{k+1}- \Pi_{X^*_F}[y_{k+1}]\|,
\end{align*}
where we used $ H(x^* ) \fyy{^\top}\left(   \Pi_{X^*_F}[y_{k+1}] -x^* \right) \geq 0$ in view of $ \Pi_{X^*_F}[y_{k+1}] \in X^*_F$ and that $x^*$ solves $\mbox{VI}(X^*_F,H)$.  Thus, we obtain  
\begin{align*} 
2{\gamma}\left(F(x^* )+\eta H(x^* )\right)\fyy{^\top}\left(y_{k+1}-x^* \right)  \geq 2{\gamma} \left( \alpha  - \eta\|H(x^*)\|\right)\mbox{dist}(y_{k+1} ,X^*_F).
\end{align*}
This implies that 
\begin{align*}
2{\gamma} \left( \alpha  - \eta \|H(x^*)\|\right)\, \mbox{dist}(y_{k+1} ,X^*_F)&\leq   (1-\gamma\eta\mu_H)\|x_k - x^*\|^2 -  \|x_{k+1} -  x^*\|^2.
\end{align*}
Multiplying both sides of the preceding inequality by $\theta_k=1/{(1-\gamma \eta\mu_H)}^{k+1}$, for $k\geq 0$, we obtain 
\begin{align*}
2{\gamma}(\alpha- \eta \|H(x^*)\| )\theta_k\, \mbox{dist}(y_{k+1} ,X^*_F) &\leq    \frac{ \|x_k - x^* \|^2}{(1-\gamma \eta\mu_H)^{k}}  - \frac{\|x_{k+1} - x^*\|^2}{ (1-\gamma \eta\mu_H)^{k+1}} .
\end{align*}
Summing the both sides over $k=0,\ldots, K-1$, for $K\geq 1$, we obtain 
\begin{align*}
2{\gamma}(\alpha- \eta \|H(x^*)\| )\textstyle\sum_{k=0}^{K-1}\theta_k \mbox{dist}(y_{k+1} ,X^*_F)&\leq      \|x_0 - x^*\|^2  -  \theta_{K-1}\|x_{K} - x^*\|^2  .
\end{align*}
Invoking the Jensen's inequality, we obtain    
\begin{align*}
  2{\gamma}(\alpha- \eta \|H(x^*)\| ) (\textstyle\sum_{j=0}^{K-1} \theta_j)\mbox{dist}(\bar{y}_K,X^*_F) \leq    \|x_0 - x^*\|^2   .
\end{align*}
Invoking $\eta \leq \frac{\alpha}{2\|H(x^*)\|}$ and $\textstyle\sum_{j=0}^{K-1} \theta_j \geq \theta_{K-1} =\frac{1}{(1-\gamma\eta\mu_H)^K}$, we obtain 
\begin{align*}
  {\gamma}\alpha\, \mbox{dist}(\bar{y}_K,X^*_F) \leq   (1-\gamma \eta \mu_H)^K \|x_0 - x^*\|^2     .
\end{align*}

\noindent  {\bf (3-ii)} From the preceding relation and invoking Lemma~\ref{lem:bilevelVI_gap_ws},  we obtain $$ \mbox{Gap}\left(\bar{y}_K, X^*_F,H\right)   \geq - \tfrac{B_H \|x_0 - x^*\|^2}{{\gamma}\alpha} (1-\gamma \eta \mu_H)^K  . $$  Combining this with Theorem~\ref{prop:unify_IR-EG_sm} (i), we obtain the result in (3-ii). 
}
\far{\subsection{Proof of Corollary~\ref{theorem: H is SC2}}\label{sec:proof_lemma_theorem: H is SC2}
{\bf [Case 1]} The proofs of (1-i) and (1-ii) are analogous to those of Theorem~\ref{theorem: H is SC} (1-i) and (1-ii), and the details are omitted. Instead, we highlight some minor distinctions. 

 \noindent {\bf (1-i)} Consider~\eqref{ineq:EG_lemma_smonotone2}. In the proof of Theorem~\ref{prop:unify_IR-EG_sm} (i), let $x^*_F \in \mbox{SOL}(X,F)$ be the unique optimal solution to problem~\eqref{eqn:optVI}, denoted by $x^*$. Following similar steps in that proof, we obtain the bound in (1-i). 

 \noindent {\bf (1-ii)} The bound in (1-ii) follows from Theorem~\ref{theorem: H is SC} (1-ii), where in view of the difference between~\eqref{ineq:EG_lemma_smonotone} and~\eqref{ineq:EG_lemma_smonotone2} in the coefficient of $\|x_k-x\|^2$, we substitute $\mu_H:=0.5\mu$ and $L_H:=L$. 

 \noindent {\bf (1-iii)} This follows directly by invoking Lemma~ \ref{lem:bilevelVI_gap_ws} and the upper bounds in (1-i) and (1-ii).

 \noindent   {\bf [Case 2]} First, we show that ${\gamma}^2L_F^2+0.5{\gamma}\eta\mu+{\gamma}^2\eta^2L^2\leq 0.5$. We have 
 \begin{align*}
 {\gamma}^2L_F^2+0.5{\gamma}\eta\mu+{\gamma}^2\eta^2L^2 &\leq 0.25 + \tfrac{(p+1)\ln(K)}{  K} + \tfrac{4(p+1)^2\ln(K)^2L^2}{ \mu^2 K^2} \leq 0.39 \leq 0.5,
 \end{align*}
 where we used the assumption ${K}/{\ln(K)} \geq 10(p+1) {L}/{\mu} $.  
 The proofs of (2-i) and (2-ii) are analogous to those of Theorem~\ref{theorem: H is SC} (2-i) and (2-ii), and are omitted. Some minor distinctions are highlighted next. 

 \noindent {\bf (2-i)} Consider~\eqref{ineq:EG_lemma_smonotone2}. In the proof of Theorem~\ref{prop:unify_IR-EG_sm} (i), let $x^*_F \in \mbox{SOL}(X,F)$ be the unique optimal solution to problem~\eqref{eqn:optVI}, denoted by $x^*$. Following similar steps in that proof, we obtain the bound in (2-i). 

 \noindent {\bf (2-ii)} The bound in (2-ii) follows from Theorem~\ref{theorem: H is SC} (2-ii), where in view of the difference between~\eqref{ineq:EG_lemma_smonotone} and~\eqref{ineq:EG_lemma_smonotone2} in the coefficient of $\|x_k-x\|^2$, we substitute $\mu_H:=0.5\mu$ and $L_H:=L$. 

 \noindent {\bf (2-iii)} This follows directly by invoking Lemma~ \ref{lem:bilevelVI_gap_ws} and the upper bounds in (2-i) and (2-ii).

    \noindent   {\bf [Case 3]}   {\bf (3-i)} The proof follows directly from Corollary~\ref{theorem: H is SC2} (3-i) where we substitute $\mu_H:=0.5\mu$.
     
     \noindent {\bf (3-ii)} Consider \eqref{ineq:EG_lemma_smonotone2}. In the proof of Theorem~\ref{prop:unify_IR-EG_sm} (i), let $x^*_F \in \mbox{SOL}(X,F)$ be the unique optimal solution to problem~\eqref{eqn:optVI}, denoted by $x^*$. Following similar steps in that proof, we obtain $f(\bar{y}_K) -f(x^*) \leq \frac{\|x_0-x^*\|^2}{2\eta \gamma \sum_{j=1}^{K-1}\theta_j}$ where $\theta_j = \frac{1}{(1-0.5\gamma\eta\mu)^{j+1}}$ for any $j\geq 0$.
 Using $\sum_{j=1}^{K-1}\theta_j \geq \theta_{K-1}$, we obtain  $f(\bar{y}_K) -f(x^*) \leq \frac{\|x_0-x^*\|^2}{2\eta \gamma}(1-0.5\gamma\eta\mu)^K$. From this bound, Lemma~ \ref{lem:bilevelVI_gap_ws}, and (3-i), we have 
\begin{align*}
\frac{\mu}{2}\|\bar{y}_K-x^*\|^2 &\leq f(\bar{y}_K) -f(x^*)+ \|\nabla f(x^*)\|\,\mbox{dist}(\bar{y}_K,\mbox{SOL}(X,F)) \\
& \leq  \tfrac{\|x_0-x^*\|^2}{2\eta \gamma}(1-0.5\gamma\eta\mu)^K +\|\nabla f(x^*)\|\tfrac{  \|x_0 - x^*\|^2}{{\gamma}\alpha}(1-0.5\gamma \eta \mu )^K.
\end{align*}
Using the assumption $\|\nabla f(x^*)\| \leq \frac{\alpha}{2\eta}$, we obtain the result. 
\subsection{Proof of Lemma~\ref{subsection:error-metric}}\label{sec:proof_lemma_subsection:error-metric}
From Definition~\ref{def:res_map} and~Definition~\ref{def:IPREG_terms}, we have
$$G_{{1}/{\hat\gamma}}(\hat x_k) = \tfrac{1}{\hat\gamma}{\left(\hat x_k -   \Pi_{X^*_F}\left[\hat x_k - \hat\gamma \nabla{f\left(\hat x_k\right)}\right] \right)} = \tfrac{1}{\hat\gamma}{\left(\hat x_k - {\hat x_{k+1} + \delta_k}\right)} .$$
The result follows by invoking $\|u+v\|^2 \leq 2\|u\|^2+2\|v\|^2$ for all $u,v \in \mathbb{R}^n$.
}
\subsection{Proof of Proposition~\ref{prop:upper_bound_gradient}}\label{sec:proof_lemma_prop:upper_bound_gradient}
Note that $y:=\Pi_{{X^*_F}}[\hat{x}_k] \in {X^*_F}$. In view of this, from Lemma~\ref{lem:Projection theorm}, for any $k\geq 0$ we have  
\begin{equation*}
   \left( \Pi_{X^*_F}[z_k] - \Pi_{X^*_F}[\hat x_k]\right) \fyy{^\top}  \left(z_k- \Pi_{X^*_F}[z_k]\right) \geq 0.
\end{equation*}
Recall that $\delta _k =\hat x_{k+1} -\Pi_{X^*_F}[z_k]  $ and
$e_k = \hat x_k - \Pi_{X^*_F}[\hat x_k]$. We obtain
\begin{equation*}
\left(\left(\hat x_{k+1} - \delta_k\right) -\left(\hat x_k - e_k\right) \right)\fyy{^\top} \left( z_k - \left(\hat x_{k+1} - \delta_k \right)\right) \geq 0.
\end{equation*}
From the preceding inequality, we have
\begin{equation}
\left(\hat x_{k+1} - \hat x_{k}\right) \fyy{^\top} \left(z_k - \hat x_{k+1}\right) + \left(\hat x_{k+1} - \hat x_{k}\right) \fyy{^\top} \delta_k +   \left(e_k - \delta_k\right) \fyy{^\top} \left(z_k - \hat x_{k+1} + \delta_k\right)
\geq 0.
\end{equation}
From Algorithm~\ref{alg:ncvx-IR-GD}, we have that $z_k = \hat x_k - \hat{\gamma} \nabla f(\hat x_k)$. We obtain
\begin{align*}
&\underbrace{\left(\hat x_{k+1} - \hat x_{k}\right) \fyy{^\top} \left(\hat x_k - \hat{\gamma} \nabla f(\hat x_k) - \hat x_{k+1}\right) }_\text{Term 1} + \underbrace{\left(\hat x_{k+1} - \hat x_{k}\right) \fyy{^\top} \delta_k}_\text{Term 2} \\
&+  \underbrace{  \left(e_k - \delta_k\right) \fyy{^\top} \left(\hat x_k - \hat{\gamma} \nabla f(\hat x_k) - \hat x_{k+1}\right) }_\text{Term 3} +  \underbrace{\left( e_k -\delta_k \right) \fyy{^\top}  \delta_k}_\text{Term 4}
\geq 0.
\end{align*}
Here we expand Term~1, apply the inequality $u\fyy{^\top} v \leq \frac{1}{2\theta} \|u\|^2 + \frac{\theta}{2} \|v\|^2$ in bounding Term~2 and  Term~3, for some arbitrary $\theta>0$, and apply the inequality $u\fyy{^\top} v \leq \frac{1}{2} \|u\|^2 + \frac{1}{2} \|v\|^2$ in bounding Term~4. We have
\begin{align*}
&\underbrace{-\hat{\gamma}\nabla f(\hat x_k)\fyy{^\top}(\hat x_{k+1}-\hat x_{k}) - \|\hat x_{k+1}-\hat x_k\|^2}_\text{Term 1} 
 +\underbrace{\frac{1}{2\theta} \|\hat x_{k+1} - \hat x_k\|^2 +\frac{\theta}{2}  \|\delta_k\|^2}_{  \text{Term 2} \leq}  
\\ &
 +\underbrace{\theta \| e_k-\delta_k \|^2 + \frac{1}{2\theta}\|\hat x_{k+1} - \hat x_k\|^2 + \frac{\hat{\gamma}^2}{2\theta}\|\nabla f(\hat x_k)\|^2   }_{  \text{Term 3} \leq} 
+ \underbrace{\frac{1}{2}\| e_k -\delta_k\|^2 + \frac{1}{2}\|\delta_k\|^2 }_{  \text{Term 4} \leq}
\geq 0.
\end{align*}
Hence, we have 
\begin{align*}
 \hat \gamma\nabla f(\hat x_k)\fyy{^\top}(\hat x_{k+1}-\hat x_{k}) &\leq \left(\theta^{-1} -1 \right)\|\hat x_{k+1}-\hat x_k\|^2+ \frac{\hat\gamma^2}{2\theta } \| \nabla f(\hat x_k)\|^2 + 0.5\left( \theta  +1 \right)  \| \delta _k\|^2 \\
 & + \left(\theta +0.5 \right ) ( 2\| e_k \|^2+2\| \delta _k\|^2).
\end{align*}
Let us choose {$\theta := \frac{2}{\hat\gamma L }$}. Note that from the assumption $\hat\gamma\leq \tfrac{1}{2L}$, we have $\theta\geq 4$. Before we substitute $\theta$, using $\theta\geq 4$ we obtain
\begin{align*}
 \hat \gamma\nabla f(\hat x_k)\fyy{^\top}(\hat x_{k+1}-\hat x_{k}) &\leq \left(\theta^{-1} -1 \right)\|\hat x_{k+1}-\hat x_k\|^2+ \frac{\hat\gamma^2C_f^2}{2\theta }   + 5\theta ( \| \delta _k\|^2  + \| e_k \|^2) .
\end{align*}
Dividing both sides by $\hat{\gamma}$ and substituting $\theta := \frac{2}{\hat\gamma L }$ we obtain
\begin{align*}
\nabla f(\hat x_k)\fyy{^\top}(\hat x_{k+1}-\hat x_{k}) \leq
\left(\tfrac{L}{2}-\tfrac{1}{\hat{\gamma}}\right  )\|\hat x_{k+1}-\hat x_k\|^2+ \tfrac{L\hat\gamma^2C_f^2}{4 }  +\tfrac{10}{L\hat{\gamma}^2} ( \| \delta _k\|^2  + \| e_k \|^2).
\end{align*}
From the $L$-smoothness property of $f$, we have 
\begin{equation*}
f(\hat x_{k+1}) \leq f(\hat x_k) + \nabla f(\hat x_k)\fyy{^\top} (\hat x_{k+1} - \hat x_k) + \tfrac{L}{2} \|\hat x_{k+1} - \hat x_k\|^2.
\end{equation*}
From the two preceding relations, we obtain
%
%
%
%
\begin{equation*}
\left(\tfrac{1}{\hat{\gamma}}-L\right) \|\hat x_{k+1}-\hat x_k\|^2 \leq f(\hat x_k) - f(\hat x_{k+1})+ \tfrac{L\hat\gamma^2C_f^2}{4 } +\tfrac{10}{L\hat{\gamma}^2} ( \| \delta _k\|^2  + \| e_k \|^2).
\end{equation*}
Note that $\tfrac{1}{2\hat{\gamma}} \leq \left(\tfrac{1}{\hat{\gamma}}-L\right) $. Utilizing this bound, then adding $\tfrac{1}{2\hat{\gamma}} \| \delta_k\|^2$ to  both sides, and invoking Lemma~\ref{subsection:error-metric}, we obtain 
\begin{equation*}
 \tfrac{\hat\gamma}{4}  \| G_{{1}/{\hat\gamma}}(\hat x_k) \|^2  \leq f(\hat x_k) - f(\hat x_{k+1})+ \tfrac{L\hat\gamma^2C_f^2}{4 } +\tfrac{20}{L\hat{\gamma}^2} ( \| \delta _k\|^2  + \| e_k \|^2),
\end{equation*}
where we used $\tfrac{1}{2\hat{\gamma}} \leq \tfrac{10}{L\hat{\gamma}^2} $ in view of the assumption $\hat\gamma\leq \tfrac{1}{2L}$.
%
%
%
%
%
\subsection{Proof of Proposition~\ref{Upper bound for e_k}}\label{sec_app_Upper bound for e_k}
\noindent \far{(1-i)} Invoking $\|e_k\| = \mbox{dist}(\hat{x}_k,\mbox{SOL}(X,F))$ and Lemma~\ref{lem:bilevelVI_gap_ws}, we have
\begin{align}\label{eqn:e_k_gap}
\|e_k\| \leq \ses{\sqrt[\mathcal{M}]{\left(\alpha^{-1}\mbox{Gap}(\hat{x}_k,X,F)\right)}}, \qquad \hbox{for all } k\geq 0.
\end{align} 
From Algorithm~\ref{alg:ncvx-IR-GD}, we have $\hat{x}_k=\bar y_{{k-1}, T_{k-1}}$ for $k \geq 1$, where $\bar y_{{k-1}, T_{k-1}}$ is the output of \far{IR-EG$_{{\texttt{s,m}}}$} after $T_{k-1}$ number of iterations employed for solving the projection problem \far{$\min_{x \in X^*_F} \ \tfrac{1}{2}\|x-z_k\|^2$. Consider the projection problem let} $k\geq 1$ be fixed. We apply Corollary~\ref{theorem: H is SC2} (case 2) where we choose $p:=2$ and $K:=T_{k-1}$. Note that the condition $\frac{T_{k-1}}{\ln(T_{k-1})} \geq 10(p+1)L/\mu=30$ is met as $L=\mu=1$ for the objective function of the projection problem and  $T_{k-1} \geq 151$ (see Algorithm~\ref{alg:ncvx-IR-GD}). Let us define $h(x) \triangleq \tfrac{1}{2}\|x-z_k\|^2$. Therefore, we have 
$${\mbox{Gap}}\left(\bar y_{{k-1}, T_{k-1}},X,F\right)       \leq\left(\tfrac{D_X^2}{\gamma}\right)\tfrac{1}{ T_{k-1}^3}+\left(\tfrac{6\sqrt{2} C_h  D_X}{ \gamma } \right)\tfrac{\ln(T_{k-1})}{T_{k-1}}, \qquad \hbox{for all } k\geq 1,$$
where $C_h \triangleq \sup_{x \in X}\|\nabla h(x)\|$. Recalling $z_k \triangleq \hat x_k - \hat{\gamma} \nabla f(\hat x_k)$, we have 
$$C_h = \sup_{x \in X} \|x-z_k\| \leq  \sup_{x \in X} \|x-\hat{x}_k\| + \hat{\gamma}\| \nabla f(\hat x_k)\| \leq  \sqrt{2}D_X + \hat{\gamma}C_f.$$ From the preceding relation and the equation~\eqref{eqn:e_k_gap}, we have for $k \geq 1$ 
\ses{\begin{align*}
\|e_k\|  &\leq \sqrt[\mathcal{M}]{ \alpha^{-1}\left(
 \tfrac{D_X^2}{\gamma T_{k-1}^3} +\tfrac{\left(12D_X^2+6\sqrt{2}C_fD_X\hat{\gamma}\right)\ln(T_{k-1})}{\gamma T_{k-1}} \right)  }\\&
\leq 
\sqrt[\mathcal{M}]{
\left(\tfrac{12.5 D_X^2+6\sqrt{2}C_fD_X\hat{\gamma}}{\alpha\gamma}\right) \tfrac{\ln(T_{k-1})}{T_{k-1}}
}
,
\end{align*}}  where we used $\tfrac{1}{ T_{k-1}^3}\leq\tfrac{0.5\ln(T_{k-1})}{T_{k-1}}$ in view of $T_{k-1} \geq 151$. Next, we relate the preceding bound to $T_k$ (This is done to simplify the presentation of the next results). Note that from \ses{$T_k:=\max\{k^{1.5 \mathcal{M}},151\}$},  $\frac{\ln(T_{k-1})}{T_{k-1}} \leq 2\frac{\ln(T_{k-1})}{T_{k}} \leq 2\frac{\ln(T_{k})}{T_{k}}$ for all $k \geq 1$.


\noindent \far{(1-ii)} Let us define $\far{y^*_{k,h}}\triangleq \Pi_{X^*_F}[z_k]$. Invoking Corollary~\ref{theorem: H is SC2} for the projection problem (where $L=\mu=1$, $p=2$) and using the bound in part (i), we obtain
\begin{align}\label{bound_delta_in_proof_lemma}
 \|\bar{y}_{k,T_{k}}-\far{y^*_{k,h}}\|^2 &\leq  \left(\tfrac{2 D_X^2}{ 3  }\right)\tfrac{1}{\ln(T_{k})T_{k}^2}    +   2\|\nabla h(\far{y^*_{k,h}})\| \ses{
\sqrt[\mathcal{M}]{
 \left(\tfrac{25 D_X^2+17C_fD_X\hat{\gamma}}{\alpha\gamma}\right) \tfrac{\ln(T_{k})}{T_{k}} 
} 
}.
\end{align}
 Note that $\bar{y}_{k,T_{k}}-\far{y^*_{k,h}} = \hat x_{k+1} -\Pi_{X^*_F}[z_k] =\delta_k$ and $\nabla h(\far{y^*_{k,h}})=\far{\far{y^*_{k,h}}-z_k}$.  Next, we derive a bound on $\|\far{y^*_{k,h}}-z_k\|$. Using $z_k=\hat{x}_k -\hat{\gamma}\nabla f(\hat{x}_k)$, we have
\begin{align*}
  \|z_k - \far{y^*_{k,h}}\|  &= \| z_k - \Pi_{X^*_F}[z_k] \|  \leq \| z_k - \Pi_{X^*_F}[\hat{x}_k] \| = \| \hat{x}_k -\hat{\gamma}\nabla f(\hat{x}_k) - \Pi_{X^*_F}[\hat{x}_k] \| \\
& \leq \| \hat{x}_k - \Pi_{X^*_F}[\hat{x}_k] \|+\hat{\gamma}\|\nabla f(\hat{x}_k)\|  \leq \|e_k\| + \hat{\gamma} C_f,
\end{align*}
where we used Definition~\ref{def:IPREG_terms} and  $\hat{x}_k \in X$. This is because $\hat{x}_k = y_{k-1,T_{k-1}}$ and that, from Lemma~\ref{lem:ave_REG},  $y_{k-1,T_{k-1}} \in X$. Thus, from \eqref{bound_delta_in_proof_lemma} we have 
\begin{align*} 
 \|\delta_k \|^2 &\leq  \left(\tfrac{ 2D_X^2}{ 3  }\right)\tfrac{1}{\ln(T_{k})T_{k}^2}    + 2 (\|e_k\| + \hat{\gamma} C_f) \ses{\sqrt[\mathcal{M}]{
 \left(\tfrac{25 D_X^2+17C_fD_X\hat{\gamma}}{\alpha\gamma}\right) \tfrac{\ln(T_{k})}{T_{k}} 
} }.
\end{align*}
The result is obtained from part (i) and  $\tfrac{1}{\ln(T_{k})T_{k}^2} \leq \tfrac{(\ln(T_{k}))^2}{T_{k}^2} $ for $k \geq 1$.

\noindent  {(2-i)} \far{We consider applying Corollary~\ref{theorem: H is SC2} (case 3) where we choose $K:=T_{k-1}$.  We first show that the condition $\eta \leq \frac{\alpha}{2\|\nabla h(y^*_{k,h})\|}$ is met for all $k\geq 1$ for Algorithm~\ref{alg:ncvx-IR-GD}. Note that $\nabla h(y^*_{k,h})=y^*_{k,h}-z_k$.  Next, we derive a bound on $\|y^*_{k,h}-z_k\|$. Using $z_k=\hat{x}_k -\hat{\gamma}\nabla f(\hat{x}_k)$, we have
\begin{align*}
  \|z_k - y^*_{k,h}\|  &= \| z_k - \Pi_{X^*_F}[z_k] \|  \leq \| z_k - \Pi_{X^*_F}[\hat{x}_k] \| = \| \hat{x}_k -\hat{\gamma}\nabla f(\hat{x}_k) - \Pi_{X^*_F}[\hat{x}_k] \| \\
& \leq \| \hat{x}_k - \Pi_{X^*_F}[\hat{x}_k] \|+\hat{\gamma}\|\nabla f(\hat{x}_k)\|  \leq \sqrt{2}D_X+  \tfrac{C_f}{2L},
\end{align*}
where we used Definition~\ref{def:IPREG_terms}, $\hat{\gamma} \leq \frac{1}{2L}$, and  $\hat{x}_k \in X$.  This implies due to $\eta \leq \frac{\alpha L}{2\sqrt{2}D_X L+  C_f}$, the condition $\eta \leq \frac{\alpha}{2\|\nabla h(y^*_{k,h})\|}$ is met for all $k\geq 1$.  Thus,  from Corollary~\ref{theorem: H is SC2} we have 
$$\|e_k\| = \mbox{dist}(\bar y_{{k-1}, T_{k-1}},X^*_F) \leq  \tfrac{ 2D_X^2}{{\gamma}\alpha}(1-0.5\gamma \eta  )^{T_{k-1}}\qquad \hbox{for all } k\geq 1.$$

\noindent {(2-ii)} Note that $\bar{y}_{k,T_{k}}-y^*_{k,h} = \hat x_{k+1} -\Pi_{X^*_F}[z_k] =\delta_k$.   Invoking Corollary~\ref{theorem: H is SC2} we have $\|\delta_k\|^2 = \|\bar{y}_{k,T_{k}}-y^*_{k,h}\|^2 \leq \tfrac{4 D_X^2}{  \gamma \eta}(1- 0.5\gamma \eta  )^{T_k}.$} 

\far{\subsection {Proof of Lemma~\ref{num-proof-mon-F}}\label{sec:proof_lemma_num-proof-mon-F}
It suffices to show that for all $x,y \in \mathbb{R}^{29}_+$ we have
$y\fyy{^\top} \left(\frac{J(x) + J\fyy{^\top}(x)}{2}\right) y \geq 0,$ where $J$ denotes the Jacobian of mapping $F$ ($ J_{i,j} = \frac{\partial F_i}{\partial x_j} $). Recall that $\Omega  \in \mathbb{R}^{4 \times 25}$ and $\Delta \in \mathbb{R}^{19 \times 25}$. Let
$\delta_a \in \mathbb{R}^{25 \times 1}$ denote a column vector of entries of the $a$th row of $\Delta$ for $a= 1,\ldots, 19$. 
Also, let $\Delta_p  \in \mathbb{R}^{19 \times 1}$ denote the $p$th column of $\Delta$ and $\Delta_{a,p}$ denote the $(a,p)$ entry of $\Delta$, for $a= 1 ,\ldots, 19$ and $p=1, \ldots,25$. Let us define the function $g_a(z)$ as $g_a(z) \triangleq t_a^0 (1+ 0.15 (\frac{z}{cap_a})^{n_a})$ for $z \geq 0$. We have 
\begin{align*} 
C(h)&= \Delta\fyy{^\top} \begin{bmatrix} g_1(\delta_1\fyy{^\top}h);  \ldots;  g_{19}(\delta_{19}\fyy{^\top}h)\end{bmatrix}\\
& =\begin{bmatrix}  \Delta_1\fyy{^\top} \begin{bmatrix}    g_1(\delta_1\fyy{^\top}h);\ldots;  g_{19}(\delta_{19}\fyy{^\top}h)\end{bmatrix} ; \ldots; \Delta_{25}\fyy{^\top}  \begin{bmatrix}  g_1(\delta_1\fyy{^\top}h);\ldots;  g_{19}(\delta_{19}\fyy{^\top}h)\end{bmatrix}\end{bmatrix} \\
&= \begin{bmatrix} \sum_{a=1}^{19}\Delta_{a,1} g_a(\delta_a\fyy{^\top} h)  ;\ldots;  \sum_{a=1}^{19}\Delta_{a,25} g_a(\delta_a\fyy{^\top} h)\end{bmatrix}.
\end{align*}
By invoking the chain rule of calculus and $\nabla_z g_a(z) = 0.15\, t_a^0 \, n_a \, \frac{z^{n_a-1}}{cap_a^{n_a}}$, we obtain
\begin{align*}
\nabla_h \left(\sum_{a=1}^{19}\Delta_{a,p} g_a(\delta_a\fyy{^\top} h)\right)= 0.15\,\sum_{a=1}^{19}\Delta_{a,p}  \, t_a^0 \, n_a \, \frac{(\delta_a\fyy{^\top}h)^{n_a-1}}{cap_a^{n_a}}\delta_a, \qquad \hbox{for all } p = 1, \ldots,25.
\end{align*}
Therefore, we obtain
\[
J(x) = \left[
\begin{array}{c|c}
    \begin{matrix}
   \sum_{a=1}^{19} \frac{0.15\, \Delta_{a,1}  \, t_a^0 \, n_a \, (\delta_a\fyy{^\top} h)^{n_a-1}}{cap_a^{n_a}}\delta_a, \ldots ,  \sum_{a=1}^{19} \frac{0.15\,\Delta_{a,25}  \, t_a^0 \, n_a \, (\delta_a\fyy{^\top} h)^{n_a-1}}{cap_a^{n_a}}\delta_a \\
    \end{matrix} & \Omega\fyy{^\top} \\
    \hline
    -\Omega & 0 
\end{array}
\right].
\]
Let us define $o_{a,p} \in \mathbb{R}_+$ such that $o_{a,p} =0.15 \, \Delta_{a,p} \, t_a^0 \, n_a \, \frac{(\delta_a\fyy{^\top} h)^{n_a-1}}{cap_a^{n_a}} $ for $a=1, \ldots, 19$ and $p=1, \ldots, 25$. Also, let $O_p\in \mathbb{R}_+^{19 \times 1}$ such that $o_{a,p}$ denotes the $a$th entry of $O_p$. Therefore, we obtain 
$$\frac{J(x)+J\fyy{^\top}(x)}{2} = \frac{1}{2} \, \left[
\begin{array}{c|c}
    \begin{matrix}
   \Delta\fyy{^\top} \begin{bmatrix} O_1 \cdots O_{25}  \end{bmatrix}+ \begin{bmatrix} O_1 \cdots O_{25}  \end{bmatrix}\fyy{^\top} \Delta \\
    \end{matrix} & 0_{25 \times 4} \\
    \hline
    0_{4 \times 25} & 0_{4 \times 4}
\end{array}
\right]_{29 \times 29} .$$
Note that the entries of $\Delta$ and $O_p$ for all $p=1, \ldots, 25$ are nonnegative. Thus, for all  $x,y \in \mathbb{R}^{29}_+$, $y\fyy{^\top} (\frac{J(x) + J\fyy{^\top}(x)}{2}) y \geq 0,$ which implies monotonicity of $F$ for any arbitrary $n_a  \geq  0$ for $a=1, \ldots, 19.$
}
\far{\subsection {Proof of Lemma~\ref{num-proof-cvx-f}}\label{sec:proof_lemma_num-proof-cvx-f}
 Let
$\delta_a \in \mathbb{R}^{25 \times 1}$ denote a column vector of entries of the $a$th row of $\Delta$ for $a= 1,\ldots, 19$, and and $\delta_{a,p}$ denote the $p$th entry of vector $\delta_p$, for $p=1,\ldots,25.$ Let us define the function $g_a(z)$ as $g_a(z) \triangleq t_a^0 (1+ 0.15 (\frac{z}{cap_a})^{n_a})$ for $z \geq 0$. Note that $c_a(\Delta h) = g_{a}(\delta_{a}\fyy{^\top}h)$. We obtain 
\begin{align*}
f(x) &= \mathbf{1}_{25}\fyy{^\top} C(h) =  (\Delta \mathbf{1}_{25} )\fyy{^\top} \begin{bmatrix} g_1(\delta_1\fyy{^\top}h);\ldots;  g_{19}(\delta_{19}\fyy{^\top}h)\end{bmatrix}  \\
& = \sum_{a=1}^{19} (\Delta\mathbf{1}_{25} )_a g_a(\delta_a\fyy{^\top}h) = \sum_{a=1}^{19}  (\delta_a\fyy{^\top}\mathbf{1}_{25})    g_a(\delta_a\fyy{^\top}h).
\end{align*}
Note that $\nabla_z g_a(z) = 0.15\, t_a^0 \, n_a \, \frac{z^{n_a-1}}{cap_a^{n_a}}$. Let us define $\beta_a=   \frac{(\delta_a\fyy{^\top} \mathbf{1}_{25})    n_a  t_a^0  }{cap_a^{n_a}}$. We obtain
\begin{align*}&\nabla_x  f(x)  = \begin{bmatrix}
0.15\sum_{a=1}^{19}  (\delta_a\fyy{^\top} \mathbf{1}_{25})    n_a  t_a^0     \frac{(\delta_a\fyy{^\top}h)^{n_a-1}}{cap_a^{n_a}}\ \delta_a\\ 0_{4\times 1}\end{bmatrix}= \begin{bmatrix}
0.15\sum_{a=1}^{19}   \beta_a   (\delta_a\fyy{^\top}h)^{n_a-1}\ \delta_a\\ 0_{4\times 1}\end{bmatrix}.
\end{align*}
Note that $ \nabla^2_x f(x) =
\left[ \nabla^2_h f(x) , 0_{25 \times 4} ;
    0_{4 \times 25} , 0_{4 \times 4}
\right]$. To show that $f$ is convex, it suffices to show that $h\fyy{^\top}\nabla^2_h f(x)h \geq 0$ for all $h \in \mathbb{R}^{25}_+$. We have 
\begin{align*}&\nabla_h  f(x)  =\begin{bmatrix}
0.15\sum_{a=1}^{19}   \beta_{a}   (\delta_a\fyy{^\top}h)^{n_a-1}\ \delta_{a,1} ;\ldots; 0.15\sum_{a=1}^{19}   \beta_a   (\delta_a\fyy{^\top}h)^{n_a-1}\ \delta_{a,25} \end{bmatrix}.
\end{align*}
The gradient of each entry of $\nabla_h  f(x)$ with respect to $h$ is as follows, for   $p=1,\ldots,25$, 
$$\nabla_h  \left(0.15\sum_{a=1}^{19}   \beta_{a}  (\delta_a\fyy{^\top}h)^{n_a-1}\ \delta_{a,p}\right) = 0.15\sum_{a=1}^{19}   \beta_{a} (n_a -1) (\delta_a\fyy{^\top}h)^{n_a-2} \delta_{a,p}  \delta_a.$$
Therefore, we obtain
\begin{align*}
&\nabla^2_h  f(x)  =0.15\sum_{a=1}^{19}   \beta_{a} (n_a -1) (\delta_a\fyy{^\top}h)^{n_a-2}\begin{bmatrix} 
 \begin{bmatrix}
  \delta_{a,1}  \delta_a \end{bmatrix}_{25 \times 1}   \cdots    \begin{bmatrix}  \delta_{a,25}  \delta_a  \end{bmatrix}_{25 \times 1}  \end{bmatrix}\\
  &=0.15\sum_{a=1}^{19}   \beta_{a} (n_a -1) (\delta_a\fyy{^\top}h)^{n_a-2}(\delta_a \delta_a\fyy{^\top})\fyy{^\top}=0.15\sum_{a=1}^{19}   \beta_{a} (n_a -1) (\delta_a\fyy{^\top}h)^{n_a-2}(\delta_a \delta_a\fyy{^\top}) .
\end{align*}
For any given $h \in \mathbb{R}^{25\times 1}_+$, we obtain
\begin{align*}
&h\fyy{^\top} \nabla^2_h f(x) h =  0.15\sum_{a=1}^{19}   \beta_{a} (n_a -1) (\delta_a\fyy{^\top}h)^{n_a-2} \|\delta_a\fyy{^\top} h \|^2 \geq 0,
\end{align*} 
where we used $n_a \geq 1$ and $\delta_a \geq 0$, for all $a$. This completes the proof. 
}
\end{document}